\documentclass[10pt]{amsart}
\usepackage[centertags]{amsmath}
\usepackage{amsfonts}
\usepackage{amssymb}
\usepackage{amsthm}
\usepackage{newlfont}
\usepackage{amscd}
\usepackage{mathrsfs}
\usepackage[all,cmtip]{xy}
\usepackage{tikz}
\usepackage[pagebackref=false,colorlinks]{hyperref}
\hypersetup{pdffitwindow=true,linkcolor=blue,citecolor=blue,urlcolor=cyan}
\usepackage{cite}
\usepackage{fancyhdr}
\usepackage{stmaryrd}
\usepackage{yfonts}

\usepackage[OT2,OT1]{fontenc}
\newcommand\cyr{%
\renewcommand\rmdefault{wncyr}%
\renewcommand\sfdefault{wncyss}%
\renewcommand\encodingdefault{OT2}%
\normalfont
\selectfont}
\DeclareTextFontCommand{\textcyr}{\cyr}

\usepackage[toc, page] {appendix}

\newcommand{\Mod}[1]{\ (\text{mod}\ #1)}

\newtheorem{thm}{Theorem}[section]
\newtheorem{cor}[thm]{Corollary}
\newtheorem{lem}[thm]{Lemma}
\newtheorem{prop}[thm]{Proposition}

\newtheorem{conj}[thm]{Conjecture}

\theoremstyle{definition}
\newtheorem{defn}[thm]{Definition}
\newtheorem{assu}[thm]{Assumption}
\newtheorem{rem}[thm]{Remark}

\begin{document}
\title[Anticyclotomic Mazur-Tate conjecture]{An anticyclotomic Mazur-Tate conjecture for modular forms}
\author[C-.H. Kim]{Chan-Ho Kim}
\address{School of Mathematics, Korea institute for Advanced Study (KIAS), 85 Hoegiro, Dongdaemun-gu, Seoul 02455, Republic of Korea}
\email{chanho.math@gmail.com}
\date{\today}
\subjclass[2010]{Primary 11R23; Secondary 11G05, 11F33, 11F67}
\keywords{refined Iwasawa theory, Mazur-Tate conjectures, Heegner points, Euler systems, congruences, modular forms, Shimura curves}
\begin{abstract}
Under certain assumptions, we prove an anticyclotomic analogue of the ``weak main conjecture" \`{a} la Mazur and Tate for modular forms over a large class of cyclic ring class extensions.
\end{abstract}
%
 
\maketitle
\setcounter{tocdepth}{1}

\section{Introduction}
\subsection{Overview}
The purpose of this article is to prove a Mazur-Tate type conjecture for modular forms at a good ordinary prime over cyclic ring class extensions of an imaginary quadratic field. It is a refinement of the Euler system divisibility of the main conjecture of Iwasawa theory for modular forms over anticyclotomic $\mathbb{Z}_p$-extensions in \cite{bertolini-darmon-imc-2005}. Here, the \emph{refinement} means that we look at Fitting ideals over (more general) finite abelian extensions, not characteristic ideals over the infinite $\mathbb{Z}_p$-extension. Thus, we would have a chance to understand the \emph{structure} of Selmer groups rather than the \emph{size} of Selmer groups (c.f.\cite[$\S$10]{kurihara-fitting} and \cite{kurihara-iwasawa-2012}) in near future.

It is known that the Euler system divisibility of the Iwasawa main conjecture implies the corresponding Mazur-Tate conjecture for $p$-stabilized newforms for cyclic extensions of $p$-power degree if one has the exact control theorem and the non-existence of a non-trivial finite submodule over the Iwasawa algebra (\cite[$\S$2]{kim-kurihara}).

We extend the strategy to \emph{general} cyclic ring class extensions. In order to do this, we first prove the Euler system divisibility of the \emph{generalized} anticyclotomic Iwasawa main conjecture allowing prime-to-$p$ cyclic ring class extensions. Then we apply the exact control theorem and the non-existence of a non-trivial finite submodule over the generalized Iwasawa algebra to this generalized Euler system divisibility. Here, the \emph{generalized} means that the corresponding Galois group is isomorphic to $\mathbb{Z}_p \times \mathbb{Z}/m'\mathbb{Z}$ with $(p,m')=1$.

Also, we compare the Bertolini-Darmon elements arising from newforms and their $p$-stabilizations to obtain the Mazur-Tate type conjecture result for non-$p$-stabilized forms. In the process, we also observe a different type of the exceptional zero phenomenon arising from prime-to-$p$ extensions ($\S$\ref{sec:tame_exceptional_zeros}).

The proof of this generalized Euler system divisibility \emph{reveals} the significance of the Euler system argument \`{a} la Bertolini-Darmon \cite[$\S$4]{bertolini-darmon-imc-2005} \emph{again}. Throughout more than a decade, the formalism of the Euler system argument is extensively used in various settings. The list includes \cite[$\S$4]{bertolini-darmon-imc-2005}, \cite[$\S$5]{darmon-iovita}, \cite[$\S$4.4]{pw-mu}, \cite[$\S$7.6]{longo-hilbert-modular-case},  \cite[$\S$7]{vanorder-dihedral}, \cite[$\S$2.3]{gomez-thesis}, \cite[$\S$6]{chida-hsieh-main-conj}, and \cite[$\S$7.3]{haining-thesis}. The Euler system argument consists of two parts. The main part is the ``character by character" divisibilities over discrete valuation rings. Here, the discrete valuations rings come from the evaluation of characters on the Iwasawa algebra. The other part is the lifting argument from the divisibilities over discrete valuations rings to the desired divisibility over the Iwasawa algebra.
As a by-product of the Euler system argument, one obtains the Euler system divisibility of the main conjecture, which gives an upper bound of the corresponding Selmer groups. 

The key observation for the generalization is that the main body of argument for the ``character by character" divisibilities over discrete valuation rings is (almost) independent of the shape of the (infinite) extension (modulo some detail, of course) ($\S$\ref{sec:euler_systems}).
This is because the main part of the argument takes place over discrete valuation rings after specialization via a character.
This idea can also be observed in \cite{longo-vigni-vanishing}, \cite{nekovar-hilbert} and \cite{gomez-thesis}.

We also observe that the serious use of the theory of Iwasawa modules is only made in the lifting argument (\cite[Proposition 3.1]{bertolini-darmon-imc-2005}) from the ``character by character" divisibilities to the bona fide divisibility.
We also generalize the lifting argument to allow prime-to-$p$ extensions (Proposition \ref{prop:generalized_divisibility}).

However, it turns out that the backbone and detail of the Euler system argument depends heavily on the properties of $p$-extensions. Several details of the Euler system argument are not valid anymore in the general cyclic case. For example, all the arguments using the properties of local rings (e.g.~Nakayama's lemma with the Iwasawa algebra) simply do not work since the the generalized Iwasawa algebra is only semi-local. In order to overcome this issue, we improve the Chebotarev density type result (Theorem \ref{thm:chebotarev}) by allowing more general base fields. Using this upgraded tool, we make the reduction process to the base field in the Euler system argument more ``relative".

In the reduction process, another key ingredient is the control theorem of compact Selmer groups with respect to quotients (Corollary \ref{cor:control}), which is based on the freeness theorem of compact Selmer groups (Theorem \ref{thm:freeness}). The original proof of the freeness theorem also depends on the properties of $p$-extensions. In order to generalize the freeness theorem, we generalize the concept of $n$-admissible sets (Definition \ref{defn:n-admissible-set}) allowing more general base fields as before. Also, we use the isotypic decomposition with respect to quotients ($\S$\ref{subsec:enhanced_decomposition}) to recycle the original approach.

Using these ideas, we carefully examine the generalized Euler system argument over discrete valuations rings and apply the generalized lifting argument. To sum up, we do not only use the Euler system argument but also refine the argument itself.
We also remark that there is a close similarity between the enhanced isotypic decomposition and Bertolini-Darmon's Euler system argument because both use the specialization via characters.

As a by-product of the generalization, we deduce the Euler system divisibility of the generalized main conjecture and a Mazur-Tate type conjecture for modular forms under cyclic ring class extensions. As a consequence, we obtain more refined and general arithmetic applications than those of the main conjecture. It seems that all the proven Euler system divisibilities mentioned in the above list can be upgraded as ours without introducing new difficulty, at least to some extent. Also, we can obtain the upper bound of $M$-Selmer groups of elliptic curves over cyclic ring class extensions for some $M$ which is not necessarily a prime power under a certain normalization assumption ($\S$\ref{sec:speculation}).

We also give a detailed exposition of the mod $p^n$ level raising at two $n$-admissible primes ($\S$\ref{subsec:level_raising_at_two_prime}) and reveal the importance of ``Condition Ihara" (Assumption \ref{assu:ihara}), which is not emphasized in previous papers, in the non-ordinary setting ($\S$\ref{subsec:pm-main-conj}). 

This work can be regarded as a part of the refined Iwasawa theory \`{a} la Kurihara \cite[$\S$10, $\S$11]{kurihara-fitting}, \cite[Remark 1, $\S$1]{kurihara-iwasawa-2012} for the anticyclotomic setting. Note that our approach is completely different from the current approaches to the cyclotomic Mazur-Tate conjecture (for examples, \cite{ota-thesis}, \cite{kim-kurihara}, \cite{epw2}) mainly due to the significant difference between Euler systems arising from Heegner points and Kato's Euler systems arising from Siegel units.


\begin{rem}[on the generality and the convention]
As we already mentioned, there are various settings where the Euler system argument is used. We adapt and refine the setting of \cite{pw-mu}, namely ordinary modular forms of weight 2 with arbitrary Fourier coefficients. However, we adapt the language of \cite{chida-hsieh-main-conj} mostly.
\end{rem}

%

\subsection{Setting the stage}
\subsubsection{Modular setup} \label{subsubsec:modular_setup}
Let $p > 3$ be a prime and $N$ be an integer $\geq 1$ with $(N,p)=1$. 
Fix embeddings $\iota_p : \overline{\mathbb{Q}} \hookrightarrow \overline{\mathbb{Q}}_p$ and $\iota_\infty : \overline{\mathbb{Q}} \hookrightarrow \mathbb{C}$.
Let $f = f(\tau) \in S_2 (\Gamma_0(N))^{\mathrm{new}}$ be a $p$-ordinary newform. Let $E$ be the Hecke field of $f$ over $\mathbb{Q}_p$, which is compatible with $\iota_p$, and $\mathcal{O}$ the ring of integers of $E$. Let $\varpi$ be a uniformizer of $\mathcal{O}$, $\mathbb{F} := \mathcal{O} / \varpi \mathcal{O}$ the residue field, and $\mathcal{O}_n := \mathcal{O} / \varpi^n$.
For a field $F$, denote the absolute Galois group of $F$ by $G_F$. 
Let 
$$\rho : G_\mathbb{Q} \to \mathrm{GL}_2(E) = \mathrm{GL}(V_f)$$
 be the Galois representation associated to $f$.
Let $T_f \subset V_f$ be a Galois-stable $\mathcal{O}$-lattice of $f$.
We assume that the residual representation 
$$\overline{\rho} : G_\mathbb{Q} \to \mathrm{GL}_2(\mathbb{F}) = \mathrm{GL}(T_f/\varpi T_f)$$
 of $f$ is absolutely irreducible. Then $\overline{\rho}$ is independent of the choice of $T_f$.
Let $T_{f,n} := T_f / \varpi^n T_f$, $A_f := V_f / T_f$, and $A_{f,n} := A_f[\varpi^n]$. 
It is known that $A_{f,n} \simeq T_{f,n}$ and $A_{f,n} \simeq \mathrm{Hom}_{\mathcal{O}} (T_{f,n}, E/\mathcal{O}(1))$ as $G_\mathbb{Q}$-modules via the self-duality.
By the ordinary assumption, we define the $p$-stabilized newform 
$$f_\alpha(\tau) := f(\tau) - \frac{1}{\alpha} f(p\tau) $$
of $f$ with the unit $U_p$-eigenvalue $\alpha = \alpha_p(f)$.
\subsubsection{Anticyclotomic setup} \label{subsubsec:anticyclotomic_setup}
Fix an imaginary quadratic field $K$ with $(\mathrm{disc}(K), Np)=1$.
Let $N(\overline{\rho})$ be the prime-to-$p$ conductor of $\overline{\rho}$.
Then we have decompositions
\[
\xymatrix{
N  = N^+N^- , & N(\overline{\rho})  = N(\overline{\rho})^+N(\overline{\rho})^-
}
\]
where the $(+)$-part is the product of the primes which split in $K$ and the $(-)$-part is the product of the primes which are inert in $K$. Note that $N(\overline{\rho})^\pm$ divides $N^\pm$, respectively.
\begin{assu}[Parity]
In the decomposition, $N^-$ is the square-free product of an \emph{odd} number of primes.
\end{assu}
\begin{rem}
The parity assumption implies that we deal with the opposite case to \cite[Conjecture 2.1 and Theorem 2.4]{darmon-refined-bsd}.
See also \cite[Theorem 1.1 and Conjecture 5.1]{longo-vigni-refined} for the higher weight case.
\end{rem}
Let $m = m' p^r$ be a non-negative integer with $(m',p)=1$ such that $(m,N)=1$.
For $m  = \prod \ell^{r_i}_i$, write $m_0 = \prod \ell^{r_i + 1}_i$ where $\ell$ runs over primes.
Let $H(m_0)$ be the ring class field of $K$ of conductor $m_0$, $G_{m_0} := \mathrm{Gal}(H(m_0)/K)$.
For any prime $\ell$, let $K(\ell^r) \subset H(\ell^{r+1})$ be the maximal $\ell$-subextension of $H(\ell^{r+1})$ over $K$.
Let $K(m)$ be the compositum of the $K(\ell^{r_i}_i)$'s. Then $K(m)$ is a cyclic extension of $K$ and is contained in $H(m_0)$.
\begin{assu} \label{assu:exact_degree}
Throughout this article, we assume that $H(1) \cap K(m) = K$. In other words, every prime $\ell$ dividing $m$ is totally ramified in $K(\ell^{r})/K$.
\end{assu}
\begin{rem}
If the class number of $K$ is prime to $m$, then Assumption \ref{assu:exact_degree} is always satisfied.
If $\ell$ is totally ramified in $K(\ell^r)/K$, then the extension degree of $K(\ell^r)$ over $K$ is $\ell^r$.
The index $[H(\ell^{r+1}):K(\ell^r)]$ is $\ell +1$/$\ell -1$ if  $\ell$ is inert/splits in $K/\mathbb{Q}$, respectively.
\end{rem}
Let $\Gamma_m :=\mathrm{Gal}(K(m)/K)$ and then $\Gamma_m$ is a cyclic group of order $m$. We fix a generator $\gamma_m \in \Gamma_m$ and isomorphism
$\mathcal{O}[ \Gamma_m ] \simeq \mathcal{O}[ T ]/( (1+T)^{m} -1 )$ defined by $\gamma_m \mapsto 1+T$.

Let $K(p^\infty) := \bigcup_{r=1}^{\infty} K(p^r)$ be the anticyclotomic $\mathbb{Z}_p$-extension and $\Gamma_{p^\infty} :=\mathrm{Gal}(K(p^\infty)/K)$. Let $\Lambda_{p^\infty} = \mathcal{O} \llbracket \Gamma_{p^\infty} \rrbracket$ be the Iwasawa algebra.

Similarly, let $K(m'p^\infty) := \bigcup_{r=1}^{\infty} K(m'p^r)$ and $\Lambda_{m'p^\infty} := \mathcal{O} \llbracket \Gamma_{m'p^\infty} \rrbracket \simeq \mathcal{O} [\Gamma_{m'}]\llbracket \Gamma_{p^\infty} \rrbracket$ be the generalized Iwasawa algebra.
\subsection{Anticyclotomic Mazur-Tate conjectures for modular forms}
We state an anticyclotomic analogue of the Mazur-Tate conjecture \cite{mazur-tate}, which we prove under certain conditions.
Let $L_p(K(m), f) \in \mathcal{O}[ \mathrm{Gal}(K(m)/K) ]$ be the Bertolini-Darmon element attached to $f$ and $K(m)$ defined in $\S$\ref{sec:theta_elements}, which is an analogue of Mazur-Tate elements.
Observing the statement of the Iwasawa main conjecture, we can easily state the following conjecture.
\begin{conj}[``Weak main conjecture" \`{a} la Mazur-Tate: the minimal Selmer case] \label{conj:mazur-tate-minimal}
$$L_p(K(m), f) \in \mathrm{Fitt}_{\mathcal{O}[ \mathrm{Gal}(K(m)/K)  ]} \left( \mathrm{Sel}(K(m), A_f)^\vee \right)$$
where $\mathrm{Sel}(K(m), A_f)$ is the minimal Selmer group as in $\S$\ref{subsec:minimal_selmer}.
\end{conj}
It seems that there is no direct way to attack this minimal Selmer case. The practical Euler system argument in all the literatures works only well with the $N^-$-ordinary Selmer group $\mathrm{Sel}_{N^-}(K(m), A_{f,n})$ (Definition \ref{defn:selmer}). Thus, we consider the following formulation.
\begin{conj}[``Weak main conjecture" \`{a} la Mazur-Tate: the $N^-$-ordinary Selmer case] \label{conj:mazur-tate-ordinary}
$$L_p(K(m), f) \in \mathrm{Fitt}_{\mathcal{O}[ \mathrm{Gal}(K(m)/K)  ]} \left( \mathrm{Sel}_{N^-}(K(m), A_{f})^\vee \right)$$
where $\mathrm{Sel}_{N^-}(K(m), A_{f}) := \varinjlim_n \mathrm{Sel}_{N^-}(K(m), A_{f,n})$ and the transition map with respect to $n$ is induced from the embedding $A_{f,n} \hookrightarrow A_{f,n+1}$.
\end{conj}
\begin{rem}
In the very original formulation of the Mazur-Tate conjecture, the coefficient ring is a subring of a global field.
In our formulation, we assume that the coefficient ring $\mathcal{O}$ is $p$-adic. See $\S$\ref{sec:speculation} for this global aspect.
\end{rem}
\subsection{Running hypotheses}
We make the list of assumptions we need for our main theorem and point out where the conditions are used throughout this article.
\begin{assu}[$N^{\pm}$-minimality]  \label{assu:minimality}
The form $f \in S_2(\Gamma_0(N))^{\mathrm{new}}$ is \textbf{$N^{\pm}$-minimal} if $N^{\pm} = N(\overline{\rho})^{\pm}$, respectively. 
\end{assu}
Condition CR is a relaxation of $N^-$-minimality but still ensures a mod $p$ multiplicity one result for the localized Hecke modules we deal with.
\begin{assu}[Condition CR: {\cite[Definition 6.4]{helm-israel}}, {\cite[Definition 1.6]{kim-summary}}]  \label{assu:conditionCR}
Let  $\overline{r}$ be a residual modular Galois representation with residual characteristic $\geq 5$.
Let $M$ be a square-free product of an odd number of primes and is divisible by $N(\overline{r})^-$.
The pair $(\overline{r}, M)$ satisfies \textbf{condition CR} if
\begin{itemize}
\item $\overline{r}$ is irreducible, and
\item if $q \mid M$ and $q \equiv \pm 1 \pmod{p}$, then $\overline{\rho}$ is ramified at $q$, i.e.~$q \mid N(\overline{r})^-$.
\end{itemize}
\end{assu}
We impose the following certain local condition at $p$, which is slightly weaker than Hypothesis (PO) in \cite{chida-hsieh-main-conj}.
Let $\alpha_p(f)$ be the unit root of the Hecke polynomial $X^2 - a_p(f) X + p$ of $f$ at $p$.
If $p$ splits in $K$, we decompose $p = \mathfrak{p} \overline{\mathfrak{p}}$ in $K$.
Let  $\mathrm{Fr}_\mathfrak{p}$ and $\mathrm{Fr}_{\overline{\mathfrak{p}}}$ be the Frobenii in $\Gamma_{m'}$ under the global reciprocity map with geometric normalization.
\begin{assu}[Condition PO] \label{assu:conditionPO}
The triple $(f, K(m')/K, p)$ satisfies \textbf{condition PO} if 
\begin{enumerate}
\item $p$ is inert in $K$ and $$a_p(f) \not\equiv \pm 1 \pmod{\varpi},$$ or
\item $p$ splits in $K$ and $$\omega(\mathrm{Fr}_\mathfrak{p}) \not\equiv \alpha_p(f) \pmod{\varpi_\omega} \textrm{ and }\omega(\mathrm{Fr}_{\overline{\mathfrak{p}}}) \not\equiv \alpha_p(f) \pmod{\varpi_\omega}$$ 
for any character $\omega : \Gamma_{m'} \to \overline{\mathbb{Q}}^\times_p$
where $\varpi_\omega$ is a uniformizer of the extension of $\mathcal{O}$ adjoining the values of $\omega$.
\end{enumerate}
\end{assu}

\begin{defn}[Big image] \label{defn:bigimage}
The residual representation $\overline{\rho}$ of $f$ \textbf{has big image} if the image of $\overline{\rho}$ contains a conjugate of $\mathrm{GL}_2(\mathbb{F}_p)$ or $p>5$.
\end{defn}

\begin{rem}
\begin{enumerate}
\item
The $N^+$-minimality assumption is essentially used in control theorems (Lemma \ref{lem:control_galois_selmer} and Lemma \ref{lem:control_mod_pn_selmer}), and these results play important roles in the proof of Theorem \ref{thm:freeness}.
Without the $N^+$-minimality assumption, Theorem \ref{thm:freeness} may fail. See \cite{kim-pollack-weston} for this issue.
Also, note that the $N^+$-minimality assumption corresponds to \cite[Assumption 2.15]{bertolini-darmon-derived-heights-1994} for primes dividing $N^+$. 
\item
Condition CR is broadly used in $\S$\ref{sec:level_raising} to obtain the uniqueness of mod $p^n$ modular forms under level raising.
\item
Condition PO corresponds to \cite[Assumption 2.15]{bertolini-darmon-derived-heights-1994} for primes dividing $p$ via \cite[Proposition 2.16]{bertolini-darmon-derived-heights-1994} and is also used in the exact control theorem. See Lemma \ref{lem:local_condition_at_p}, Lemma \ref{lem:control_galois_selmer}, and $\S$\ref{sec:tame_exceptional_zeros}.
Condition PO is slightly weaker than Hypothesis (PO) in \cite{chida-hsieh-main-conj} if $m' = 1$.
\item
The big image assumption is required for Theorem \ref{thm:chebotarev} and, indeed, is satisfied for most cases by \cite[Main Theorem (0.1).2]{ribet-img-1} and
The absolute irreducibility of $\overline{\rho}$, $p> 5$, and Condition CR imply that the image is large enough to establish Theorem \ref{thm:chebotarev} without presumably knowing $\mathrm{GL}_2(\mathbb{F}_p) \subset \overline{\rho} \left( G_\mathbb{Q} \right)$ up to conjugation. See \cite[Lemma 6.1 and Lemma 6.2]{chida-hsieh-main-conj} (with help of \cite{ribet-pacific}) for detail. 
\end{enumerate}
\end{rem}

\subsection{Main theorems and arithmetic consequences} \label{subsec:main_theorem_consequences}
We state the main theorems and their arithmetic consequences.
\begin{thm}[Main theorem I: generalized Iwasawa main conjecture] \label{thm:main_theorem_1}
Assume the following conditions:
\begin{enumerate}
\item $(\overline{\rho}, N^-)$ satisfies Condition CR;
\item $(f, K(m')/K, p)$ satisfies Condition PO;
\item $\overline{\rho}$ has big image;
\item $f$ is $N^+$-minimal;
\item $m'$ is prime to $Np$.
\end{enumerate}
Then 
$$L_p(K(m'p^\infty), f_\alpha) \in \mathrm{Fitt}_{\Lambda_{m'p^\infty}} \left( \mathrm{Sel}(K(m'p^\infty), A_{f})^\vee \right)$$
holds.
\end{thm}

\begin{thm}[Main theorem II: Mazur-Tate conjecture] \label{thm:main_theorem}
Assume the following conditions:
\begin{enumerate}
\item $(\overline{\rho}, N^-)$ satisfies Condition CR;
\item $(f, K(m')/K, p)$ satisfies Condition PO;
\item $\overline{\rho}$ has big image;
\item $f$ is $N^+$-minimal;
\item $m$ is prime to $N$.
\end{enumerate}
Then Conjecture \ref{conj:mazur-tate-ordinary}
$$L_p(K(m), f) \in \mathrm{Fitt}_{\mathcal{O}[ \mathrm{Gal}(K(m)/K)  ]} \left( \mathrm{Sel}_{N^-}(K(m), A_{f})^\vee \right)$$
holds.
If we further assume that $f$ is also $N^-$-minimal (so $N$-minimal), then Conjecture \ref{conj:mazur-tate-minimal}
$$L_p(K(m), f) \in \mathrm{Fitt}_{\mathcal{O}[ \mathrm{Gal}(K(m)/K)  ]} \left( \mathrm{Sel}(K(m), A_{f})^\vee \right)$$
holds.
\end{thm}
\begin{rem}
See \cite[Corollary 2, 3, 4, and 5, Introduction]{bertolini-darmon-imc-2005} for the arithmetic applications of the Euler system divisibility. See $\S$\ref{sec:mainconj} for the current status of the main conjecture.
\end{rem}

Before stating refined arithmetic applications of the main theorems, we briefly review the relevant background on the group ring $\mathcal{O}[\Gamma_m]$. See \cite[(1.5)]{mazur-tate} for detail.

Let $\chi : \mathrm{Gal}(K(m)/K) \to \mathcal{O}^\times_{\chi}$ be a ring class character of finite order where $\mathcal{O}_\chi$ is the $\mathcal{O}$-algebra generated by the values of $\chi$. We identify $\chi$ with the $\mathcal{O}$-algebra homomorphism
$\chi : \mathcal{O}[\Gamma_m] \to \mathcal{O}_\chi$ defined by $\chi: \gamma_m \mapsto \chi(\gamma_m)$.
Let $I_\chi := \mathrm{ker}(\chi) \subseteq \mathcal{O}[\Gamma_m]$ be the augmentation ideal at $\chi$.
\begin{defn}[Vanishing order]
Let $\theta  \in \mathcal{O}[\Gamma_m]$. We say $\theta$ \textbf{vanishes to infinite order at $\chi$} if $\theta$ is contained in any power of $I_\chi$. We say $\theta$ \textbf{vanishes to order $r$ at $\chi$} if $\theta \in I^r_\chi \setminus I^{r+1}_\chi$.
\end{defn}

From now on, assume the Hecke field of $f$ is $\mathbb{Q}$.
Let $E_f$ be an elliptic curve over $\mathbb{Q}$ corresponding to $f$ via \cite[Theorem A]{bcdt}. We also assume the $N$-minimality of $f$ to have the coincidence of the minimal Selmer group of $f$, the $N^-$-ordinary Selmer group of $f$, and the classical Selmer group of $E_f$. 

The following corollary of Theorem \ref{thm:main_theorem} follows.
\begin{cor}[``Weak vanishing conjecture" \`{a} la Mazur-Tate; {\cite[Conjecture 1]{mazur-tate}}]
Under the same conditions of Theorem \ref{thm:main_theorem} and the $N^-$-minimality of $f$,
the rank $$\mathrm{rank}_{\mathbb{Z}_{p,\chi}} \left(  E_f(K(m))_\chi \right)$$ is equal to or less than the vanishing order of $L_p(K(m), f)$ at $\chi$ where 
$$E_f(K(m))_\chi := E_f(K(m)) \otimes_{\mathbb{Z}[\Gamma_m],\chi} \mathbb{Z}_{p,\chi}$$
 is the $\chi$-isotypic quotient of Mordell-Weil group $E_f(K(m))$.
 \end{cor}
\begin{rem}
See \cite[Proposition 3]{mazur-tate} for proof. It is a refinement of \cite[Corollary 2, Introduction]{bertolini-darmon-imc-2005}.
\end{rem}

In the case of vanishing order zero, we obtain a more refined formula from Theorem \ref{thm:main_theorem}.
\begin{cor}[Vanishing order zero case] \label{cor:vanishing_order_zero}
Assume the same conditions of Theorem \ref{thm:main_theorem} and the $N^-$-minimality of $f$.
If $L_p(K(m), f)$ has vanishing order zero at $\chi$, then
$$\mathrm{ord}_p\left( \# \vert E_f(K(m))_\chi \vert \right) + \mathrm{ord}_p \left( \# \vert \textrm{{\cyr SH}}(E_f/K(m))[p^\infty]_\chi \vert \right) \leq \mathrm{ord}_p \left( \chi \left( L_p(K(m), f) \right) \right) $$
where $\textrm{{\cyr SH}}(E_f/K(m))[p^\infty]_\chi$ is the $\chi$-isotypic quotient of the $p$-part of the Shafarevich-Tate group of $E_f$ over $K(m)$.
\end{cor}
\begin{rem}
It is the weak vanishing conjecture at $\chi$ with vanishing order zero and also a refinement of \cite[Corollary 4, Introduction]{bertolini-darmon-imc-2005}. This corollary gives a partial answer to the question raised in \cite[Remark 1 to Corollary 4, Introduction]{bertolini-darmon-imc-2005}. See $\S$\ref{sec:speculation} for another aspect of the question.
Note that the proof of this corollary itself does not appeal to any interpolation formula \`{a} la Gross-Walspurger-Zhang at all. However, the nonvanishing of $L_p(K(m), f)$ at almost all $\chi$ requires \cite[Theorem C and $\S$6.1]{hung-nonvanishing} generalizing \cite[Theorem 1.4]{vatsal-nonvanishing}.
\end{rem}

We generalize our main theorem to $M$-Selmer groups where $M$ is not necessarily a power of prime.
See $\S$\ref{sec:speculation} for detail.
Let $M = \prod_i p^{r_i}_i$ be an integer where $p_i >3$ is a prime such that
\begin{enumerate}
\item each $p_i$ is good ordinary for $E$;
\item each residual representation $\overline{\rho}_{E, p_i}$ is surjective;
\item condition PO holds for $(E, K(m'_i)/K, p_i)$ for each $p_i$ where $m'_i$ is the prime-to-$p_i$ part of $m$;
\item $N = N(\overline{\rho}_{E, p_i}) $ for each $p_i$.
\end{enumerate} 
Let $f$ be the newform corresponding to $E$ and assume that the Jacquet-Langlands correspondence ensures the integral normalization of the corresponding quaternionic form for all $p_i$. See Assumption \ref{assu:normalization} for detail.
Let $L_M(K(m), E)$ be the partially adelic anticyclotomic $L$-function defined in $\S$\ref{sec:speculation} under Assumption \ref{assu:normalization}.
\begin{thm}[Theorem \ref{thm:main_theorem} for $M$-Selmer groups]
Under Assumption \ref{assu:normalization}, we have
$$L_M(K(m), E) \in \mathrm{Fitt}_{\mathbb{Z}/M\mathbb{Z}[\Gamma_m]} \left(\mathrm{Sel}(K(m),E[M])^\vee \right)$$
in $(\mathbb{Z}/M\mathbb{Z})[\Gamma_m]$ where $(-)^\vee = \mathrm{Hom}(-, \mathbb{Q}/\mathbb{Z})$ here.
\end{thm}
In the same manner as in Corollary \ref{cor:vanishing_order_zero}, we obtain
\begin{cor}[Corollary \ref{cor:global_consequence}; Corollary \ref{cor:vanishing_order_zero} for $M$-Selmer groups]
Suppose that Assumption \ref{assu:normalization} holds.
Let $\chi : \Gamma_m \to \overline{\mathbb{Q}}^{\times}$ be a character and assume that $\chi(L_M(K(m), f)) \neq 0$.
Then
$$\sum_{p \vert M}  \left( \mathrm{ord}_{p} \left( \# \vert E(K(m))_\chi \vert \right) + \mathrm{ord}_{p} \left( \# \vert \textrm{{\cyr SH}}(E/K(m))[M]_\chi \vert \right) \right) \leq \sum_{p \vert M} \mathrm{ord}_{p} \left( \chi \left( L_M(K(m), f) \right) \right) $$
where $\textrm{{\cyr SH}}(E/K(m))[M]_\chi$ is the $\chi$-isotypic quotient of the $M$-torsion subgroup of the Shafarevich-Tate group of $E$ over $K(m)$.
\end{cor}

\subsection{The strategy of proof and the structure of this article}
\subsubsection{The strategy}
We follow the strategy of \cite{bertolini-darmon-imc-2005} and the notation of \cite{chida-hsieh-main-conj} very closely, but we will explicitly point out the differences when we encounter them.

The following diagram shows the strategy of proof and the detail is given in Proposition \ref{prop:mainconj_to_mazurtate}.
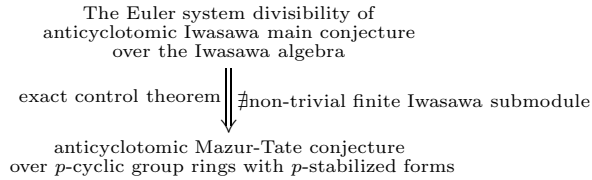
\[
\xymatrix{
{ \substack{ \textrm{The Euler system divisibility of} \\ \textrm{anticyclotomic Iwasawa main conjecture} \\ \textrm{over the Iwasawa algebra} } } 
 \ar@{=>}[d]^-{\nexists\textrm{non-trivial finite Iwasawa submodule}}_{\textrm{exact control theorem}} \\
 { \substack{ 
 \textrm{anticyclotomic Mazur-Tate conjecture} \\ \textrm{ over $p$-cyclic group rings with $p$-stabilized forms} 
  } }
}
\]
Modelling the above strategy, we extend the result to general cyclic ring class extensions as follows:
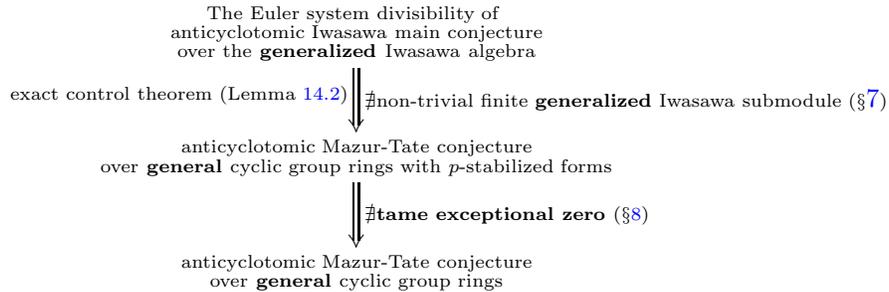
\[
\xymatrix{
{\substack{
\textrm{The Euler system divisibility of }\\ \textrm{ anticyclotomic Iwasawa main conjecture } \\ \textrm{over the} \textbf{ generalized } \textrm{Iwasawa algebra}
}}
 \ar@{=>}[d]^-{\nexists\textrm{non-trivial finite \textbf{generalized} Iwasawa submodule } (\S\ref{sec:reduction_to_main_theorem}) }_{ \textrm{exact control theorem (Lemma \ref{lem:control_galois_selmer})} } \\
{\substack{
\textrm{anticyclotomic Mazur-Tate conjecture} \\ \textrm{over}  \textbf{ general } \textrm{cyclic group rings with $p$-stabilized forms}
}}
 \ar@{=>}[d]^-{\nexists\textrm{\textbf{tame exceptional zero} ($\S$\ref{sec:tame_exceptional_zeros})}} \\
{\substack{
\textrm{anticyclotomic Mazur-Tate conjecture} \\ \textrm{over} \textbf{ general } \textrm{cyclic group rings} 
}}
}
\]
The reduction of the Euler system divisibility to the Mazur-Tate conjecture for the $p$-stabilized forms is explained in $\S$\ref{sec:reduction_to_main_theorem} and its reduction to the conjecture for the non-$p$-stabilized forms is given in $\S$\ref{sec:tame_exceptional_zeros}.
Here, the \emph{generalized} Iwasawa algebra means that we allow finite prime-to-$p$ extensions in the infinite Galois extension. In other words, it is isomorphic to $\mathbb{Z}_p [\Gamma_{m'}]\llbracket T \rrbracket$ where $\Gamma_{m'}$ is the cyclic group of order $m'$ with $(m',p) = 1$.
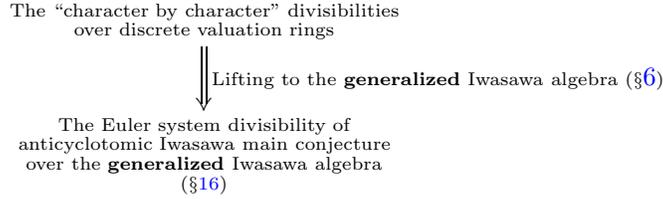
\[
\xymatrix{
{
\substack{
\textrm{The ``character by character" divisibilities}\\ \textrm{over discrete valuation rings}
}
}
 \ar@{=>}[d]^-{\textrm{Lifting to the \textbf{generalized} Iwasawa algebra } (\S\ref{sec:generalized_divisibility_criterion})   } \\
{\substack{
\textrm{The Euler system divisibility of} \\ \textrm{anticyclotomic Iwasawa main conjecture} \\ \textrm{over the \textbf{generalized} Iwasawa algebra} \\ \textrm{($\S$\ref{sec:completion_proof})}
}}
}
\]
The generalized Euler system divisibility comes from the character by character divisibilities and the lifting argument, which is given in $\S$\ref{sec:generalized_divisibility_criterion}.

In order to obtain the character by character divisibilities, the following diagram shows the strategy.
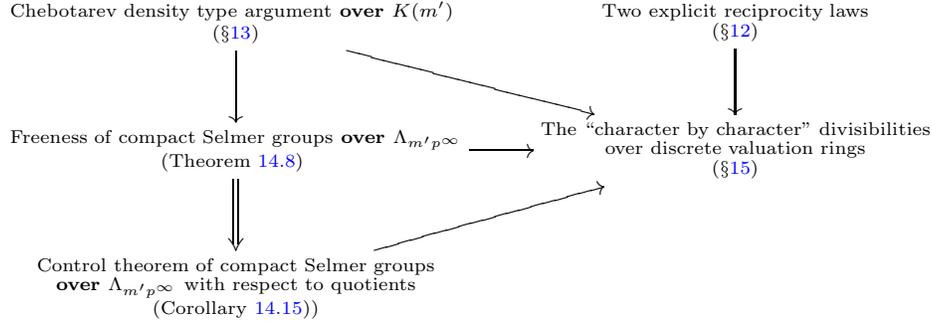
\[
\xymatrix{
{\substack{
\textrm{Chebotarev density type argument} \textbf{ over $K(m')$ } \\ \textrm{ ($\S$\ref{sec:chebotarev}) } 
}}
\ar[d] \ar[dr] & 
{\substack{
\textrm{Two explicit reciprocity laws} \\ \textrm{ ($\S$\ref{sec:explicit_reciprocity}) } 
}}
\ar[d] \\
{\substack{
\textrm{Freeness of compact Selmer groups} \textbf{ over $\Lambda_{m'p^\infty}$} \\ \textrm{(Theorem \ref{thm:freeness}) } 
}}
\ar@{=>}[d] \ar[r] &
{\substack{
\textrm{The ``character by character" divisibilities} \\ \textrm{over discrete valuation rings} \\  \textrm{ (\S\ref{sec:euler_systems}) } 
}}
 \\
{\substack{
\textrm{Control theorem of compact Selmer groups} \\ \textbf{over $\Lambda_{m'p^\infty}$ } \textrm{with respect to quotients} \\  \textrm{ (Corollary \ref{cor:control})) } 
}}
\ar[ur]
}
\]
In the above diagram, the Chebotarev density type argument over $K(m')$ and 
the freeness of compact Selmer groups over $\Lambda_{m'p^\infty}$ require some work for the generalization. The $m'=1$ case is only covered before. The following diagram points out how the generalizations are made.
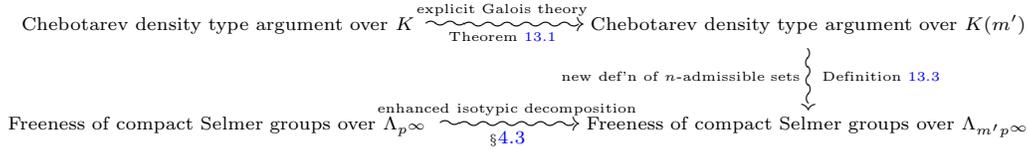
\[
{ \scriptsize
\xymatrix{
\textrm{Chebotarev density type argument over $K$}  \ar@{~>}[rr]^-{\textrm{explicit Galois theory}}_-{\textrm{Theorem \ref{thm:chebotarev}}} & & \textrm{Chebotarev density type argument over $K(m')$} \ar@{~>}[d]_-{\textrm{new def'n of $n$-admissible sets}  }^-{\textrm{ Definition \ref{defn:n-admissible-set} }} \\
\textrm{Freeness of compact Selmer groups over $\Lambda_{p^\infty}$} \ar@{~>}[rr]^-{\textrm{enhanced isotypic decomposition}}_-{\S\ref{subsec:enhanced_decomposition}} & & \textrm{Freeness of compact Selmer groups over $\Lambda_{m'p^\infty}$}
}
}
\]
\subsubsection{The structure}
We summarize the content of each section.

In $\S$\ref{sec:mainconj}, we briefly review the current state of art of the main conjecture.

In $\S$\ref{sec:theta_elements}, we quickly review modular forms of weight two in the quaternionic setting and the construction of the Bertolini-Darmon elements, which is the anticyclotomic analogue of Mazur-Tate elements.

In $\S$\ref{sec:fitting}, we quickly review the Fitting ideals, fix the convention on various characters, and 
also review the enhanced isotypic decomposition, which plays an important role in $\S$\ref{subsec:freeness}. This is one of the key ingredients.

In $\S$\ref{sec:selmer}, we recall $\Delta$-ordinary $S$-relaxed Selmer groups and collect relevant facts on local and global Galois cohomology. The concept of $n$-admissible primes is introduced.

In $\S$\ref{sec:generalized_divisibility_criterion}, we discuss the generalized divisibility criterion, which lifts the character by character result to the main result. This is another ingredient for our main result. The non-existence of nontrivial finite $\mathcal{O}\llbracket \Gamma_{m'p^\infty}\rrbracket$-submodules plays an important role here.

In $\S$\ref{sec:reduction_to_main_theorem}, we explain how to deduce the Mazur-Tate conjecture for $p$-stabilized forms from the generalized Euler system divisibility.

In $\S$\ref{sec:tame_exceptional_zeros}, we explain how to deduce the Mazur-Tate conjecture for non-$p$-stabilized forms from the Mazur-Tate conjecture for $p$-stabilized forms under Assumption \ref{assu:conditionPO}.

In $\S$\ref{sec:shimura_curves}, we recall the basic notion of Shimura curves and CM points.

In $\S$\ref{sec:level_raising}, we study two kinds of mod $\varpi^n$ level raising results: level raisings at one $n$-admissible prime and at two $n$-admissible primes. We give a detailed and explicit argument for the latter one, which requires some work of Ihara. It turns out that Condition Ihara (Assumption \ref{assu:ihara}) is required for the argument in the non-ordinary setting.

In $\S$\ref{sec:construction_cohomology}, we construct the cohomology classes arising from Heegner points and mod $\varpi^n$ congruences of Hecke eigensystems. We also state the local properties of the cohomology classes.

In $\S$\ref{sec:explicit_reciprocity}, we introduce a mod $\varpi^n$ version of theta elements attached to mod $\varpi^n$ modular forms. Then we describe the first and second explicit reciprocity laws.

In $\S$\ref{sec:chebotarev}, we generalize the existence of infinitely many $n$-admissible primes allowing more general base fields, namely $K(m')$, via the Chebotarev density type argument and also generalize the concept of $n$-admissible sets over $K(m')$. This is also one of the key ingredients.

In $\S$\ref{sec:control_freeness}, we discuss control theorems and give a detailed proof of the freeness of $\Delta$-ordinary $S$-relaxed compact Selmer groups over $K(m)$ following \cite{bertolini-darmon-derived-heights-1994}. For the generalization to $K(m)$, we uses the enhanced isotypic decomposition introduced in $\S$\ref{subsec:enhanced_decomposition}.

In $\S$\ref{sec:euler_systems}, we proceed the Euler system argument (via induction).
We examine the Euler system argument over $K(m)$. Two parts in the original proof depends heavily on the properties of $p$-extensions ($\S$\ref{subsec:first_step} and $\S$\ref{subsec:appeal_to_second}). We overcome these issues by using the generalization of the infinitude of $n$-admissible primes over $K(m')$ in $\S$\ref{sec:chebotarev}.
Here, we obtain the character by character result.

In $\S$\ref{sec:completion_proof}, we complete the proof of the main theorem with help of $\S$\ref{sec:fitting}.

In $\S$\ref{sec:speculation}, we investigate a more global aspect of the Mazur-Tate type conjecture. In fact, it is more coincident with the spirit of the original Mazur-Tate conjecture. See \cite[Introduction]{mazur-tate}.

\section{Remarks on the main conjecture} \label{sec:mainconj}
We briefly review the state of art of the anticyclotomic main conjecture for modular forms of weight 2 up to now and explain how the main conjecture and the exact control theorem imply the Mazur-Tate conjecture for anticyclotomic extensions of $p$-power conductor.

\subsection{Ordinary case}
The Euler system method yields the following statement.
\begin{thm}[\cite{bertolini-darmon-imc-2005}, \cite{pw-mu}, and \cite{chida-hsieh-main-conj}] \label{thm:mainconj_minimal}
Assume the following conditions:
\begin{enumerate}
\item $(\overline{\rho}, N^-)$ satisfies Condition CR.
\item $a_p(f) \not\equiv \pm 1 \pmod{\varpi}$. (Condition PO)
\item $f$ is $N^+$-minimal.
\item $\overline{\rho}$ has big image.
\end{enumerate}
Then $L_p(K(p^\infty), f_\alpha) \in \mathrm{char}_\Lambda \left( \mathrm{Sel}(K(p^\infty), A_f)^\vee  \right) $.
\end{thm}
\begin{rem}[on the $N^+$-minimality]
In the original proof, the freeness of compact $\Delta$-ordinary $S$-relaxed Selmer groups over $K(p^\infty)$ is obtained from the freeness of compact $\Delta$-ordinary $S$-relaxed Selmer groups  over $K(p^r)$ by taking the inverse limit. Here, we need the $N^+$-minimality assumption.
\end{rem}
\begin{rem}[Condition CR vs.~Condition CR$^+$]
Condition CR$^+$, which is more restrictive than condition CR, is imposed in \cite{chida-hsieh-main-conj}, but it can be replaced by condition CR for the case of weight 2. The reason why we need such conditions is to have a mod $p$ multiplicity one result of the localized Hecke module we deal with. Condition CR comes from the arithmetic of Jacobians of Shimura curves and the Jacquet-Langlands correspondence, and Condition CR$^+$ comes from the patching argument \`{a} la Taylor-Wiles, Diamond, and Fujiwara. One advantage of Condition CR$^+$ makes it possible to deal with the higher weight case, but a certain ramification condition at primes dividing $N^+$ is imposed.
\end{rem}
Using the vanishing of algebraic $\mu$-invariant of the $N^+$-minimal case via Theorem \ref{thm:mainconj_minimal} and the Iwasawa-theoretic argument in \cite{kim-pollack-weston}, one can obtain the following statement. 
\begin{prop}[Theorem \ref{thm:freeness} over $K(p^\infty)$, \cite{kim-pollack-weston}] \label{prop:kpw}
The freeness of compact $\Delta$-ordinary $S$-relaxed Selmer groups over $K(p^\infty)$  holds without assuming the $N^+$-minimality.
\end{prop}
Recycling the original argument with Proposition \ref{prop:kpw}, we can remove $N^+$-minimality in Theorem \ref{thm:mainconj_minimal}.
\begin{cor} \label{cor:main_conj_nonminimal}
Assume the following conditions:
\begin{enumerate}
\item $(\overline{\rho}, N^-)$ satisfies Condition CR.
\item $f$ satisfies Condition PO.
\item $\overline{\rho}$ has big image.
\end{enumerate}
Then $L_p(K(p^\infty), f_\alpha) \in \mathrm{char}_\Lambda \left( \mathrm{Sel}(K(p^\infty), A_f)^\vee  \right) $.
\end{cor}
On the other side, the Eisenstein congruence method \`{a} la Ribet-Skinner-Urban yields the opposite divisibility.
\begin{thm}[{\cite[Theorem 3.37]{skinner-urban}}]
Assume the following conditions:
\begin{enumerate}
\item $p$ splits in $K$.
\item $\overline{\rho}$ is $N^-$-minimal.
\item the image of $\overline{\rho}$ contains a conjugate of $\mathrm{GL}_2(\mathbb{F}_p)$.
\end{enumerate}
Then 
$ \mathrm{char}_\Lambda \left( \mathrm{Sel}(K(p^\infty), A_f)^\vee  \right)  \subseteq \left( L_p(K(p^\infty), f_\alpha) \right) $.
\end{thm}
Combining the above two statements, we have the following statement.
\begin{cor}[The main conjecture]
Assume the following conditions:
\begin{enumerate}
\item $p$ splits in $K$.
\item $\overline{\rho}$ is $N^-$-minimal.
\item $f$ satisfies Condition PO.
\item the image of $\overline{\rho}$ contains a conjugate of $\mathrm{GL}_2(\mathbb{F}_p)$.
\end{enumerate}
Then $ \mathrm{char}_\Lambda \left( \mathrm{Sel}(K(p^\infty), A_f)^\vee  \right) = \left( L_p(K(p^\infty), f_\alpha) \right) $.
\end{cor}

\subsection{Non-ordinary case} \label{subsec:pm-main-conj}
For this subsection, assume that the modular form $f$ is non-ordinary at $p$.
One can obtain a similar result using the $\pm$-theory \`{a} la Kobayashi-Pollack (\cite{kobayashi-thesis}, \cite{pollack-thesis})  under the assumption $a_p(f)=0$.
We need the following condition for the non-ordinary case.
\begin{assu}[Condition Ihara] \label{assu:ihara}
We say that $N^+$ satisfies \textbf{Condition Ihara} if there exists an integer $q$ dividing $N^+$ such that $q \geq 4$ and $p$ does not divide its Euler totient $\varphi(q)$.
\end{assu}
\begin{thm}[{\cite{darmon-iovita}, \cite{pw-mu}, and \cite{kim-pollack-weston}}]
Assume the following conditions:
\begin{enumerate}
\item $p$ splits in $K$.
\item $(\overline{\rho}, N^-)$ satisfies Condition CR.
\item $a_p(f) =0$.
\item $\overline{\rho}$ has big image.
\item $K(p^\infty)/K$ is totally ramified.
\item $N^+$ satisfies Condition Ihara.
\end{enumerate}
Then 
$L^\pm_p(K(p^\infty), f) \in \mathrm{char}_\Lambda \left( \mathrm{Sel}^\pm(K(p^\infty), A_f)^\vee  \right)$,
respectively.
\end{thm}
\begin{rem}
Condition (1) and (5) are required to use the local Iwasawa theory for Lubin-Tate extensions following \cite{iovita-pollack}.
\end{rem}
\begin{rem}
Note that Condition Ihara is missing in previous papers. In order to obtain the second explicit reciprocity law in this setting, Condition Ihara is required since $p$ is not in the level for the non-ordinary setting.
Condition Ihara is used for mod $\varpi^n$ level raising at two $n$-admissible primes ($\S$\ref{subsec:level_raising_at_two_prime}). This condition with the $\Gamma_1(q)$-level structure with $q \geq 4$ makes the arithmetic subgroup defining Shimura curves torsion-free. In some sense, this condition removes the analogue of elliptic points for Shimura curves. The exactly same phenomenon is also observed in the definite case as in \cite[$\S$4]{buzzard-families}.
Note that $p$ is in the level in all the former literatures except \cite{darmon-iovita} and \cite[$\S$2.4 and the latter part of Theorem 4.1]{pw-mu}. On the other hand, all the former literatures also need to remove $p$ in the level of Hecke modules localized at the maximal ideal using Condition PO to invoke Ihara's lemma. See \cite[Lemma 5.3]{chida-hsieh-main-conj} for detail. If $N^+$ is divisible by 4, then Condition Ihara always holds.
\end{rem}

\section{Modular forms and the Bertolini-Darmon elements} \label{sec:theta_elements}
We review modular forms of weight two in the quaternionic setting and construct the Bertolini-Darmon elements, which is an anticyclotomic analogue of Mazur-Tate elements, defined in \cite[$\S$2.7]{bertolini-darmon-mumford-tate-1996}. See \cite[$\S$2.3]{kim-summary} for a more detailed construction.

\subsection{Modular forms of weight two} \label{subsec:modular_forms}
Let $f \in S_2(\Gamma_0(N))^{\mathrm{new}}$ be a normalized newform.
Let $B$ be the definite quaternion algebra over $\mathbb{Q}$ of discriminant $N^-$ and $R \subseteq B$ be an oriented Eichler order of level $N^+$. For any $\mathbb{Z}$-module $M$, $\widehat{M} := M \otimes_{\mathbb{Z}} \widehat{\mathbb{Z}}$.
Using an integral version of Jacquet-Langlands correspondence \cite[Theorem 2.1 and Remark 2.2]{kim-summary}, we regard $f$ as an normalized $\mathcal{O}$-valued modular form on $B^\times \backslash \widehat{B}^\times / \widehat{R}^\times$, which is defined up to multiplication by $\mathcal{O}^\times$.
We define the space of $\mathcal{O}$-valued (quaternionic) modular forms by
$$M^{N^-}_2(N^+) := M^{N^-}_2(N^+, \mathcal{O}) = \lbrace f: B^{\times} \backslash \widehat{B}^\times / \widehat{R}^\times \to \mathcal{O} \rbrace$$
and the space of $\mathcal{O}$-valued (quaternionic) cuspforms by
$$S^{N^-}_2(N^+) := S^{N^-}_2(N^+, \mathcal{O}) = \lbrace f: B^{\times} \backslash \widehat{B}^\times / \widehat{R}^\times \to \mathcal{O} \rbrace \textrm{ modulo the constant function}$$
and let $\mathbb{T}^{N^-}(N^+)$ be the Hecke algebra faithfully acting on $S^{N^-}_2(N^+)$.
The integral Jacquet-Langlands correspondence identifies
\begin{itemize}
\item $\mathbb{T}^{N^-}(N^+)$ and the $N^-$-new quotient of the classical Hecke algebra $\mathbb{T}(\Gamma_0(N))$, and
\item $S^{N^-}_2(N^+)$ and $S_2(\Gamma_0(N))^{N^-\textrm{-new}}$ as Hecke modules.
\end{itemize}
For each $b \in \widehat{B}^\times$, we write $[b]_U \in B^{\times} \backslash \widehat{B}^\times / \widehat{R}^\times$ represented by $b$ where $U = \widehat{R}^\times$.
We have a canonical Hecke-equivariant identification 
$$M^{N^-}_2(N^+, \mathbb{Z}) \simeq \mathbb{Z}[B^{\times} \backslash \widehat{B}^\times / \widehat{R}^\times]$$
where $\mathbb{Z}[B^{\times} \backslash \widehat{B}^\times / \widehat{R}^\times]$ is the divisor group which is a Hecke module via Picard functoriality.
Let $\tau \in \widehat{B}^\times$ be the Atkin-Lehner involution defined by
$\tau_q = \left( \begin{smallmatrix} 0 & 1 \\ N^+ & 0 \end{smallmatrix} \right)$ if $q \mid N^+$ and $\tau_q = 1$ otherwise.
Then we can explicitly define the canonical identification by $[b]_U \mapsto f_b$
where
$$f_b (b') :=  \left \lbrace
    \begin{array}{ll}
     \# \left\vert  \left( B^\times \cap b\widehat{R}^\times b^{-1} \right) /\mathbb{Q}^\times \right\vert  & \textrm{ if } [b]_U = [b']_U \\ 
     0  & \textrm{ if } [b]_U \neq [b']_U
    \end{array}
    \right.$$
We define the perfect pairing
$$\langle -,- \rangle_U : M^{N^-}_2(N^+,A) \times M^{N^-}_2(N^+,A) \to A$$
where $A$ is any $p$-adic ring by
$$\langle f_1, f_2 \rangle_U = \sum_{[b]_U} f_1 (b)\cdot f_2(b\tau) \cdot \# \left\vert  \left( B^\times \cap b\widehat{R}^\times b^{-1} \right) /\mathbb{Q}^\times \right\vert^{-1}.$$

\begin{rem} \label{rem:normalization_modular_forms}
If $p >3$, $\# \left\vert  \left( B^\times \cap b\widehat{R}^\times b^{-1} \right) /\mathbb{Q}^\times \right\vert$ is always invertible in $A$.
If 
$$R^\times_q \subseteq \left\lbrace A \in \mathrm{GL}_2(\mathbb{Z}_q) : A \equiv \left(  \begin{matrix}
* & * \\ 0 & 1
\end{matrix}  \right) \pmod{q}\right\rbrace $$
where $R^\times_q$ is the $q$-part of $\widehat{R}^\times$ for some prime $q \geq 4$, then $\# \left\vert  \left( B^\times \cap b\widehat{R}^\times b^{-1} \right) /\mathbb{Q}^\times \right\vert =1$.
\end{rem}

Then the action of $\mathbb{T}^{N^-}(N^+)$ on $M^{N^-}_2(N^+)$ is self-adjoint with respect to the paring $\langle -,- \rangle_U$; we have $\langle tf_1,f_2 \rangle_U  = \langle f_1,tf_2 \rangle_U$  for all $t \in  \mathbb{T}^{N^-}(N^+)$.
By the explicit correspondence, we can also define the pairing
$$\langle -,- \rangle_U : M^{N^-}_2(N^+,A) \times A[B^{\times} \backslash \widehat{B}^\times / \widehat{R}^\times] \to A$$
with the same notation and compare with the evaluation as follows:
$$\langle f, b\tau \rangle_U = \langle f, f_{b\tau} \rangle_U = f(b) .$$
\subsection{Construction of the Bertolini-Darmon elements} \label{subsec:theta_elements}
Let $\mathrm{rec}_K : K^\times \backslash \widehat{K}^\times \to \mathrm{Gal}(K^{\mathrm{ab}}/K)$ be the reciprocity map in class field theory with geometric normalization.
Let $\mathcal{O}_{K,m_0} = \mathbb{Z} + m_0\mathcal{O}_K$ be the order of $K$ of conductor $m_0$ for $m_0 \geq 1$.
Fix an oriented optimal embedding $\psi : K \to B$. Then $\psi$ induces a family of maps
$$\psi_{m_0}: K^\times \backslash \widehat{K}^\times / \widehat{\mathcal{O}}^\times_{K,m_0}\widehat{\mathbb{Q}}^\times \to B^\times \backslash \widehat{B}^\times / \widehat{R}^\times$$
for $m_0 \geq 1$ as in \cite[$\S$4.3.1]{chida-hsieh-main-conj} (with slight generalization), and the reciprocity map $\mathrm{rec}_K$ induces an isomorphism
$$[-]_{m_0}: K^\times \backslash \widehat{K}^\times / \widehat{\mathcal{O}}^\times_{K,m_0}\widehat{\mathbb{Q}}^\times  \overset{\simeq}{\to} G_{m_0} = \mathrm{Gal}(H(m_0)/K).$$
Considering the following commutative diagram
\[
\xymatrix{
K^\times \backslash \widehat{K}^\times \ar[r]^-{\mathrm{rec}_K} \ar@{->>}[d]  & \mathrm{Gal}(K^{\mathrm{ab}}/K) \ar@{->>}[d] \\
K^\times \backslash \widehat{K}^\times / \widehat{\mathcal{O}}^\times_{K,m_0}\widehat{\mathbb{Q}}^\times  \ar[r]^-{[-]_{m_0}}_-{\simeq} \ar[d]^-{\psi_{m_0}} & G_{m_0} = \mathrm{Gal}(H(m_0)/K) \\
B^\times \backslash \widehat{B}^\times / \widehat{R}^\times \ar[r]^-{f} & \mathcal{O},
}
\]
we define an element in group ring $\mathcal{O}[G_{m_0}]$ depending on $f$ and $\psi$ by
$$ \widetilde{\theta}_{m_0}(f) := \sum_{ a  } f(\psi_{m_0}(a)) [a]_{m_0} \in \mathcal{O}[G_{m_0}]$$
where $a$ runs over $K^\times \backslash \widehat{K}^\times / \widehat{\mathcal{O}}^\times_{K,m_0}\widehat{\mathbb{Q}}^\times$.
\begin{rem}
The map $x_{m_0}$ in \cite[$\S$4.3.1]{chida-hsieh-main-conj} corresponds to $\psi_{p^{m_0}}$ in our setting.
The image $\psi_{m_0}(a)$ is called a Gross points of conductor $m_0$. If $K(p^\infty)/K$ is totally ramified, then the map $\psi_{p^r}$ becomes much simpler as in \cite[$\S$2.2]{darmon-iovita}.
\end{rem}

For $m  = \prod q^{r_i}_i$, write $m_0 = \prod q^{r_i + 1}_i$.
We define the \textbf{theta element $\theta_{m}(f)$ attached to $f$ and $\psi$} by the image of $\widetilde{\theta}_{m_0}(f)$ of the natural projection $\mathcal{O}[G_{m_0}] \to \mathcal{O}[\Gamma_m]$ 
and \textbf{its involution $\theta^*_{m}(f)$} by the image of $\theta_{m}(f)$ of the involution map $\mathcal{O}[\Gamma_m] \to \mathcal{O}[\Gamma_m]$ defined by $\gamma \mapsto \gamma^{-1}$ for $\gamma \in \Gamma_m$.
\begin{rem}
Theta elements $\theta_{m}(f)$ and $\theta^*_{m}(f)$ are well-defined only up to multiplication by $\Gamma_m$ and a unit in $\mathcal{O}$.
\end{rem}
We define the \textbf{Bertolini-Darmon element $L_p(K(m), f)$} by
$$L_p(K(m), f):=\theta_{m}(f) \cdot \theta^*_{m}(f)  \in \mathcal{O}[\Gamma_m],$$
and it is well-defined up to a unit in $\mathcal{O}$.
\begin{rem}
See \cite[Theorem A]{hung-nonvanishing} for the interpolation formula of $L_p(K(m), f)$ for ring class characters of general conductors.
\end{rem}

Let $f_\alpha$ be the $p$-stabilization of $f$ with the unit $U_p$-eigenvalue $\alpha = \alpha_p(f)$.
Let 
$$\mathrm{cores}^{r}_{r-1} : \mathcal{O}[\Gamma_{m'p^{r-1}}] \to \mathcal{O}[\Gamma_{m'p^{r}}]$$
 be the corestriction map defined by
$g \mapsto \sum_h hg$ where $h$ runs over $\Gamma_{m'p^{r}} / \Gamma_{m'p^{r-1}}$ for all $r \geq 1$.
We define the \textbf{$p$-stabilized theta element $\theta_{m'p^r}(f_\alpha)$} by
$$\theta_{m'p^r}(f_\alpha) := \frac{1}{\alpha^r} \left( \theta_{m'p^r}(f) - \frac{1}{\alpha} \mathrm{cores}^{r}_{r-1} \theta_{m'p^{r-1}}(f) \right)$$ 
for $r \geq 1$.
Then $\left\lbrace \theta_{m'p^r}(f_\alpha) \right\rbrace_{r}$ forms a norm-compatible sequence and
the inverse limit $\theta_{m'p^\infty}(f_\alpha)$ with respect to $r$
is an element in $\mathcal{O}\llbracket\Gamma_{m'p^\infty}\rrbracket$.
In the same manner, we define the \textbf{generalized $p$-adic $L$-function $L_p(K(m'p^\infty), f_\alpha)$} by
$$L_p(K(m'p^\infty), f_\alpha) = \theta_{m'p^\infty} (f_\alpha) \cdot \theta^*_{m'p^\infty}(f_\alpha) \in \mathcal{O}\llbracket\Gamma_{m'p^\infty}\rrbracket .$$
\begin{rem} \label{rem:lift_cores}
Note that $\mathrm{cores}^{r}_{r-1} (a) = \Phi_{p^r}(\gamma_{p^r}) \cdot a \in \mathcal{O}[\Gamma_{m'p^{r}}]$ for $a \in \mathcal{O}[\Gamma_{m'p^{r-1}}]$. Moreover,
$\Phi_{p^r}(\gamma_{p^r}) \cdot a = \Phi_{p^r}(\gamma_{p^r}) \cdot \widetilde{a}$ for any lift $\widetilde{a}$ of $a$ to $\mathcal{O}[\Gamma_{m'p^{r}}]$.
\end{rem}

\section{Fitting ideals, character convention, and the enhanced isotypic decomposition} \label{sec:fitting}
\subsection{Fitting ideals}
Let $A$ be a Noetherian ring.
Let $X$ be a finitely generated $A$-module, so it is of finitely presentation.
Fix a presentation
\[
\xymatrix{
A^s \ar[r]^-{h} & A^r \ar[r] & X \ar[r] & 0 
}
\]
where $h \in \mathrm{M}_{r \times s}(A)$.
\begin{defn}[Fitting ideals]
The \textbf{Fitting ideal $\mathrm{Fitt}_A(X)$ of $X$ over $A$} is the ideal of $A$ generated by the determinant of the $r \times r$-minors of the matrix $h$.
\end{defn}
It is well-known that the Fitting ideal is independent of the choice of a presentation of $X$ (so $h$).

\begin{lem}[Exact sequences; {\cite[9, Appendix]{mazur-wiles-main-conj}}] \label{lem:fitting_exact}
Let $0 \to X_1 \to X_2 \to X_3 \to 0$ be an exact sequence of $A$-modules.
Then $\mathrm{Fitt}_{A} (X_1) \cdot \mathrm{Fitt}_{A} (X_3) \subseteq \mathrm{Fitt}_{A} (X_2)$.
\end{lem}
\begin{lem}[Projective limit; {\cite[10, Appendix]{mazur-wiles-main-conj}}] \label{lem:fitting_lim}
Let $A$ be a complete Noetherian local ring, $X$ an $A$-module of finite presentation, which can be defined by the projective limit of quotient $R$-modules:
$$X \simeq \varprojlim_n X_n $$
for $n \in \mathbb{N}$.
Then $\mathrm{Fitt}_{A}(X) = \bigcap^\infty_{n=1} \mathrm{Fitt}_{A}(X_n)$.
\end{lem}

\begin{lem}[Base change; {\cite[$\S$3.1.5]{skinner-urban}}]  \label{lem:fitting-ideals-base-change} $ $
\begin{enumerate}
\item For any Noetherian $A$-algebra $B$,
 $\mathrm{Fitt}_{B}(X \otimes_{A} B) = \mathrm{Fitt}_A(X)B$
 in $B$.
\item In particular, if $I \subset A$ is an ideal, then 
 $\mathrm{Fitt}_{A/I}(X /IX) = \mathrm{Fitt}_A(X) \Mod{I}$
in $A/I$.
\end{enumerate}
\end{lem}

\subsection{Convention on characters} \label{subsec:convention_characters}
Since we use various characters in this article, we fix the convention here.
Consider three characters
\[
\xymatrix{
\chi: \Gamma_{m'p^\infty}  \to \overline{\mathbb{Q}}^\times_p , &
\omega: \Gamma_{m'}  \to \overline{\mathbb{Q}}^\times_p , &
\varphi: \Gamma_{m'p^\infty}  \to \overline{\mathbb{Q}}^\times_p
}
\]
where 
\begin{itemize}
\item $\chi$ is a ring class character of finite order used in $\S$\ref{subsec:main_theorem_consequences},
\item $\omega$ is a character on the prime-to-$p$ part $\Gamma_{m'}$ used in $\S$\ref{subsec:enhanced_decomposition} and $\S$\ref{sec:control_freeness}, and
\item $\varphi$ is a character used in $\S$\ref{sec:euler_systems} for the specialization.
\end{itemize}
We recall the $\mathbb{Q}_p$-conjugacy relation of the characters in \cite[$\S$1.4]{kurihara-fitting}.
\begin{defn}[Conjugacy relation of characters over $\mathbb{Q}_p$]
We say that two characters $\omega_1$ and $\omega_2$ are \textbf{$\mathbb{Q}_p$-conjugate} if $\omega_1 = \sigma \omega_2$ for some $\sigma \in G_{\mathbb{Q}_p}$.
\end{defn}

For any ($p$-adic) ring $\mathcal{O}'$ of characteristic zero, let $\mathcal{O}'_{?}$ be the $\mathcal{O}'$-algebra generated by the values of $? \in \lbrace \chi, \omega, \varphi \rbrace$. Using these rings, we regard character $?$ as a $\mathcal{O}'_{?}$-valued character.
We also denote a uniformizer $\mathcal{O}'_{?}$ by $\varpi'_?$,
the residue field by $\mathbb{F}'_?$, and 
$\mathcal{O}'_{?,n} := \mathcal{O}'_{?} /(\varpi'_?)^{n}$.
Note that all $\mathcal{O}_{\chi}$, $\mathcal{O}_\omega$, and $\mathcal{O}_\varphi$ are discrete valuation rings of characteristic zero. Also, they admit the action of $\Gamma_{m'p^\infty}$ or $\Gamma_{m'}$ via the characters. 
\begin{rem}
Although the character $\chi$ and $\varphi$ equal, we use different notation since their roles are completely different.
The character $\varphi$ is only used as the map of group rings $\mathcal{O}\llbracket\Gamma_{m'p^\infty}\rrbracket \to \mathcal{O}_{\varphi}$. 
\end{rem}
Let $M$ be an $\mathcal{O}\llbracket\Gamma_{m'p^\infty}\rrbracket$-module where $\mathcal{O}$ is any $p$-adic ring of characteristic zero.
Then we define the character-isotypic quotients
\[
\xymatrix{
M_{\chi} := M \otimes_{\chi} \mathcal{O}_{\chi} , &
M_{\varphi} := M \otimes_{\varphi} \mathcal{O}_{\varphi}
}
\]
where the tensor products are taken over $\mathcal{O}\llbracket\Gamma_{m'p^\infty}\rrbracket$ via $\chi$ and $\varphi$, respectively, and
$$M_{\omega} := M \otimes_{\omega} \mathcal{O}_{\omega}$$
where the tensor products are taken over $\mathcal{O}[\Gamma_{m'}]$ via $\omega$.
We call $M_{?}$ the \textbf{$?$-isotypic quotient of $M$} or the \textbf{specialization of $M$ via $? \in \lbrace \chi, \omega, \varphi \rbrace$}.

\subsection{The enhanced isotypic decomposition} \label{subsec:enhanced_decomposition}
We recall an \emph{enhanced} version of the isotypic decomposition of modules over group rings following \cite[$\S$1.3, 1.4]{kurihara-fitting}. Here we use isotypic quotients not isotypic subspaces.

Let $m = m'p^r$ be an integer with $(m',p)  = 1$.
Let $\omega: \Gamma_{m'} \to \mathcal{O}_\omega$
be a character as in $\S$\ref{subsec:convention_characters}.
Then we have a decomposition
$$\mathcal{O}[\Gamma_{m'}] = \oplus_\omega \mathcal{O}_\omega$$
where $\omega$ runs over $\mathrm{Hom}(\Gamma_{m'}, \overline{\mathbb{Q}}^\times_p)$ up to the $\mathbb{Q}_p$-conjugation.
The existence of the decomposition essentially comes from the fact $m'$ is invertible in $\mathcal{O}$. The decomposition works even over $\mathcal{O}_n = \mathcal{O}/\varpi^n$ for the same reason.
Then we have the enhanced isotypic decomposition
$M  = \oplus_\omega \left( M \otimes_{\omega}  \mathcal{O}_\omega \right)$.
The following lemma is straightforward.
\begin{lem} \label{lem:kurihara_decomposition_freeness}
Let $M$ be a $\mathcal{O}\llbracket\Gamma_{m'p^\infty}\rrbracket$-module.
If $M_\omega$ is free of rank $s$ over $\mathcal{O}_\omega\llbracket\Gamma_{p^\infty}\rrbracket$ for each $\omega \in \mathrm{Hom}(\Gamma_{m'}, \overline{\mathbb{Q}}^\times_p)$ up to the $\mathbb{Q}_p$-conjugation, then
$M$ is free of rank $s$ over $\mathcal{O}\llbracket\Gamma_{m'p^\infty}\rrbracket$.
\end{lem}

\section{Galois cohomology and Selmer groups} \label{sec:selmer}
We recall the language of Galois cohomology following \cite[$\S$1]{chida-hsieh-main-conj}.
\subsection{Galois cohomology}
Let $L$ be an algebraic extension of $\mathbb{Q}$. For a discrete $G_\mathbb{Q}$-module $M$, we put
\[
\xymatrix{
\mathrm{H}^1(L_\ell, M) : = \bigoplus_{\lambda \vert \ell} \mathrm{H}^1(L_\lambda, M) , & \mathrm{H}^1(I_{L_\ell}, M) : = \bigoplus_{\lambda \vert \ell} \mathrm{H}^1(I_{L_\lambda}, M)
}
\]
where $\lambda$ runs over all primes of $L$ dividing $\ell$ and $I_{L_\lambda} \subset G_{L_\lambda}$ is the inertia subgroup.
Denote the \textbf{restriction map at $\ell$} by
$$\mathrm{res}_\ell : \mathrm{H}^1(L, M) \to \mathrm{H}^1(L_\ell, M).$$
Define the \textbf{finite part of $\mathrm{H}^1(L_\ell, M)$} by
$$\mathrm{H}^1_{\mathrm{fin}}(L_\ell, M) := \ker \left( \mathrm{H}^1(L_\ell, M) \to \mathrm{H}^1(I_{L_\ell}, M) \right)$$
and the \textbf{singular part of  $\mathrm{H}^1(L_\ell, M)$} by 
$$\mathrm{H}^1_{\mathrm{sing}}(L_\ell, M) := \dfrac{\mathrm{H}^1(L_\ell, M)}{\mathrm{H}^1_{\mathrm{fin}}(L_\ell, M)} .$$
The natural map induced by the restriction
$$\partial_\ell : \mathrm{H}^1(L, M) \to \mathrm{H}^1_{\mathrm{sing}}(L_\ell, M)$$
is called the \textbf{residue map}. 

If $\kappa \in \mathrm{H}^1(L,M)$ with $\partial_\ell (\kappa) = 0$, then the image of $\kappa$ under $\mathrm{res}_\ell$ lies in $\mathrm{H}^1_{\mathrm{fin}}(L_\ell, M)$. We write $v_\ell (\kappa) \in \mathrm{H}^1_{\mathrm{fin}}(L_\ell, M)$ for the image. For an $n$-admissible prime $\ell$, $v_\ell (\kappa) \in \mathrm{H}^1_{\mathrm{fin}}(L_\ell, T_{f,n})$ can be extended to the case $\partial_\ell (\kappa) \neq 0$ due to the decomposition in Proposition \ref{prop:local_computation}.(5). See Definition \ref{defn:localization_n-admissible}.

%

\subsection{Galois representations} \label{subsec:galois_representations}
Let $\rho_f$ be the Galois representations attached to a newform $f \in S_2(\Gamma_0(N))^{\mathrm{new}}$.
Then $\rho_f$ satisfies the following properties:
\begin{enumerate}
\item $\rho_f$ is unramified outside $pN$;
\item the restriction to $G_{\mathbb{Q}_p}$ is of the form $\left( \begin{matrix}
\chi^{-1}_p \varepsilon & *\\
0  & \chi_p
\end{matrix} \right)$, where $\chi_p$ is unramified and $\chi_p(\mathrm{Frob}_p) = \alpha$;
\item for all $\ell$ dividing $N$ exactly, the restriction to $G_{\mathbb{Q}_\ell}$ is of the form $\left( \begin{matrix}
\pm\varepsilon & *\\
0  & \pm \mathbf{1}
\end{matrix} \right)$.
\end{enumerate}

\subsection{Ordinary local conditions}
Let $L/K$ be a finite extension. Let $M$ be one of $V_f$, $T_f$, $A_f$, $T_{f,n}$, and $A_{f,n}$. 
For $\ell$ a prime factor of $pN^-$ or an $n$-admissible prime (Definition \ref{defn:n-admissible-primes}),
we define $G_{\mathbb{Q}_\ell}$-invariant submodule $F^+_\ell M$.
If one have $F^+_\ell V_f$, then one can define $F^+_\ell M$ for other $M$ as follows. 
Consider the exact diagram
\[
\xymatrix@R=1.5em{
 & 0 \ar[d] &  0 \ar[d]  & 0 \ar[d] \\
 & T_f \ar[d]^-{\varpi^n} & F^+_\ell V_f \ar[d]   & A_{f,n} \ar[d]^-{\beta} \\
0 \ar[r] & T_f \ar[r]^-{\mathrm{inc}} \ar[d]^-{\alpha} & V_f \ar[r]^-{\pi} \ar[d] & A_f \ar[r] \ar[d]^-{\varpi^n} & 0 \\
 & T_{f,n} \ar[d] &   V_f / F^+_\ell V_f \ar[d]  & A_f \ar[d] \\
 & 0  & 0   & 0 .
}
\]
Then we define
\[
\xymatrix@R=0.5em@C=0.5em{
F^+_\ell T_f  := \mathrm{inc}^{-1} \left( F^{+}_\ell V_f \right) , & F^{+}_\ell A_f  := \pi \left( F^{+}_\ell V_f \right) , \\
F^+_\ell T_{f,n}  := \alpha \left( F^{+}_\ell T_f \right)  = \alpha \left( \mathrm{inc}^{-1} \left( F^{+}_\ell V_f \right) \right) , & F^{+}_\ell A_{f,n}  := \beta^{-1} \left( F^{+}_\ell A_f \right) = \beta^{-1} \left( \pi \left( F^{+}_\ell V_f \right) \right) .
}
\]
Using the above definitions, for $M = T_{f,n}$ and $A_{f,n}$,
we define the \textbf{ordinary submodule $\mathrm{H}^1_{\mathrm{ord}}(L_\ell, M)$ of $\mathrm{H}^1(L_\ell, M)$}
by
$$\mathrm{H}^1_{\mathrm{ord}} (L_\ell, M) := \ker \left(\mathrm{H}^1 (L_\ell, M) \to \mathrm{H}^1 (L_\ell, M/F^+_\ell M) \right).$$
Thus, it suffices to specify $F^+_\ell V_f$ for $\ell$ a prime factor of $pN^-$ or an $n$-admissible prime. We use $\S$\ref{subsec:galois_representations} here.
\begin{enumerate}
\item ($\ell = p$) Let $F^+_p V_f \subset V_f$ be the $E$-rank-one $G_{\mathbb{Q}_p}$-subspace on which the inertia subgroup $I_p$ acts via $\varepsilon$.
\item ($\ell$ dividing $N^-$) Let $F^+_\ell V_f \subset V_f$ be the $E$-rank-one $G_{\mathbb{Q}_\ell}$-subspace on which the inertia subgroup $I_p$ acts via $\varepsilon$ or $\varepsilon \tau_\ell$, where $\varepsilon$ is the $p$-adic cycotomic character and $\tau_\ell$ is the non-trivial unramified quadratic character of $G_{\mathbb{Q}_\ell}$.
\end{enumerate}
The case for $n$-admissible primes is defined as follows.
\begin{defn}[$n$-admissible primes; {\cite[Definition 2.20]{bertolini-darmon-derived-heights-1994}, \cite[$\S$2.2]{bertolini-darmon-imc-2005}}] \label{defn:n-admissible-primes}
A rational prime $\ell$ is called \textbf{$n$-admissible for $f$} if it satisfies the following conditions:
\begin{enumerate}
\item $\ell$ does not divide $pN$;
\item $\ell$ is inert in $K/\mathbb{Q}$;
\item $\ell^2 \not\equiv 1 \pmod{p}$;
\item $ \epsilon_\ell a_\ell(f) \equiv \ell+1 \pmod{\varpi^n}$ with $\epsilon_\ell = \pm 1$.
\end{enumerate}
\end{defn}
If $\ell$ is $n$-admissible, then $V_f$ is unramified at $\ell$ and $(\mathrm{Frob}_\ell - \epsilon_\ell)  (\mathrm{Frob}_\ell - \epsilon_\ell \ell) = 0$ in $A_{f,n}$ and $T_{f,n}$ for the Frobenius $\mathrm{Frob}_\ell$ of $G_{\mathbb{Q}_\ell}$.
We let
\begin{itemize}
\item $F^+_\ell A_{f,n} \subset A_{f,n}$ be the unique $\mathcal{O}_n$-corank-one submodule on which $\mathrm{Frob}_\ell$ acts via the multiplication by $\epsilon_\ell \ell$, and
\item $F^+_\ell T_{f,n} \subset T_{f,n}$ be the unique $\mathcal{O}_n$-rank-one submodule on which $\mathrm{Frob}_\ell$ acts via the multiplication by $\epsilon_\ell \ell$.
\end{itemize}
Then one can define the ordinary submodules $\mathrm{H}^1_{\mathrm{ord}} (L_\ell, A_{f,n})$ and  $\mathrm{H}^1_{\mathrm{ord}} (L_\ell, T_{f,n})$ similarly.

\subsection{Residual $\Delta$-ordinary $S$-relaxed Selmer groups} \label{subsec:selmer}
Let $\Delta$ be a square-free integer such that $\Delta / N^-$ is a product of $n$-admissible primes.
Let $S$ be a square-free integer with $(S, \Delta N^+p) = 1$.
\begin{defn}[$\Delta$-ordinary $S$-relaxed Selmer groups] \label{defn:selmer}
Let $M = A_{f,n}$ or $T_{f,n}$. We define
\textbf{$\Delta$-ordinary $S$-relaxed Selmer groups $\mathrm{Sel}^S_\Delta (L, M)$ attached to data $(f, n, \Delta, S)$} to be the group of elements $s$ in $\mathrm{H}^1(L,M)$ satisfying the following properties:
\begin{enumerate}
\item $\partial_q (s) = 0$ for all $q \nmid p\Delta S$.
\item $\mathrm{res}_q(s) \in \mathrm{H}^1_{\mathrm{ord}}(L_q, M)$ for $q \mid p\Delta$.
\item $\mathrm{res}_q(s)$ is arbitrary for $q \mid S$.
\end{enumerate}
In other words, the following sequence is exact:
\[
\xymatrix{ 
0 \ar[r] & \mathrm{Sel}^S_\Delta (L, M) \ar[r] & \mathrm{H}^1 (L, M) \ar[r]^-{\prod_q \mathrm{res}_q} &
\displaystyle\prod_{q \mid p}  \dfrac{\mathrm{H}^1(L_q, M)}{\mathrm{H}^1_{\mathrm{ord}}(L_q, M)} \times
\prod_{q \mid \Delta}  \dfrac{\mathrm{H}^1(L_q, M)}{\mathrm{H}^1_{\mathrm{ord}}(L_q, M)} \times
\prod_{q \nmid S}  \dfrac{\mathrm{H}^1(L_q, M)}{\mathrm{H}^1_{\mathrm{fin}}(L_q, M)}
}\]
where $q$ runs over all places of $L$.
\end{defn}
\begin{rem}
This definition of Selmer groups depends on both $A_{f,n}$ (or $T_{f,n}$) and $\Delta$. Note that $\Delta$ may not be recovered from $A_{f,n}$ (or $T_{f,n}$).
\end{rem}

\subsection{$p$-adic minimal Selmer groups and the comparison with $N^-$-ordinary Selmer groups} \label{subsec:minimal_selmer}
We define the \textbf{minimal Selmer groups of $f$ over $K(m'p^\infty)$} by the following exact sequence:
\[
\xymatrix{ 
0 \ar[r] & \mathrm{Sel}(K(m'p^\infty), A_f) \ar[r] & \mathrm{H}^1 (K(m'p^\infty), A_f) \ar[r]^-{\prod_q \mathrm{res}_q} &
\displaystyle\prod_{q \mid p}  \dfrac{\mathrm{H}^1(K(m'p^\infty)_q, A_f)}{\mathrm{H}^1_{\mathrm{ord}}(K(m'p^\infty)_q, A_f)} \times
\prod_{q \nmid p} \mathrm{H}^1(K(m'p^\infty)_q, A_f),
}\]
and we can also define the residual minimal Selmer group $\mathrm{Sel}(K(m'p^\infty), A_{f,n})$ similarly (c.f.~\cite[$\S$4.2]{epw}). See also \cite[$\S$3.1]{pw-mu}.
\begin{rem}
These minimal Selmer groups are closer to the classical Selmer groups of elliptic curves than the $N^-$-ordinary Selmer groups.
Also, these Selmer groups are more relevant to the context of anticyclotomic Iwasawa theory for modular forms than Greenberg Selmer groups as in \cite[$\S$3.1]{pw-mu} due to the existence of primes which splits completely in the infinite extension.
\end{rem}

Let $\Sigma$ be the finite set of places of $K$ containing the places dividing $Nm'p\infty$, and $K_{\Sigma}$ be the maximal extension of $K$ unramified outside $\Sigma$.
Then we have
\[
\xymatrix@R=1em{ 
0 \ar[r] & \mathrm{Sel}(K(m'p^\infty), A_f)[\varpi^n] \ar[r] & \mathrm{H}^1(K_{\Sigma}/K(m'p^\infty), A_{f})[\varpi^n] \ar@{=}[d]^-{\textrm{Lemma } \ref{lem:control_mod_pn}} \\
0 \ar[r] & \mathrm{Sel}_{N^-}(K(m'p^\infty), A_{f,n})  \ar[r] &  \mathrm{H}^1(K_{\Sigma}/K(m'p^\infty), A_{f,n})
}\]
since the residual representation is irreducible. Moreover, the local conditions at primes dividing $N^-$ of minimal Selmer groups and $N^-$-ordinary Selmer groups coincide since such primes split completely in $K(m'p^\infty)/K$.
Thus, we have the following sequence
\[
\xymatrix{ 
0 \ar[r] & \mathrm{Sel}_{N^-}(K(m'p^\infty), A_{f,n})  \ar[r] & \mathrm{Sel}(K(m'p^\infty), A_f)[\varpi^n] \ar[r] & \displaystyle\prod_{w} A^{ G_{K(m'p^\infty)_w} }_f / \varpi^n A^{ G_{K(m'p^\infty)_w} }_f
}
\]
where $w$ runs over primes of $K(m'p^\infty)$ dividing $N^+$.
Also, since we have the exact sequence
\[
\xymatrix{ 
0 \ar[r] & \mathrm{Sel}(K(m'p^\infty), A_{f,n})  \ar[r] & \mathrm{Sel}(K(m'p^\infty), A_f)[\varpi^n] \ar[r] & \displaystyle\prod_{w} A^{ G_{K(m'p^\infty)_w} }_f / \varpi^n A^{ G_{K(m'p^\infty)_w} }_f
}
\]
where $w$ runs over primes of $K(m'p^\infty)$ dividing $N$ (\cite[Corollary 4.2.5]{epw}),
the inclusion
\begin{equation} \label{eqn:selmer_inclusion}
\mathrm{Sel}(K(m'p^\infty), A_{f,n}) \subseteq \mathrm{Sel}_{N^-}(K(m'p^\infty), A_{f,n})
\subseteq \mathrm{Sel}(K(m'p^\infty), A_f)[\varpi^n]
\end{equation}
holds and the latter inclusion is of finite index. The index is bounded and independent of $n$.
If $f$ is $N^+$-minimal, then
$$\mathrm{Sel}(K(m'p^\infty), A_{f,n}) \subseteq \mathrm{Sel}_{N^-}(K(m'p^\infty), A_{f,n})
= \mathrm{Sel}(K(m'p^\infty), A_f)[\varpi^n].$$
If $f$ is $N$-minimal, then
$$\mathrm{Sel}(K(m'p^\infty), A_{f,n}) = \mathrm{Sel}_{N^-}(K(m'p^\infty), A_{f,n})
= \mathrm{Sel}(K(m'p^\infty), A_f)[\varpi^n].$$
See \cite[Lemma 4.1.2]{epw} for detail.
\begin{prop}[Comparison between Selmer groups I]
If $f$ is $N^+$-minimal, then
$$\mathrm{Sel}_{N^-}(K(m'p^\infty), A_{f}) := \displaystyle\varinjlim_n \mathrm{Sel}_{N^-}(K(m'p^\infty), A_{f,n}) = \mathrm{Sel}(K(m'p^\infty), A_{f}) .$$
\end{prop}
This comparison is enough for our purpose. However, we also prove the comparison statement without assuming $N^+$-minimality with $m'=1$. In other words, we give a detailed proof of \cite[Proposition 3.6]{pw-mu} and \cite[Proposition1.3]{chida-hsieh-main-conj}.
\begin{prop} \label{prop:cotorsion_vanishing_mu}
If $\mathrm{Sel}_{N^-}(K(p^\infty), A_{f})$ is $\Lambda$-cotorsion with vanishing of $\mu$-invariant, then
$\mathrm{Sel}(K(p^\infty), A_{f})$ is also $\Lambda$-cotorsion with vanishing of $\mu$-invariant.
\end{prop}
\begin{proof}
We have 
$$\mathrm{Sel}_{N^-}(K(p^\infty), A_{f})[\varpi] = \mathrm{Sel}_{N^-}(K(p^\infty), A_{f,1})$$ since the $N^-$-ordinary Selmer groups are defined as the injective limit of the residual Selmer groups.
By assumption, $\mathrm{Sel}_{N^-}(K(p^\infty), A_{f,1})$ is finite as in \cite[Proof of Corollary 2.3]{kim-pollack-weston}.
Since the inclusion
$$\mathrm{Sel}_{N^-}(K(p^\infty), A_{f,1})
\subseteq \mathrm{Sel}(K(p^\infty), A_f)[\varpi]$$ is of finite index (as in Inclusion (\ref{eqn:selmer_inclusion})),
$\mathrm{Sel}(K(p^\infty), A_f)[\varpi]$ is also finite. Thus, we obtain the conclusion.
\end{proof}
We still have the same conclusion without assuming the $N^+$-minimality of $f$.
\begin{cor}[Comparison between Selmer groups II] 
Assume the same conditions of Corollary \ref{cor:main_conj_nonminimal}. Then
$$\mathrm{Sel}_{N^-}(K(p^\infty), A_{f}) = \mathrm{Sel}(K(p^\infty), A_{f}) .$$
\end{cor}
\begin{proof}
By Corollary \ref{cor:main_conj_nonminimal}, $\mathrm{Sel}_{N^-}(K(p^\infty), A_{f})$ is $\Lambda_{p^\infty}$-cotorsion with vanishing of $\mu$-invariant. By Proposition \ref{prop:cotorsion_vanishing_mu}, $\mathrm{Sel}(K(p^\infty), A_{f})$ is also $\Lambda_{p^\infty}$-cotorsion. 
We want to apply the work of Hachimori-Matsuno \cite[Corollary]{hachimori-matsuno} (Proposition \ref{prop:greenberg-hachimori-matsuno}) to minimal Selmer groups.
Since the residual representation is irreducible and there exists a relevant Cassels-Tate pairing for minimal Selmer groups as in \cite{flach-cassels-tate}, it suffices to check the control theorem works for minimal Selmer groups. 
Since we have the control theorem for $N^-$-ordinary Selmer groups (Lemma \ref{lem:control_galois_selmer})
and the index between $N^-$-ordinary Selmer groups and minimal Selmer groups is bounded independently on $n$, the control theorem for minimal Selmer groups holds. Thus, we obtain the conclusion.
\end{proof}

\begin{rem} \label{rem:comparision_of_selmer}
It is not difficult to see that the minimal Selmer groups and the $N^-$-ordinary Selmer groups coincide under the $N$-minimality assumption. For the comparison over $K(p^\infty)$ between minimal Selmer groups and $N^-$-ordinary Selmer groups under Condition CR, see \cite[Proposition 3.6]{pw-mu} and \cite[Proposition 1.3]{chida-hsieh-main-conj}. In the comparison argument over $K(p^\infty)$, we first take the direct limit $r \to \infty$ in $K(p^r)$ then take the direct limit $n \to \infty$ in $A_{f,n}$. Thus, $\mathrm{Sel}_{N^-}(K(p^r), A_{f})$ does not appear in the comparison over $K(p^\infty)$.
Also, the comparison over $K(p^\infty)$ uses the theory of $\Lambda$-modules.
In order to compare them over $K(m)$ without any minimality condition, it seems that it involves certain Tamagawa exponents over $K(m)$ at $\ell$ dividing $N^+$. See \cite[Definition 3.3]{pw-mu}, \cite[Corollary 2 and Corollary 6.15]{chida-hsieh-main-conj}, and \cite[$\S$1 and $\S$4.2]{lundell-level-lowering} for hints.
\end{rem}

\subsection{Local computation, local duality, and global reciprocity}
We recall some computation and duality results of local cohomology \emph{over $K(m)$}. All the proofs over $K(p^r)$ and $K(p^\infty)$ in \cite{bertolini-darmon-imc-2005} and \cite{chida-hsieh-main-conj} work for $K(m)$ in the exactly same way.
\subsubsection{Local computation}
\begin{prop}[Local computation] \label{prop:local_computation} $ $
\begin{enumerate}
\item Let $\ell \neq p$ and be split in $K/\mathbb{Q}$. Let $m = m'p^r$ with $(m',p)=1$. For $r >>0$, we have
\[
\xymatrix{
\mathrm{H}^1_{ \mathrm{sing} } (K(m)_{\ell}, T_{f,n})  = 0 , &
\mathrm{H}^1_{ \mathrm{fin} }(K(m)_{\ell}, A_{f,n})  = 0 .
}
\]
\item Let $\ell \neq p$ and be non-split in $K/\mathbb{Q}$.
\begin{align*}
\mathrm{H}^1_{ \mathrm{sing} } (K(m)_{\ell}, T_{f,n}) & \simeq \mathrm{H}^1_{ \mathrm{sing} } (K_{\ell}, T_{f,n}) \otimes_{\mathcal{O}_n} \mathcal{O}_n [ \Gamma_{m} ]  \\
\mathrm{H}^1_{ \mathrm{fin} }(K(m)_{\ell}, A_{f,n}) & \simeq \mathrm{Hom}_{\mathcal{O}_n} \left( \mathrm{H}^1_{ \mathrm{sing} }(K_{\ell}, T_{f,n}) \otimes_{\mathcal{O}_n} \mathcal{O}_n [ \Gamma_{m} ] , E/\mathcal{O} \right)
\end{align*}
\item Let $\ell$ be an $n$-admissible prime. 
\[
\xymatrix{
\mathrm{H}^1_{ \mathrm{sing} } (K_{\ell}, T_{f,n}) \simeq \mathcal{O}_n , &
\mathrm{H}^1_{ \mathrm{fin} } (K_{\ell}, A_{f,n})  \simeq \mathcal{O}_n .
}
\]
\item Let $\ell$ be an $n$-admissible prime. 
Then
$\mathrm{H}^1_{ \mathrm{sing} } (K(m)_{\ell}, T_{f,n})$ and 
$\mathrm{H}^1_{ \mathrm{fin} } (K(m)_{\ell}, T_{f,n})$
are free of rank one over $\mathcal{O}_n [ \Gamma_{m} ]$.
\item Let $\ell$ be an $n$-admissible prime. 
We have the decomposition
\begin{align*}
\mathrm{H}^1(K(m)_{\ell}, T_{f,n}) & = \mathrm{H}^1_{ \mathrm{fin} }(K(m)_{\ell}, T_{f,n}) \oplus \mathrm{H}^1_{ \mathrm{ord} }(K(m)_{\ell}, T_{f,n}) \\
\mathrm{H}^1(K(m)_{\ell}, A_{f,n}) & = \mathrm{H}^1_{ \mathrm{fin} }(K(m)_{\ell}, A_{f,n}) \oplus \mathrm{H}^1_{ \mathrm{ord} }(K(m)_{\ell}, A_{f,n}) .
\end{align*}
\end{enumerate}
\end{prop}
\begin{proof} Note that all the references below work for finite layers.
\begin{enumerate}
\item See \cite[Lemma 2.4]{bertolini-darmon-imc-2005} and note that the same argument works for $\ell$ dividing $N$ as in \cite[Lemma 1.4.(1)]{chida-hsieh-main-conj}.
\item See \cite[Lemma 2.5]{bertolini-darmon-imc-2005} and note that the same argument works for $\ell$ dividing $N$ as in \cite[Lemma 1.4.(2)]{chida-hsieh-main-conj}.
\item See \cite[Lemma 2.6]{bertolini-darmon-imc-2005} and \cite[Lemma 1.5.(1)]{chida-hsieh-main-conj}.
\item See \cite[Lemma 2.7]{bertolini-darmon-imc-2005} and \cite[Lemma 1.5.(2)]{chida-hsieh-main-conj}.
\item See \cite[$\S$4.1]{bertolini-darmon-imc-2005} and \cite[Lemma 1.5.(3)]{chida-hsieh-main-conj}.
\end{enumerate}
\end{proof}

\begin{defn}[Localized projections at $n$-admissible primes] \label{defn:localization_n-admissible}
For $\kappa_m \in \mathrm{H}^1(K(m), T_{f,n})$, 
$v_\ell(\kappa_m)$ is defined to be the projection of the restriction of $\kappa_m$ at $\ell$ to $\mathrm{H}^1_{ \mathrm{fin} }(K(m)_{\ell}, T_{f,n})$
according to the decomposition in Proposition \ref{prop:local_computation}.(5).
\end{defn}
\subsubsection{Local duality}
We recall the Tate local duality at $\ell$ over $K(m)$, which is a \emph{perfect} pairing
\[
\xymatrix{
\langle -,- \rangle_{\ell}: \mathrm{H}^1(K(m)_{\ell}, T_{f,n}) \times \mathrm{H}^1(K(m)_{\ell}, A_{f,n}) \ar[r] & E/\mathcal{O} .
}
\]
The pairing is compatible with action of $\mathcal{O}_n [ \Gamma_m ]$, so we have an isomorphism
$\mathrm{H}^1_{ \mathrm{fin} }(K(m)_{\ell}, T_{f,n}) \simeq \mathrm{H}^1_{ \mathrm{fin} }(K(m)_{\ell}, A_{f,n})^\vee$
where $M^\vee = \mathrm{Hom}_{\mathcal{O}}(M, E/\mathcal{O})$ as $\mathcal{O}_n [ \Gamma_m ]$-modules.
\begin{prop}[Local dualities] \label{prop:local_dualities} $ $
\begin{enumerate} 
\item Let $\ell \neq p$. Then 
$\mathrm{H}^1_{ \mathrm{fin} }(K(m)_{\ell}, T_{f,n})$ and $\mathrm{H}^1_{ \mathrm{fin} }(K(m)_{\ell}, A_{f,n})$ are orthogonal complements under the pairing $\langle -,- \rangle_{\ell}$. In particular, 
$\mathrm{H}^1_{ \mathrm{sing} }(K(m)_{\ell}, T_{f,n})$ and $\mathrm{H}^1_{ \mathrm{fin} }(K(m)_{\ell}, A_{f,n})$ are the Pontryagin dual to each other.
\item Let $\ell$ be an $n$-admissible or $\ell$ divide $pN^-$. Assume Condition CR (for $\ell \mid N^-$) and Condition PO (for $\ell =p$). Then
$\mathrm{H}^1_{ \mathrm{ord} }(K(m)_{\ell}, T_{f,n})$ and $\mathrm{H}^1_{ \mathrm{ord} }(K(m)_{\ell}, A_{f,n})$ are orthogonal complements under the pairing $\langle -,- \rangle_{\ell}$.
\item Let $\ell$ be an $n$-admissible.
Then $\mathrm{H}^1_{ \mathrm{fin} }(K(m)_{\ell}, T_{f,n})$ and $\mathrm{H}^1_{ \mathrm{ord} }(K(m)_{\ell}, A_{f,n})$ are the Pontryagin dual to each other.
\end{enumerate}
\end{prop}
\begin{proof} 
See \cite[Proposition 1.6, Proposition 1.7, and Lemma 1.8]{chida-hsieh-main-conj} for proof with help of Lemma \ref{lem:local_condition_at_p}. Note that all the proofs work for $K(m)$.
\end{proof}

\begin{lem}[{\cite[Lemma 1.8]{chida-hsieh-main-conj}}] \label{lem:local_condition_at_p}
Assume Condition PO (Assumption \ref{assu:conditionPO}).
Then
$$\mathrm{H}^0(K(m)_p, A_{f,n}/F^+A_{f,n}) = 0.$$
\end{lem}
\begin{proof}
Let $F^- A_{f,n} := A_{f,n}/F^+A_{f,n}$.
Let $\wp$ be a place of $K(m)$ lying above $p$.
Then the Frobenius $\mathrm{Frob}_\wp \in \mathrm{Gal}(K(m)/K)$ acts on $F^- A_{f,n}$ as $\alpha_{p}(f)^{r_p}$ where $\alpha_p(f)$ is the unit $U_p$-eigenvalue and $r_p = 1$ or 2 if $p$ splits or is inert in $K$, respectively.
Let $m''$ be the inertia degree of $\wp$ in $K(m)/K$. If $p$ is inert in $K$, then $m''= 1$.
Then
\begin{align*}
\left( F^- A_{f,n} \right)^{G_{K(m)_\wp}} & = \left( F^- A_{f,n} \right) [\left(\mathrm{Frob}_\wp\right)^{r_p m''} - 1 ] & \textrm{ unramifiedness at $\wp$} \\
& =  \left( F^- A_{f,n} \right) [ \alpha_p(f)^{r_p m''} - 1 ] & \textrm{ property of } \rho_f\\
& = \lbrace 0 \rbrace . & \textrm{ Condition PO }
\end{align*}
\end{proof}
\subsubsection{Global reciprocity}
Let $\kappa_m \in \mathrm{H}^1(K(m), T_{f,n})$ and
$s_m \in \mathrm{H}^1(K(m), A_{f,n})$.
Via global class field theory, we have
$$\sum_{q} \left\langle \mathrm{res}_q (\kappa_m) , \mathrm{res}_q(s_m)  \right\rangle_{q} = 0$$
where $q$ runs over all finite places of $\mathbb{Q}$.
\begin{prop}[Global reciprocity] \label{prop:global_reciprocity}
Let $\kappa_m \in \mathrm{Sel}^S_{\Delta}(K(m), T_{f,n})$ and 
$s_m \in \mathrm{Sel}_{\Delta}(K(m), A_{f,n}) $.
Then 
$$\sum_{q \mid S} \left\langle \partial_q(\kappa_m), v_q(s_m) \right\rangle_{q} = 0.$$
\end{prop}
\begin{proof}
The local conditions of $\mathrm{Sel}^S_{\Delta}(K(m), T_{f,n})$ and $\mathrm{Sel}_{\Delta}(K(m), A_{f,n})$ are orthogonal at the primes not dividing $S$.
\end{proof}

\section{The generalized divisibility criterion} \label{sec:generalized_divisibility_criterion}
We recall the original divisibility criterion.
\begin{prop}[{\cite[Proposition 3.1]{bertolini-darmon-imc-2005}}] \label{prop:original_divisibility}
Let $X$ be a finitely generated $\Lambda$-module and $L \in \Lambda$. Suppose that $\varphi(L)$ belongs to $\mathrm{Fitt}_{\mathcal{O}_\varphi}(X \otimes_{\varphi} \mathcal{O}_\varphi)$ for any homomorphism $\varphi : \Lambda \to \mathcal{O}_\varphi$ as in $\S$\ref{subsec:convention_characters} and $\mathrm{Fitt}_{\mathcal{O}_\varphi}(X \otimes_{\varphi} \mathcal{O}_\varphi)$ is the Fitting ideal of $\mathcal{O}_\varphi$-module of $X \otimes_{\varphi} \mathcal{O}_\varphi$. Then $L$ belongs to the characteristic ideal of $X$.
\end{prop}
This section is devoted to prove the following proposition.
\begin{prop}[The generalized divisibility criterion] \label{prop:generalized_divisibility}
Suppose that
$$\varphi(L) \in \varphi \left(\mathrm{Fitt}_{\Lambda_{m'p^\infty}} \left( \mathrm{Sel}_{N^-}(K(m'p^\infty), A_f)^\vee\right) \right) $$
in $\mathcal{O}_\varphi$
for any character $\varphi : \Lambda_{m'p^\infty} \to \mathcal{O}_\varphi$.
Then
$$L \in \mathrm{Fitt}_{\Lambda_{m'p^\infty}} \left( \mathrm{Sel}_{N^-}(K(m'p^\infty), A_f)^\vee\right). $$
\end{prop}
The decomposition in $\S$\ref{subsec:enhanced_decomposition} implies that the assumption for a character $\varphi$ is equivalent to the statement
$$\overline{\varphi} (L_\omega) \in \overline{\varphi} \left(\mathrm{Fitt}_{\Lambda_{m'p^\infty}} \left( \mathrm{Sel}_{N^-}(K(m'p^\infty), A_f)^\vee\right) \otimes_\omega \mathcal{O}_\omega \right) $$
holds for all character $\overline{\varphi} : \mathcal{O}_\omega\llbracket\Gamma_{p^\infty}\rrbracket \to \mathcal{O}_\varphi$ with $\varphi = \omega \times \overline{\varphi}$.
Since Fitting ideals behave well under quotients and $\mathcal{O}_\omega$ is projective (so flat) over $\mathcal{O}[\Gamma_{m'}]$, we have
\begin{align} 
& \mathrm{Fitt}_{\Lambda_{m'p^\infty}} \left( \mathrm{Sel}_{N^-}(K(m'p^\infty), A_f)^\vee\right) \otimes_\omega \mathcal{O}_\omega  \nonumber\\
 = & \mathrm{Fitt}_{\Lambda_{m'p^\infty}} \left( \mathrm{Sel}_{N^-}(K(m'p^\infty), A_f)^\vee\right) \mathcal{O}_\omega & \textrm{(flatness)}  \label{eqn:fitting_quotient}   \\
 = & \mathrm{Fitt}_{\mathcal{O}_\omega\llbracket\Gamma_{p^\infty}\rrbracket} \left( \mathrm{Sel}_{N^-}(K(m'p^\infty), A_f)^\vee \otimes_\omega \mathcal{O}_\omega \right) .  & \textrm{(Lemma \ref{lem:fitting-ideals-base-change}.(1))} \nonumber
\end{align}
Thus, the original divisibility criterion over $\mathcal{O}_\omega\llbracket\Gamma_{p^\infty}\rrbracket$ (Proposition \ref{prop:original_divisibility}) implies that
\begin{equation} \label{eqn:divisibility_char_compoenents}
L_\omega \in \mathrm{char}_{\mathcal{O}_\omega\llbracket\Gamma_{p^\infty}\rrbracket} \left( \mathrm{Sel}_{N^-}(K(m'p^\infty), A_f)^\vee \otimes_\omega \mathcal{O}_\omega \right) .
\end{equation}
In order to compare the chracteristic ideal and the Fitting ideal, we need to prove the non-existence of non-trivial finite submodule of the dual Selmer modules over $\mathcal{O}_\omega\llbracket\Gamma_{p^\infty}\rrbracket$ for each $\omega$.
We first recall the result over $\mathcal{O}\llbracket\Gamma_{p^\infty}\rrbracket$.
\begin{prop}[Greenberg{\cite[Proposition 4.14]{greenberg-lnm}}, Hachimori-Matsuno {\cite[Corollary]{hachimori-matsuno}}] \label{prop:greenberg-hachimori-matsuno}
Suppose that
\begin{enumerate}
\item $\mathrm{Sel}_{N^-}(K(m'p^\infty), A_f)^\vee$ is $\mathcal{O}\llbracket\Gamma_{p^\infty}\rrbracket$-torsion, and
\item $\mathrm{H}^0(K(m'), A_{f,1}) = 0$.
\end{enumerate}
Then $\mathrm{Sel}_{N^-}(K(m'p^\infty), A_f)^\vee$ has no non-trivial finite $\mathcal{O}\llbracket\Gamma_{p^\infty}\rrbracket$-submodule.
\end{prop}
\begin{rem}
The statement in \cite{hachimori-matsuno} uses the classical Selmer groups of elliptic curves.
However, the argument in \cite{hachimori-matsuno} works for any Selmer group satisfying the control theorem over $\Lambda_{p^\infty}$ if there exists a relevant Cassels-Tate pairing. Such a pairing is given in \cite{flach-cassels-tate}.
Thus, it does not harm the validity of Proposition \ref{prop:greenberg-hachimori-matsuno} although $N^-$-ordinary Selmer groups are used.
\end{rem}
\begin{cor} \label{cor:greenberg-hachimori-matsuno}
The module $\mathrm{Sel}_{N^-}(K(m'p^\infty), A_f)^\vee$ has no non-trivial finite $\mathcal{O}\llbracket\Gamma_{p^\infty}\rrbracket$-submodule.
\end{cor}
\begin{proof}
The $(\omega = \mathbf{1})$-component of $L_p(K(m'p^\infty), f)$ is non-zero due to \cite[Theorems A and C]{hung-nonvanishing}. Thus, the characteristic ideal of the dual Selmer group over $\mathcal{O}\llbracket \Gamma_{p^\infty} \rrbracket$ is non-zero due to inclusion (\ref{eqn:divisibility_char_compoenents}).
Also, the residual representation is absolutely irreducible and $K(m')/\mathbb{Q}$ is Galois.
\end{proof}
Indeed, this corollary correspond to the $\omega = \mathbf{1}$-component. In general, we expect the following statement.
In order to have Corollary \ref{cor:nontrivial_finite_omega}, it suffice to show the following statement.
\begin{lem} \label{lem:NTFSmp}
The module $\mathrm{Sel}(K(m'p^\infty), A_f)^\vee$ has no non-trivial finite $\mathcal{O}\llbracket\Gamma_{m'p^\infty}\rrbracket$-submodule.
\end{lem}
\begin{proof}
Let $C$ be a non-trivial finite $\mathcal{O}\llbracket\Gamma_{m'p^\infty}\rrbracket$-submodule $\mathrm{Sel}(K(m'p^\infty), A_f)^\vee$.
Let $c \in C$ be a non-zero element.
Then the submodule of $C$ generated by $c$ over $\mathcal{O}\llbracket\Gamma_{p^\infty}\rrbracket$ is the trivial module due to Corollary \ref{cor:greenberg-hachimori-matsuno}.
Thus, the submodule of $C$ generated by $c$ over $\mathcal{O}[\Gamma_{m'}]$ must be non-trivial.
In other words, 
$$\mathcal{O}\llbracket\Gamma_{m'p^\infty}\rrbracket c = \mathcal{O}[\Gamma_{m'}]c$$
since $\mathcal{O}\llbracket\Gamma_{p^\infty}\rrbracket c = 0$.
By decomposition $\mathcal{O}[\Gamma_{m'}]c = \oplus_\omega \mathcal{O}_\omega c_\omega$, one of $\mathcal{O}_\omega c_\omega$ should be non-trivial. ($c_{\mathbf{1}}$ is trivial though.)
However, $\mathcal{O}c$ is trivial since $\mathcal{O} \subset \Lambda_{p^\infty}$, so $\mathcal{O}c_\omega$ is trivial.
Considering the natural surjective quotient map $$\mathcal{O}[\Gamma_{m'}] \otimes_{\mathcal{O}} \mathcal{O}c_\omega \to \mathcal{O}_\omega c_\omega,$$
we get contradiction.
\end{proof}

\begin{cor} \label{cor:nontrivial_finite_omega}
For each $\omega$, the module $\mathrm{Sel}(K(m'p^\infty), A_f)^\vee \otimes_\omega \mathcal{O}_\omega$ has no nontrivial finite $\mathcal{O}_\omega\llbracket\Gamma_{p^\infty}\rrbracket$-submodule.
\end{cor}

Due to Corollary \ref{cor:nontrivial_finite_omega} and \cite[Proposition 3]{nuccio-fitting}, we have
$$L_\omega \in \mathrm{Fitt}_{\mathcal{O}_\omega\llbracket\Gamma_{p^\infty}\rrbracket} \left( \mathrm{Sel}(K(m'p^\infty), A_f)^\vee \otimes_\omega \mathcal{O}_\omega \right) .$$
Considering the equality (\ref{eqn:fitting_quotient}), we have
$$L \in \mathrm{Fitt}_{\mathcal{O}\llbracket\Gamma_{m'p^\infty}\rrbracket} \left( \mathrm{Sel}(K(m'p^\infty), A_f)^\vee \right) ,$$
so we prove Proposition \ref{prop:generalized_divisibility}.

\begin{rem}
Note that the the finite layer analogue of Proposition \ref{prop:original_divisibility} does not hold in general. The difficulty of the finite layer analogue is observed in \cite{joongul_character_values}.
\end{rem}

\section{Reduction of Euler system divisibilities to Mazur-Tate conjectures} \label{sec:reduction_to_main_theorem}
In this section, we show that the generalized Euler system divisibility (Theorem \ref{thm:main_theorem_1}) and the exact control theorem (Lemma \ref{lem:control_galois_selmer} with $S=1$) imply Theorem \ref{thm:main_theorem} with the $p$-stabilized form $f_\alpha$.

\begin{prop} \label{prop:mainconj_to_mazurtate}
Under the assumption of Theorem \ref{thm:main_theorem_1}, we have
$$L_p(K(m'p^r), f_\alpha) \in \mathrm{Fitt}_{\mathcal{O}[ \mathrm{Gal}(K(m'p^r)/K)  ]} \left( \mathrm{Sel}_{N^-}(K(m'p^r), A_{f})^\vee \right) .$$
\end{prop}
\begin{proof}
Write $X_{m'p^\infty} = \mathrm{Sel}_{N^-}(K(m'p^\infty), A_{f})^\vee$ and
$X_{m'p^r} = \mathrm{Sel}_{N^-}(K(m'p^r), A_{f})^\vee$ where $(-)^\vee$ is the Pontryagin dual.
Let $L_p(K(m'p^\infty), f_\alpha) = \left( L_p(K(m'p^r), f_\alpha) \right)_r \in \Lambda_{m'p^\infty}$ be the generalized $p$-adic $L$-function defined in $\S$\ref{sec:theta_elements}.
Theorem  \ref{thm:main_theorem_1} implies that
$$ L_p(K(m'p^\infty), f_\alpha)  \in \mathrm{Fitt}_{\Lambda_{m'p^\infty}} \left( X_{m'p^\infty} \right) .$$
By the exact control theorem (Lemma \ref{lem:control_galois_selmer} with $S=1$), we have
$$\mathrm{Sel}_{N^-}(K(m'p^\infty), A_{f})^{\left(\Gamma_{m'p^\infty}\right)^{p^r}} \simeq \mathrm{Sel}_{N^-}(K(m'p^r), A_{f}),$$
and the corresponding dual statement is
$$\left( X_{m'p^\infty} \right)_{\left(\Gamma_{m'p^\infty}\right)^{p^r}} \simeq X_{m'p^r}.$$
Let $\gamma$ be a topological generator of $\Gamma_{m'p^\infty}$ and $I_r$ be the ideal of $\Lambda_{m'p^\infty}$ generated by $\gamma^{p^r} - 1$.
By Lemma \ref{lem:fitting-ideals-base-change}.(2) and the above isomorphism, we have
$$\mathrm{Fitt}_{\Lambda_{m'p^\infty}} \left( X_{m'p^\infty} \right) \Mod{I_r} = \mathrm{Fitt}_{\mathcal{O}[\Gamma_{m'p^r}]} \left( X_{m'p^r} \right) .$$
Thus, if we take mod $I_r$ reduction of the inclusion
$$L_p(K(m'p^\infty), f_\alpha) = \left( L_p(K(m'p^r), f_\alpha) \right)_r \in \mathrm{Fitt}_{\Lambda_{m'p^\infty}} \left( X_{m'p^\infty} \right) ,$$
then we obtain $$L_p(K(m'p^r), f_\alpha) \in  \mathrm{Fitt}_{\mathcal{O}[\Gamma_{m'p^r}]} \left( X_{m'p^r} \right).$$
\end{proof}
%
%

\section{$p$-stabilization and tame exceptional zeroes} \label{sec:tame_exceptional_zeros}
In this section, we deduce the non-$p$-stabilized version of the Mazur-Tate conjecture from the $p$-stabilized version under a condition on the $p$-adic multiplier (Assumption \ref{assu:conditionPO}).
\subsection{$p$-adic multipliers}
\subsubsection{$p$ splits in $K$}
Suppose that $p$ splits into $\mathfrak{p} \overline{\mathfrak{p}}$ in $K$.
Then we define the $p$-adic multiplier $e_p(f, K)$ by
$$e_p(f, K) := (1 - \frac{1}{\alpha} \mathrm{Fr}_{\mathfrak{p}})(1 - \frac{1}{\alpha} \mathrm{Fr}_{\overline{\mathfrak{p}}}) \in \mathcal{O}[\Gamma_{m'}] $$
where $\mathrm{Fr}_{\mathfrak{p}}$, $\mathrm{Fr}_{\overline{\mathfrak{p}}}$ are the geometric Frobenii in $\Gamma_{m'} = \mathrm{Gal}(K(m')/K)$.
Then Condition PO (Assumption \ref{assu:conditionPO} for splitting $p$) directly implies that
$e_p(f, K) \in \mathcal{O}[\Gamma_{m'}]^\times$.
%
\subsubsection{$p$ is inert in $K$}
Suppose that $p$ is inert in $K$.
Then 
$$e_p(f, K) :=  (1 - \frac{1}{\alpha^2} )$$ and it is also invertible in $\mathcal{O}$ due to Condition PO (Assumption \ref{assu:conditionPO} for inert $p$).
\subsection{Reduction via induction}
We prove the following statement.
\begin{prop}
Under Condition PO, $\theta_{m'p^r}(f)$ is a multiple of $\theta_{m'p^r}(f_\alpha)$ by an element of $\mathcal{O}[\Gamma_{p^r}]$, i.e.
$$\theta_{m'p^r}(f) = c \cdot \theta_{m'p^r}(f_\alpha)$$
where $c \in \mathcal{O}[\Gamma_{m'p^r}]$.
\end{prop}
\begin{proof}
Using the norm compatibility of the $p$-stabilized theta elements, it follows from induction on $r$. The first step of induction requires $e_p(f,K) \in \mathcal{O}[\Gamma_{m'}]^\times$, which follows from Condition PO.
\end{proof}
\begin{cor} \label{cor:reduction_p-stabilization}
Let $\mathfrak{a}$ be an ideal of $\mathcal{O}[\Gamma_{m'p^r}]$.
If $\theta_{m'p^r}(f_\alpha) \in \mathfrak{a}$, then $\theta_{m'p^r}(f) \in \mathfrak{a}$.
\end{cor}
Even if Assumption \ref{assu:conditionPO} for splitting $p$ fails for some $\omega$, we are still able to obtain the following partial result.
\begin{cor}
Assume that $e_p(f,K)_\omega$ is invertible in $\mathcal{O}_\omega[\Gamma_{p^r}]$. If $$L(K(m'p^\infty), f_\alpha)_\omega \in \mathrm{Fitt}_{\mathcal{O}_\omega\llbracket \Gamma_{p^\infty} \rrbracket} (\mathrm{Sel}(K(m'p^\infty), A_f)^\vee  \otimes_\omega \mathcal{O}_\omega ),$$
then
$$L(K(m'p^\infty), f)_\omega \in \mathrm{Fitt}_{\mathcal{O}_\omega\llbracket \Gamma_{p^\infty} \rrbracket} (\mathrm{Sel}(K(m'p^\infty), A_f)^\vee \otimes_\omega \mathcal{O}_\omega ) .$$
\end{cor}

\section{Shimura curves and CM points} \label{sec:shimura_curves}
We quickly review Shimura curves and CM points on them. See \cite[$\S$3]{chida-hsieh-main-conj} and \cite[$\S$1]{nekovar-hilbert} for detail. 
\subsection{Shimura curves and complex uniformization} \label{subsec:complex_uniformization}
Let $B^\Delta$ be the definite quaternion algebra over $\mathbb{Q}$ of discriminant $\Delta = \Delta_B$.
Let $\ell$ be a prime which is inert in $K$ and is prime to $\Delta_B$.
Let $\mathcal{B}^{\Delta\ell}$ be the \emph{indefinite} quaternion algebra over $\mathbb{Q}$ of discriminant $\Delta \ell$.
Fix a $\mathbb{Q}$-embedding $t^{\Delta\ell} : K \hookrightarrow \mathcal{B}^{\Delta\ell}$ and an isomorphism
$\varphi_{B^{\Delta}, \mathcal{B}^{\Delta\ell}}: \widehat{B^{\Delta}}^{(\ell)} \simeq \widehat{\mathcal{B}^{\Delta\ell}}^{(\ell)}$.
Let
\begin{itemize}
\item $U = U_{N^+}$ be the level structure of $B^\Delta$ corresponding to the unit group of an adelic Eichler order $\widehat{R}^\times_{N^+} = \widehat{R}^\times$ of level $\Gamma_0(N^+)$, and
\item $U_0(\ell) = U_{N^+\ell}$ be the level structure of $B^\Delta$ corresponding to the unit group of an adelic Eichler order $\widehat{R}^\times_{N^+\ell} = \widehat{R_0(\ell)}^\times$ (contained in $\widehat{R}^\times_{N^+}$) of level $\Gamma_0(N^+\ell)$.
\end{itemize}
Let $\mathcal{U} := \mathcal{U}_{N^+}  = \varphi_{B^{\Delta}, \mathcal{B}^{\Delta\ell}} \left( U^{(\ell)}  \right) \mathcal{O}^\times_{ \mathcal{B}^{\Delta\ell} }$ be the $\Gamma_0(N^+)$-level structure of $\mathcal{B}^{\Delta\ell}$. We can also similarly define the $\Gamma_{0}(N^+) \cap \Gamma_{1}(p)$-level structure $\mathcal{U}_{N+,p}$.
 
Let $M^{\Delta\ell}_\mathcal{U} = M^{\Delta\ell}(N^+)$ be the Shimura curve over $\mathbb{Q}$ attached to $\mathcal{B}^{\Delta\ell}$ of level $\mathcal{U}$.
The complex uniformization of $M^{\Delta\ell}_\mathcal{U}$ is given by the double coset
$$M^{\Delta\ell}_\mathcal{U}(\mathbb{C}) = \mathcal{B}^{\Delta\ell, \times} \backslash \left( \left( \mathbb{C} \setminus \mathbb{R}\right) \times \widehat{\mathcal{B}^{\Delta\ell}}^{\times} \right) / \mathcal{U}.$$

For $z \in \mathbb{C} \setminus \mathbb{R}$ and $b \in \widehat{\mathcal{B}^{\Delta\ell}}^{\times}$,
denote by $[z, b]_{\mathcal{U}}$ the point of  $M^{\Delta\ell}_\mathcal{U}(\mathbb{C})$ represented by $(z, b)$.

\subsection{Bad reduction of Shimura curves}
\subsubsection{The dual graph of the fiber of the $\ell$-adic upper half plane and the oriented Bruhat-Tits tree}
Let $\mathfrak{h}_\ell$ be the $\ell$-adic upper half plane as a rigid analytic variety over $\mathbb{Q}_\ell$, and $\widehat{\mathfrak{h}}_\ell$ be a natural formal model of $\mathfrak{h}_\ell$.
Let $\mathscr{T}_\ell = \mathcal{V}(\mathscr{T}_\ell) \coprod \mathcal{E}(\mathscr{T}_\ell)$ be the dual graph of the special fiber of $\widehat{\mathfrak{h}}_\ell$, where
$\mathcal{V}(\mathscr{T}_\ell)$ and $\mathcal{E}(\mathscr{T}_\ell)$ are the set of vertices and edges of $\mathscr{T}_\ell$, respectively. Then $\mathscr{T}_\ell$ is the Bruhat-Tits tree of $B^\times_\ell / \mathbb{Q}^\times_\ell \simeq \mathrm{PGL}_2(\mathbb{Q}_p)$.
Let $\overset{\to}{\mathcal{E}}(\mathscr{T}_\ell)$ be the set of oriented edges. Then we have the identifications
$\mathcal{V}(\mathscr{T}_\ell) = B^\times_\ell / U_\ell \mathbb{Q}^\times_\ell$ and $ \overset{\to}{\mathcal{E}}(\mathscr{T}_\ell) = B^\times_\ell / U_0(\ell)_\ell \mathbb{Q}^\times_\ell$ via the orbit-stabilizer theorem.
\subsubsection{The dual graph of the fiber of the formal model of the Shimura curves and the Bruhat-Tits tree}
Let $\mathfrak{M}^{\Delta\ell}(N^+)$ be an integral model of  $M^{\Delta\ell}(N^+) \times_{\mathbb{Q}} \mathbb{Q}_\ell$ over $\mathbb{Z}_\ell$, and $\widehat{\mathfrak{M}}^{\Delta\ell}(N^+)$ be the formal completion of $\mathfrak{M}^{\Delta\ell}(N^+)$ along its special fiber. 
Let
\[
\xymatrix{
\widehat{\mathfrak{M}}^{\Delta\ell}(N^+)_{\mathcal{O}_{K_\ell}} :=\widehat{\mathfrak{M}}^{\Delta\ell}(N^+) \times_{\mathbb{Z}_\ell} \mathcal{O}_{K_\ell} , &
\widehat{\mathfrak{M}}^{\Delta\ell}(N^+)_{\mathbb{F}_{\ell^2}} := \widehat{\mathfrak{M}}^{\Delta\ell}(N^+)_{\mathcal{O}_{K_\ell}} \times_{\mathcal{O}_{K_\ell}} \mathbb{F}_{\ell^2}
}
\]
be the base change of $\widehat{\mathfrak{M}}^{\Delta\ell}(N^+)$ over $\mathcal{O}_{K_\ell}$ and its semi-stable reduction fiber over $\mathbb{F}_{\ell^2}$, respectively.
Let $\mathcal{G} = \mathcal{V}(\mathcal{G}) \coprod \mathcal{E}(\mathcal{G})$ be the dual graph of $\widehat{\mathfrak{M}}^{\Delta\ell}(N^+)_{\mathbb{F}_{\ell^2}}$.
The set $\mathcal{V}(\mathcal{G})$ of vertices of $\mathcal{G}$ consists of the irreducible components of 
$\widehat{\mathfrak{M}}^{\Delta\ell}(N^+)_{\mathbb{F}_{\ell^2}}$, and
the set $\mathcal{E}(\mathcal{G})$ of edges of $\mathcal{G}$ consists of the singular points of 
$\widehat{\mathfrak{M}}^{\Delta\ell}(N^+)_{\mathbb{F}_{\ell^2}}$.
\subsubsection{Reduction map and identifications} \label{subsubsec:reduction_identification}
We write $B$ for $B^\Delta$ for convenience.
Let $$\mathrm{red}_\ell : \mathfrak{M}^{\Delta\ell}(N^+)(\mathbb{C}_\ell) \to \mathcal{G} = \mathcal{V}(\mathcal{G}) \coprod \mathcal{E}(\mathcal{G})$$ be the reduction map modulo $\ell$. 
The $\ell$-adic uniformization \`{a} la Cerednik-Drinfeld (e.g.~\cite[$\S$5.2]{bertolini-darmon-imc-2005}, \cite[$\S$3.2]{chida-hsieh-main-conj}) induces identifications
\[
\xymatrix@R=1.2em{
\mathcal{V}(\mathcal{G}) \ar@{=}[d]\\
B^\times \backslash \left( \mathcal{V}(\mathscr{T}_\ell) \times \mathbb{Z}/2\mathbb{Z} \times \widehat{B}^{(\ell), \times} / \widehat{R}^{(\ell),\times} \right) \ar@{=}[d] \\
B^\times \backslash \left( B^\times_\ell / U_\ell \mathbb{Q}^\times_\ell \times \mathbb{Z}/2\mathbb{Z} \times \widehat{B}^{(\ell), \times} / \widehat{R}^{(\ell),\times} \right) \ar[d]^-{\simeq} & B^\times(b_\ell U_\ell, j, b^{(\ell)}U^{(\ell)}) \ar@{|->}[d] \\
B^\times \backslash \widehat{B}^{\times} / \widehat{R}^{\times} \times \mathbb{Z}/2\mathbb{Z} & ( [b_\ell b^{(\ell)} ]_U, j + \mathrm{ord}_\ell (\mathrm{nrd}(b_\ell)) )
}
\]
and
\[
\xymatrix@R=1.2em{
\mathcal{E}(\mathcal{G}) \ar[d]^-{\simeq} & \overset{\to}{\mathcal{E}}(\mathcal{G}) \ar[d]^-{\simeq}\\
B^\times \backslash \widehat{B}^{\times} / \widehat{R_0(\ell)}^{\times} \times \lbrace 0 \rbrace \ar@{^{(}->}[r] & B^\times \backslash \widehat{B}^{\times} / \widehat{R_0(\ell)}^{\times} \times \mathbb{Z}/2\mathbb{Z}
}
\]
where $\mathrm{nrd}$ is the reduced norm map.

\subsection{Bad reduction of Jacobians of Shimura curves}
\subsubsection{Hecke action}
Let $J^{\Delta\ell}(N^+)$ be the Jacobian of $M^{\Delta\ell}(N^+)$. If $L/\mathbb{Q}$ is a field extension, let $\mathrm{Div}^0 M^{\Delta\ell}(N^+)(L)$ be the group of divisors on $J^{\Delta\ell}(N^+)(L)$ of degree zero on each connected component of $M^{\Delta\ell}(N^+) \times_{\mathbb{Q}} L$.
For $D \in \mathrm{Div}^0 M^{\Delta\ell}(N^+)(L)$, write 
$cl(D) \in J^{\Delta\ell}(N^+)(L)$ for the point represented by $D$.
Let $\mathbb{T}^{\Delta\ell, (\ell)}(N^+)$ be the $\ell$-deprived Hecke algebra faithfully acting on $J^{\Delta\ell}(N^+)$ via Picard functoriality.
The isomorphism
$\varphi_{B^{\Delta}, \mathcal{B}^{\Delta\ell}}: \widehat{B^{\Delta}}^{(\ell)} \simeq \widehat{\mathcal{B}^{\Delta\ell}}^{(\ell)}$
induces an isomorphism of the corresponding $\ell$-deprived Hecke algebras
$\varphi_{B^{\Delta}, \mathcal{B}^{\Delta\ell}, *}: \mathbb{T}^{\Delta, (\ell)}(\ell N^+) \simeq \mathbb{T}^{\Delta\ell, (\ell)}(N^+)$.
We extend the map to the full Hecke algebras
$$\mathbb{T}^{\Delta}(\ell N^+) \twoheadrightarrow \mathbb{T}^{\Delta\ell}(N^+) \to \mathrm{End} ( J^{\Delta\ell}(N^+)_{\mathbb{Q}} )$$
by defining
$U_\ell \mapsto \mathcal{R}_\ell \pi_\ell \mathcal{R}_\ell$
for some $\pi_\ell \in \mathcal{B}_\ell$ with $\mathrm{nrd}(\pi_\ell) = \ell$ where $\mathcal{R}$ is an Eichler order in $\mathcal{B}^{\Delta\ell}$ corresponding to $\mathcal{U}$.

Let $\mathfrak{J}^{\Delta\ell}(N^+)$ be the N\'{e}ron model of $J^{\Delta\ell}(N^+)_{\mathbb{Q}_\ell}$ over $\mathbb{Z}_\ell$.
Let $\mathfrak{J}_s := \mathfrak{J}^{\Delta\ell}(N^+) \times_{\mathbb{Z}_\ell} \mathbb{F}_{\ell^2}$ be the special fiber over $\mathbb{F}_{\ell^2}$ and
 $\mathfrak{J}^\circ_s$ be the connected component of $\mathfrak{J}_s$ at the identity.
 Let $\Phi^{\Delta\ell}(N^+) :=  \mathfrak{J}_s / \mathfrak{J}^\circ_s$ be the group of connected components of $\mathfrak{J}_s$, which is an \'{e}tale group scheme over $\mathbb{F}_{\ell^2}$ admitting the action of $\mathbb{T}^{\Delta}(\ell N^+)$.
Let $\mathrm{red}_\ell : J^{\Delta\ell}(N^+)(K_\ell) \to \Phi^{\Delta\ell}(N^+)$ be the reduction map.
\subsubsection{The graph description of the component group} \label{subsubsec:graph}
We describe $\Phi^{\Delta\ell}(N^+)$ in terms of the graph $\mathcal{G}$. See \cite[$\S$5.5]{bertolini-darmon-imc-2005}, \cite[$\S$1.6]{nekovar-hilbert}, and \cite[$\S$3.4]{chida-hsieh-main-conj} for detail.
Let
\[
\xymatrix@R=0em{
s: \overset{\to}{\mathcal{E}}(\mathcal{G}) \ar[r] & \mathcal{V}(\mathcal{G}) & e =(v_1 \to v_2) \ar@{|->}[r]^-{s} & s(e) = v_1 \\
t: \overset{\to}{\mathcal{E}}(\mathcal{G}) \ar[r] & \mathcal{V}(\mathcal{G}) & e =(v_1 \to v_2)\ar@{|->}[r]^-{t} & t(e) =v_2
}
\]
be the source and target maps. Using the identifications in $\S$\ref{subsubsec:reduction_identification},
we induces maps (with the same notation)
\[
\xymatrix@R=1.2em{
\overset{\to}{\mathcal{E}}(\mathcal{G}) \ar[r]^-{s} & \mathcal{V}(\mathcal{G}) & \overset{\to}{\mathcal{E}}(\mathcal{G}) \ar[r]^-{t} & \mathcal{V}(\mathcal{G}) \\
{\mathcal{E}}(\mathcal{G}) \ar@{^{(}->}[u] \ar[r]^-{s} & B^\times \backslash \widehat{B}^\times / \widehat{R}^\times \times \lbrace 0 \rbrace \ar@{^{(}->}[u] & {\mathcal{E}}(\mathcal{G}) \ar@{^{(}->}[u] \ar[r]^-{t} & B^\times \backslash \widehat{B}^\times / \widehat{R}^\times \times \lbrace 1 \rbrace \ar@{^{(}->}[u] \\
B^\times b \widehat{R_0(\ell)}^\times \ar@{|->}[r]^-s & (B^\times b \widehat{R}^\times, 0) &  B^\times b \widehat{R_0(\ell)}^\times \ar@{|->}[r]^-t & (B^\times b \xi_\ell \widehat{R}^\times, 1)
}
\]
where $\xi_\ell = \left( \begin{smallmatrix} 1 & 0 \\ 0 & \ell \end{smallmatrix} \right) \in \mathrm{GL}_2(\mathbb{Q}_\ell) \simeq B^\times_\ell$. We also induce the maps for the chain and cochain complexes of $\mathcal{G}$, which can be identified with degeneracy maps between the spaces of automorphic forms as follows:
\[
\xymatrix@R=1.2em{
\mathbb{Z}[\mathcal{E}(\mathcal{G})] \ar[r]^-{d_* = -s_* + t_*} \ar@{=}[d] & \mathbb{Z}[\mathcal{V}(\mathcal{G})] \ar@{=}[d]
& \mathbb{Z}[\mathcal{V}(\mathcal{G})] \ar[r]^-{d^* = -s^* + t^*} \ar@{=}[d] & \mathbb{Z}[\mathcal{E}(\mathcal{G})] \ar@{=}[d]\\
S^{N^-}_2(N^+\ell, \mathbb{Z}) \ar[r]^-{-\alpha_* + \beta_*} & S^{N^-}_2(N^+, \mathbb{Z})^{\oplus 2}
& S^{N^-}_2(N^+, \mathbb{Z})^{\oplus 2} \ar[r]^-{-\alpha^* + \beta^*} & S^{N^-}_2(N^+\ell, \mathbb{Z})
}
\]
where $(\alpha^*f)(b) = f(b)$, $(\beta^*f)(b) = f(b\xi_\ell)$ are the injective degeneracy maps, and $\alpha_*$ and $\beta_*$ are the corresponding trace maps as in \cite[$\S$1.2.1]{nekovar-hilbert}.

Let $\mathbb{Z}[\mathcal{V}(\mathcal{G})]_0 := d_* \left( \mathbb{Z}[\mathcal{E}(\mathcal{G})] \right) \subseteq \mathbb{Z}[\mathcal{V}(\mathcal{G})]$. Following \cite[$\S$1.6.5]{nekovar-hilbert} (using \cite[Appendix by Edixhoven]{bertolini-darmon-rigidanalytic-1997}), we have a canonical isomorphism
\[
\xymatrix{
\mathbb{Z}[\mathcal{E}(\mathcal{G})] / \mathrm{Im}(d^*) \ar[r]^-{d_*}_-{\simeq} & \mathbb{Z}[\mathcal{V}(\mathcal{G})]_0 / \mathrm{Im}(d_* \circ d^*) \simeq \Phi^{\Delta\ell}(N^+). 
}
\]

Then we have the following commutative diagram:
\begin{equation} \label{eqn:comm_diagram_1st_explicit}
\begin{gathered}
\xymatrix{
\mathrm{Div}^0 M^{\Delta\ell}(N^+)(K_\ell) \ar[r]^-{cl}  \ar[d]^-{cl_\mathcal{V}}& J^{\Delta\ell}(N^+)(K_\ell) \ar[d]^-{\mathrm{red}_\ell}\\
\mathbb{Z}[\mathcal{V}(\mathcal{G})]_0 / \mathrm{Im}(d_* \circ d^*) \ar[r]^-{\simeq} & \Phi^{\Delta\ell}(N^+) .
}
\end{gathered}
\end{equation}
where $cl_\mathcal{V}$ is
the specialization map of divisors
defined in  \cite[$\S$1.6.6]{nekovar-hilbert}. See \cite[(3.10)]{chida-hsieh-main-conj} for the explicit formula of $cl_\mathcal{V}$.

\subsection{CM points unramified at $\ell$} \label{subsec:cm_points}
Let $z_0 \in \mathbb{C} \setminus \mathbb{R}$ be a point fixed by the image of $K^\times$ in $\mathrm{GL}_2(\mathbb{R})$ under the composition $\otimes_{\mathbb{Q}} \mathbb{R} \circ t^{\Delta\ell} : K^\times \to \mathcal{B}^{\Delta\ell, \times} \to \mathrm{GL}_2(\mathbb{R})$.
The \textbf{set of CM points by $K$ unramified at $\ell$ on $M^{\Delta\ell}_{\mathcal{U}}$} is defined by
$$\mathrm{CM}_{K, \ell}\left( M^{\Delta\ell}(N^+) \right)  := \left\lbrace [z_0, b]_{\mathcal{U}} : b \in \widehat{\mathcal{B}^{\Delta\ell}}^\times, b_\ell =1\right\rbrace  \subset M^{\Delta\ell}(N^+) (K^{\mathrm{ab}}) .$$
Shimura's reciprocity law implies the existence of embedding $\iota_\ell : \mathrm{CM}_{K, \ell}\left( M^{\Delta\ell}(N^+) \right) \subset M^{\Delta\ell}(N^+) (K_\ell) $
as in \cite[$\S$1.8.3]{nekovar-hilbert} and \cite[$\S$3.6]{chida-hsieh-main-conj}.
Let
$\mathrm{CM}^0_{K, \ell}\left( M^{\Delta\ell}(N^+) \right) \subset \mathrm{Div}^0 M^{\Delta\ell}(N^+)(K_\ell$) be the subgroup generated by the degree zero divisors supported in $\mathrm{CM}_{K, \ell}\left( M^{\Delta\ell}(N^+) \right)$.
Then we define the reduction map $\mathrm{red}_\mathcal{V}$ by
\[
\xymatrix@R=0em{
\mathrm{CM}^0_{K, \ell}\left( M^{\Delta\ell}(N^+) \right) \ar[r]^-{\mathrm{red}_\mathcal{V}} & \mathbb{Z}[\mathcal{V}(\mathcal{G})]_0 \\
\sum_i n_i [z,b]_{\mathcal{U}} \ar@{|->}[r] &\sum_i n_i \left[\varphi^{-1}_{B^{\Delta}, \mathcal{B}^{\Delta\ell}}(b) \right]_U
}
\]
and we have $\mathrm{red}_\mathcal{V} \circ \iota_\ell  = cl_{\mathcal{V}}$ where $cl_{\mathcal{V}}$ is defined in $\S$\ref{subsubsec:graph}.
\section{Level raising mod $\varpi^n$} \label{sec:level_raising}
Let $f_\alpha = \sum a_n (f) q^n \in S_2(\Gamma_0(N^+p\Delta))^{\mathrm{new}}$ be the $\alpha = \alpha_p(f)$-stabilization of $f$.
Using the integral Jacquet-Langlands correspondence as in \ref{subsec:modular_forms}, we regard $f_{\alpha}$ as a modular form in $S^{\Delta}_2(N^+p)$. We define the $\mathcal{O}$-algebra homomorphism
$$\pi_{f_\alpha} : \mathbb{T}^{\Delta}(N^+p)  \to \mathcal{O}$$
 associated to $f_\alpha$ by the assignments
$T_q  \mapsto a_q(f)$ for $q \nmid N^+p\Delta$, $U_q  \mapsto a_q(f)$ for $q \mid N^+$, $U_q  \mapsto \alpha_p(f)$  for $q = p$, $U_q  \mapsto \varepsilon_q(f)$ for $q \mid \Delta$, and $\langle a \rangle  \mapsto 1$.

\subsection{Level raising mod $\varpi^n$ at one $n$-admissible prime: from definite to indefinite} \label{subsec:level_raising_at_one_prime}
We explain level raising mod $\varpi^n$ at an $n$-admissible prime. See \cite[Theorem 5.15]{bertolini-darmon-imc-2005}, \cite[$\S$4.2]{chida-hsieh-main-conj}, and \cite[$\S$7.4.1]{longo-hilbert-modular-case} for detail.

\begin{defn}[$n$-admissible forms] \label{defn:n-admissible-forms}
An \textbf{$n$-admissible form $\mathcal{D} = (\Delta, g)$} is a pair consisting of
\begin{itemize}
\item a square-free integer $\Delta$ of an odd number of prime factors
\item an $\mathcal{O}_n$-valued quaternionic eigenform $g \in S^{\Delta}_2(N^+p, \mathcal{O}_n)$
\end{itemize}
such that the following conditions hold:
\begin{enumerate}
\item $N^- \mid \Delta$ and every prime factor of $\Delta/N^-$ is $n$-admissible;
\item $g \not\equiv 0 \pmod{\varpi}$;
\item $g$ is a $\mathbb{T}^{\Delta}(N^+p)$-eigenform and $\pi_g \equiv \pi_{f_\alpha} \pmod{\varpi^n}$, where
$\pi_g : \mathbb{T}^{\Delta}(N^+p)  \to \mathcal{O}_n$ is the $\mathcal{O}$-algebra homomorphism induced by $g$.
\end{enumerate}
We write $\mathcal{I}^{\Delta}_g := \ker(\pi_g)$.
\end{defn}
The concept of $n$-admissible forms allows us to remove all the rigid pair argument and the $p$-isolated condition in \cite{bertolini-darmon-imc-2005}.
Our $n$-admissible forms are slightly different from Chida-Hsieh's ones \cite[Definition 4.1]{chida-hsieh-main-conj} since the level is maximal at $p$. We do not need any $p$-power in the level since we focus only on the forms of weight two.

For an $n$-admissible form $\mathcal{D} = (\Delta, g)$, we define a surjective $\mathcal{O}$-linear map $\psi_g$ by
\[
\xymatrix@R=0em{
S^{\Delta}_2(N^+p, \mathcal{O}) / \mathcal{I}^{\Delta}_g \ar@{->>}[r]^-{\psi_g} & \mathcal{O}_n \\
h \ar@{|->}[r] & \langle g, h \rangle_U .
}
\]
\begin{prop}[{\cite[Proposition 4.2]{chida-hsieh-main-conj}}] Under Condition CR, $\psi_g$ becomes an isomorphism
$$\psi_g : S^{\Delta}_2(N^+, \mathcal{O}) / \mathcal{I}^{\Delta}_g \simeq \mathcal{O}_n$$
\end{prop}

Fix an $n$-admissible form $\mathcal{D}=(\Delta, g)$ and an $n$-admissible prime $\ell \nmid \Delta$ with $a_\ell (f) \equiv \varepsilon_{\ell} \left( \ell+1 \right) \Mod{\varpi^n}$.

We first \emph{formally} extend $\pi_g$ to an $\mathcal{O}$-algebra homomorphism 
$\pi_{g'_\ell} : \mathbb{T}^{\Delta}(\ell N^+p) \to \mathcal{O}_n$ by defining
$\pi_{g'_\ell} (U_\ell) = \varepsilon_\ell$. Since $U_\ell$-eigenvalue is $\varepsilon_\ell$ modulo $\varpi^n$, the extended map $\pi_{g'_\ell}$ factors through the $\ell$-new quotient $\mathbb{T}^{\Delta\ell}(N^+p)$ modulo $\varpi^n$.
We define the $\mathcal{O}$-algebra homomorphism $\pi_{g_\ell}: \mathbb{T}^{\Delta\ell}(N^+p) \to \mathcal{O}_n$ by the following commutative diagram
\[
\xymatrix{
\mathbb{T}^{\Delta}(\ell N^+p) \ar@/^2pc/[rr]^-{\pi_{g'_\ell}} \ar@{->>}[r] \ar@{->>}[rd]_-{\ell\textrm{-new quotient}} & \mathbb{T}^{\Delta}(N^+p)  \ar[r]^-{\pi_g} & \mathcal{O}_n \\
& \mathbb{T}^{\Delta\ell}(N^+p) \ar[ru]_-{\pi_{g_\ell}}
}
\]
and write $\mathcal{I}^{\Delta\ell}_{g_\ell} = \ker(\pi_{g_\ell})$.
\begin{rem}[on the existence of $\Delta\ell$-new form $g_\ell$ corresponding to $\pi_{g_\ell}$]
The existence of $g_\ell$ in $S^{\Delta\ell}_2(N^+p)$ itself is not a formal result at all because we also need to match up the relation between Hecke modules where the Hecke algebras act faithfully in the argument. In other words, we need to check that the Hecke eigensystem of $\pi_{g_\ell}$ arises in the Hecke module over $\mathcal{O}_n$ where $\mathbb{T}^{\Delta\ell}(N^+p)$ acts faithfully.
Note that Condition PO also applies to $g_\ell$. Thus, the mod $\varpi^n$ Hecke eigensystem of $g_\ell$ is also $p$-old. This is in fact tacitly used in  \cite[Theorem 5.15]{bertolini-darmon-imc-2005}, whose proof uses Ribet's exact sequence extensively.  One may use \cite{diamond-taylor-non-optimal} instead.
In fact, the existence of $g_\ell$ and Theorem \ref{thm:level_raising} are the content of \cite[Theorem 5.15]{bertolini-darmon-imc-2005} under Condition CR.
\end{rem}

Let $\mathcal{B}^{\Delta\ell}$ be the indefinite quaternion algebra of discriminant $\Delta\ell$.
Let $M^{\Delta\ell}(N^+p)$ be the Shimura curve attached to $\mathcal{B}^{\Delta\ell}$ of level $\mathcal{U}_0(p)$.
Let $J^{\Delta\ell}(N^+p)$ be the Jacobian of $M^{\Delta\ell}(N^+)$.
Let $\Phi^{\Delta\ell}(N^+p) $ be the group of connected components of the special fibre of the N\'{e}ron model of $J^{\Delta\ell}(N^+p)$ over $\mathcal{O}_{K_{\ell}}$. Write $M_\mathcal{O} := M \otimes_{\mathbb{Z}} \mathcal{O}$.
\begin{thm}[Level raising mod $\varpi^n$; {\cite[Theorem 4.3]{chida-hsieh-main-conj}}] \label{thm:level_raising}
Assume Condition CR.
Then $$\Phi^{\Delta\ell}(N^+p)_{\mathcal{O}} / \mathcal{I}^{\Delta\ell}_{g_\ell} \simeq S^{\Delta}_2(N^+p) / \mathcal{I}^{\Delta}_g \simeq \mathcal{O}_n.$$
\end{thm}

Let $\mathrm{Ta}_p(J^{\Delta\ell}(N^+p) )$ be the $p$-adic Tate module of $J^{\Delta\ell}(N^+p)$.
\begin{cor}[{\cite[Corollary 4.4]{chida-hsieh-main-conj}}] \label{cor:arithmetic_jochnowitz}
Under Condition CR, we have
$$\mathrm{Ta}_p(J^{\Delta\ell}(N^+p) )_{\mathcal{O}} / \mathcal{I}^{\Delta\ell}_{g_\ell} \simeq T_{f,n} .$$
\end{cor} 
Consider the (enhanced) reduction map
$$\widetilde{\mathrm{red}}_\ell : J^{\Delta\ell}(N^+p)(K(m)) \to \Phi^{\Delta\ell}(N^+p)_{\mathcal{O}} \otimes \mathcal{O}_n[\Gamma_m]$$
defined by
$$\widetilde{\mathrm{red}}_\ell (D) := \sum_{\sigma \in \Gamma_m} \mathrm{red}_\ell \left( \iota_\ell \left(  \sigma D\right)    \right) \sigma$$
where $\iota_\ell : J^{\Delta\ell}(N^+p)(K(m)) \to J^{\Delta\ell}(N^+p)(K_\ell)$ is an embedding since $\ell$ splits completely in $K(m)/K$. This $\iota_\ell$ coincides with the embedding with the same notation in $\S$\ref{subsec:cm_points}.

The next theorem is used to prove the first explicit reciprocity law.
\begin{thm}[{\cite[Theorem 4.5]{chida-hsieh-main-conj}}] \label{thm:reduction_diagram}
Assume Condition CR.
\begin{enumerate}
\item There is an isomorphism of $\mathcal{O}_n$-modules
\[
\xymatrix{
\psi_g : \Phi^{\Delta\ell}(N^+p)_{\mathcal{O}} / \mathcal{I}^{\Delta\ell}_{g_\ell} \ar[r]^-{\simeq} & \mathrm{H}^1_{\mathrm{sing}}(K_\ell, T_{f,n})
}
\]
which is canonical up to the choice of an isomorphism in Corollary \ref{cor:arithmetic_jochnowitz}.
\item There is a commutative diagram
\[
\xymatrix{
J^{\Delta\ell}(N^+p)(K(m)) / \mathcal{I}^{\Delta\ell}_{g_\ell} \ar[r]^-{\mathrm{Kum}} \ar[d]^-{\widetilde{\mathrm{red}}_\ell} & \mathrm{H}^1(K(m), T_{f,n}) \ar[d]^-{\partial_\ell} \\
\Phi^{\Delta\ell}(N^+p)_{\mathcal{O}} \otimes \mathcal{O}_n[\Gamma_m] \ar[r]^-{\psi_g}_-{\simeq} & \mathrm{H}^1_{\mathrm{sing}}(K(m)_\ell, T_{f,n}) .
}
\]
\end{enumerate}
\end{thm}
\subsection{Ihara input}
We recall two results: (a special case of) Ihara's lemma for weight 2 forms on Shimura curves and Ihara's work on rational points on Shimura curves over finite fields.
\subsubsection{Ihara's lemma for Shimura curves}
Let $q \geq 4$ be an integer dividing $M^+$.
Let $\mathcal{B}$ be the indefinite quaternion algebra over $\mathbb{Q}$ of discriminant $\Delta_{\mathcal{B}}$ and $\mathcal{V} = \widehat{\mathcal{R}}^{\times}_{M^+/q, q}$ be the open subgroup of $\widehat{\mathcal{B}}^\times$ representing the $\Gamma_0(M^+/q)\cap \Gamma_1(q)$-level structure.
Let $M^{\Delta_{\mathcal{B}}}_\mathcal{V} = M^{\Delta_{\mathcal{B}}}(N^+/q,q)$ be the corresponding Shimura curve.
Let $\mathscr{L}_2(N^+/q,q, \mathbb{F}) := \mathrm{H}^1_{\mathrm{\acute{e}t}}(M^{\Delta_{\mathcal{B}}}(M^+/q,q)_{\overline{\mathbb{Q}}}, \underline{\mathbb{F}})$ be the \'{e}tale cohomology of the Shimura curve with $\mathbb{F}$-coefficients, where $\underline{\mathbb{F}}$ is the constant \'{e}tale sheaf associated to $\mathbb{F}$.  
Let $\ell$ be a prime not dividing $p\Delta_{\mathcal{B}}$.
Let $\mathfrak{m}_{g'_\ell} \subseteq \mathbb{T}^{\Delta_{\mathcal{B}}}(\ell M^+/q,q)$ be the maximal ideal containing the kernel of the map
$$\pi_{g'_\ell} : \mathbb{T}^{\Delta_{\mathcal{B}}}(\ell M^+/q,q)  \to \mathbb{F}$$
as in $\S$\ref{subsec:level_raising_at_one_prime} although the roles of the definite and the indefinite change.
Then it is also ordinary and non-Eisenstein.
\begin{thm}[Ihara's lemma for Shimura curves; {\cite[Theorem 2]{diamond-taylor-non-optimal}}] \label{thm:ihara_lemma_shimura_curves}
Let $q \geq 4$ be an integer dividing $M^+$.
Assume that $p$ and $\ell$ do not divide $M^+$. Then 
the map
$$\alpha_\ell = 1 + \eta_\ell : \left( \mathscr{L}_2(M^+/q,q, \mathbb{F})^{\oplus 2} \right)_{\mathfrak{m}_{g'_\ell}} \to 
\left( \mathscr{L}_2(\ell M^+/q,q, \mathbb{F}) \right)_{\mathfrak{m}_{g'_\ell}}$$
is injective where $\eta_\ell$ is the degeneracy map at $\ell$.
\end{thm}
The module $\mathscr{L}_2(M^+/q, q, \mathbb{F})^{\oplus 2}$ admits a $\mathbb{T}^{\Delta_{\mathcal{B}}}(\ell M^+/q,q)$-module structure. See \cite[Proposition 5.8]{bertolini-darmon-imc-2005} and \cite[$\S$3.5]{chida-hsieh-main-conj} for detail.

Now we put $M^+ = N^+p$ and $q =p$. Although it does not satisfy the condition $p \nmid M^+$, the following lemma shows that we can still apply Theorem \ref{thm:ihara_lemma_shimura_curves} to our setting.
\begin{lem}[{\cite[Lemma 5.3]{chida-hsieh-main-conj}}] \label{lem:preparation_ihara}
Suppose that $\mathfrak{m}$ satisfies Condition PO. Then
$$\mathscr{L}_2(N^+, \mathbb{F})_{\mathfrak{m}} \simeq \mathscr{L}_2(N^+p, \mathbb{F})_{\mathfrak{m}} \simeq \mathscr{L}_2(N^+,p, \mathbb{F})_{\mathfrak{m}} .$$
\end{lem}
Due to Lemma \ref{lem:preparation_ihara}, we can remove the maximal condition on the level at $p$ in Theorem \ref{thm:ihara_lemma_shimura_curves}.
\begin{cor} \label{cor:ihara_lemma_shimura_curves}
Suppose that $\mathfrak{m}_{g'_\ell}$ satisfies Condition PO.
Assume that $p$ and $\ell$ do not divide $N^+$. Then 
the map
$$\alpha_\ell = 1 + \eta_\ell : \left( \mathscr{L}_2(N^+,p, \mathbb{F})^{\oplus 2} \right)_{\mathfrak{m}_{g'_\ell}} \to 
\left( \mathscr{L}_2(\ell N^+,p, \mathbb{F}) \right)_{\mathfrak{m}_{g'_\ell}}$$
is injective where $\eta_\ell$ is the degeneracy map at $\ell$.
\end{cor}

\subsubsection{Rational points on Shimura curves over finite fields} \label{subsubsec:rational_pts_shimura}
Let $\mathcal{B}$ be the indefinite quaternion algebra over $\mathbb{Q}$ of discriminant $\Delta_{\mathcal{B}}$.
Fix an embedding $\mathcal{B}$ into $\mathrm{M}_2(\mathbb{R})$, which is compatible with the convention in $\S$\ref{subsec:complex_uniformization}.
Let $\mathcal{R} \subset \mathcal{B}$ be an Eichler order of level $N^+p$.
Let $\Gamma_{\infty, N^+p}$ be the image of the group of reduced norm one elements in $\mathcal{R}[1/\ell]^\times / \lbrace \pm 1\rbrace$ in $\mathrm{PSL}_2(\mathbb{R})$. Then $M^{\Delta_{\mathcal{B}}} (N^+)(\mathbb{C}) \simeq \Gamma_{\infty, N^+p} \backslash \mathfrak{h}$ as compact Riemann surfaces where $\mathfrak{h}$ is the complex upper-half plane. See also \cite[$\S$5.1]{bertolini-darmon-imc-2005}.
Let $\Gamma_{\infty, N^+,p} \subset \Gamma_{\infty, N^+p}$ be the finite index subgroup which is determined by
$M^{\Delta_{\mathcal{B}}} (N^+,p)(\mathbb{C}) \simeq \Gamma_{\infty, N^+,p} \backslash \mathfrak{h}$.  Then
$\Gamma_{\infty, N^+,p}$ is torsion-free. See \cite[Proof of Proposition 9.2]{bertolini-darmon-imc-2005}. See also \cite[Corollary 3, $\S$4]{diamond-taylor-non-optimal} for the geometric property of the torsion-freeness.
In this setting, we have the following statement.
\begin{prop}[{\cite[(G), $\S$3]{ihara-shimura-curves}}] \label{prop:ihara_rational_pts_shimura}
Let $\ell$ be a prime not dividing $N^+p \Delta_{\mathcal{B}}$. Then we have
$$\mathfrak{J}^{\Delta_{\mathcal{B}}}(N^+,p)(\mathbb{F}_{\ell^2}) / \mathfrak{J}^{\Delta_{\mathcal{B}}}(N^+,p)(\mathbb{F}_{\ell^2})^{ss} \simeq \left( \Gamma_{\infty, N^+,p} \right)^{ab}$$
where 
$\mathfrak{J}^{\Delta_{\mathcal{B}}}(N^+,p)(\mathbb{F}_{\ell^2})^{ss} \subset \mathfrak{J}^{\Delta_{\mathcal{B}}}(N^+,p)(\mathbb{F}_{\ell^2})$ is the subgroup generated by the divisors supported on the supersingular points in $\mathfrak{J}^{\Delta_{\mathcal{B}}}(N^+,p)(\mathbb{F}_{\ell^2})$ and
$\left( \Gamma_{\infty, N^+,p} \right)^{ab}$ is the abelianization of $\Gamma_{\infty, N^+,p}$.
\end{prop}

\subsection{Level raising mod $\varpi^n$ at two $n$-admissible prime: from definite to definite passing through indefinite} \label{subsec:level_raising_at_two_prime}
We study the level raising at two $n$-admissible primes in detail. We focus on how ``Ihara input" is used.
In order to apply Ihara's work, one needs to work with the level ``smaller" than the $\Gamma_1(q)$-level structure with $q \geq 4$
as described in Remark \ref{rem:normalization_modular_forms}.
Note that the level-raised form exists only in the localized Hecke module.
See \cite[$\S$9]{bertolini-darmon-imc-2005}, \cite[$\S$7.4.3]{longo-hilbert-modular-case} and \cite[$\S$11, Level raising at two primes]{vanorder-dihedral} for detail. Also, see \cite[$\S$5.3]{chida-hsieh-main-conj} for the moduli description.

We first recall some definitions to draw the important commutative diagram, which exactly describes the level raising at two $n$-admissible primes.
Let $\ell_1$ and $\ell_2$ be distinct $n$-admissible primes not dividing $\Delta$.
Let $\mathbb{T}^{\Delta\ell_1}(N^+p)$ be the Hecke algebra acting on Shimura curve $X^{\Delta\ell_1}(N^+p)$.
Let $J^{\Delta\ell_1}(N^+p)(K_{\ell_2})$ be the $K_{\ell_2}$-points of the Jacobian of $X^{\Delta\ell_1}(N^+p)$.
Let $\mathfrak{J}^{\Delta\ell_1}(N^+)(\mathbb{F}_{(\ell_2)^2})$ be the $\mathbb{F}_{(\ell_2)^2}$-points of the reduction at $\ell_2$ of the N\'{e}ron model $\mathfrak{J}^{\Delta\ell_1}(N^+p)$ over $\mathcal{O}_{K_\ell}$.
Let $\mathfrak{J}^{\Delta\ell_1}(N^+p)(\mathbb{F}_{(\ell_2)^2})^{ss}$ be the subgroup generated by the divisors supported on the supersingular points in $\mathfrak{J}^{\Delta\ell_1}(N^+p)(\mathbb{F}_{(\ell_2)^2})$. 
We first draw the following commutative diagram and explain the nontrivial part.
\[
\xymatrix{
J^{\Delta\ell_1}(N^+p)(K_{\ell_2}) / \mathcal{I}^{\Delta\ell_1}_{g_{\ell_1}} \ar@{->>}[r]^-{\mathrm{Kum}_{\ell_2}} \ar[d]^-{\simeq}_{\mathrm{red}_{\ell_2}} &\mathrm{H}^1(K_{\ell_2}, \mathrm{Ta}_p\left( J^{\Delta\ell_1}(N^+p) \right) / \mathcal{I}^{\Delta\ell_1}_{g_{\ell_1}}  ) \ar[r]^-{\simeq}  & \mathrm{H}^1(K_{\ell_2}, T_{f,n}) \ar@/^1PC/[d]^-{v_{\ell_2}} \\
\mathfrak{J}^{\Delta\ell_1}(N^+p)(\mathbb{F}_{(\ell_2)^2}) / \mathcal{I}^{\Delta\ell_1}_{g_{\ell_1}} & & \mathrm{H}^1_{\mathrm{fin}}(K_{\ell_2}, T_{f,n}) \ar@{^{(}->}[u]^-{\simeq} \ar[d]^-{\simeq} \\
 \mathfrak{J}^{\Delta\ell_1}(N^+p)(\mathbb{F}_{(\ell_2)^2})^{ss}  / \mathcal{I}^{\Delta\ell_1}_{g_{\ell_1}} \ar@{^{(}->}[u] & & \mathcal{O}_n \\
\mathfrak{J}^{\Delta\ell_1}(N^+p)(\mathbb{F}_{(\ell_2)^2})^{ss} \ar@{->>}[u] \\
\mathbb{Z} [ B^{\Delta\ell_1\ell_2, \times} \backslash \widehat{B}^{\Delta\ell_1\ell_2, \times} / \widehat{R}^\times_{N^+p} ] \ar[u]^-{\simeq}  \ar[uurr]^-{g^{*}_{\ell_1\ell_2}}
}
\]
Using the above diagram and the pairing in $\S$\ref{subsec:modular_forms}, we define the level-raised $n$-admissible form $\mathcal{D}^{\ell_2\ell_2}_n = (\Delta\ell_1\ell_2, g_{\ell_1\ell_2})$ by the equality
$$g^{*}_{\ell_1\ell_2}(b) = \langle g_{\ell_1\ell_2}, b\rangle_U.$$
\begin{rem} $ $
\begin{enumerate}
\item The form $g_{\ell_1\ell_2}$ is a mod $\varpi^n$ $\mathbb{T}^{\Delta\ell_1\ell_2}(N^+p)$-eigenform due to \cite[Lemma 9.1]{bertolini-darmon-imc-2005}.
\item The reduction map at $\mathbb{F}_{(\ell_2)^2}$
$$ \mathrm{red}_{\ell_2}: J^{\Delta\ell_1}(N^+p)(K_{\ell_2}) / \mathcal{I}^{\Delta\ell_1}_{g_{\ell_1}} \to \mathfrak{J}^{\Delta\ell_1}(N^+p)(\mathbb{F}_{(\ell_2)^2}) / \mathcal{I}^{\Delta\ell_1}_{g_{\ell_1}}$$
is an isomorphism because $J^{\Delta\ell_1}(N^+p)$ has good reduction at $\ell_2$.
\item The natural inclusion $$\mathrm{H}^1_{\mathrm{fin}}(K_{\ell_2}, T_{f,n}) \hookrightarrow \mathrm{H}^1(K_{\ell_2}, T_{f,n})$$ also becomes an isomorphism because $\mathrm{Ta}_p\left( J^{\Delta\ell_1}(N^+p) \right) / \mathcal{I}^{\Delta\ell_1}_{g_{\ell_1}}   \simeq T_{f,n}$ is unramified at $\ell_2$.
\item Since the image of the Kummer map at $\ell_2$ of $J^{\Delta\ell_1}(N^+p)(K_{\ell_2}) / \mathcal{I}^{\Delta\ell_1}_{g_{\ell_1}}$ is $\mathrm{H}^1_{\mathrm{fin}}(K_{\ell_2}, T_{f,n})$, the Kummer map $\mathrm{Kum}_{\ell_2}$ becomes surjective.
\end{enumerate}
\end{rem}
Following \cite[Proposition 9.2]{bertolini-darmon-imc-2005}, \cite[$\S$7.4.3]{longo-hilbert-modular-case}, and \cite[Proposition 11.12]{vanorder-dihedral}, we give a detailed exposition of the following statement.
\begin{prop} \label{prop:surjectivity_via_ihara}
Under Condition Ihara, the composition
$$J^{\Delta\ell_1}(N^+p)(\mathbb{F}_{(\ell_2)^2})^{ss} \twoheadrightarrow
J^{\Delta\ell_1}(N^+p)(\mathbb{F}_{(\ell_2)^2})^{ss} / \mathcal{I}^{\Delta\ell_1}_{g_{\ell_1}} \hookrightarrow
J^{\Delta\ell_1}(N^+p)(\mathbb{F}_{(\ell_2)^2}) / \mathcal{I}^{\Delta\ell_1}_{g_{\ell_1}}$$
in the above diagram is surjective. Thus, $g_{\ell_1\ell_2}$ is surjective.
If we further assume Condition CR, then $g_{\ell_1\ell_2}$ is determined uniquely up to $\mathcal{O}^\times_n$.
\end{prop}
\begin{rem}
The surjectivity of $g_{\ell_1\ell_2}$ is essential because it ensures nonvanishing modulo $\varpi$.
\end{rem}
The rest of this section is devoted to prove Proposition \ref{prop:surjectivity_via_ihara}.
We first recall the following result on trees.
\begin{prop}[{\cite[Proposition 13, $\S$II.2.8]{serre-trees}}] \label{prop:serre_trees}
Let $\mathscr{T}$ be a tree and $G$ be a group acting on $\mathscr{T}$.
Let $\Sigma_0$ be the set of representatives  of $\mathcal{V}(\mathscr{T})/G$ and
$\Sigma_1$ the set of representatives  of $\mathcal{E}(\mathscr{T})/G$.
Let $M$ be a $G$-module. Then we have an exact cohomology sequence
$$\cdots \to \mathrm{H}^i(G,M) \to \prod_{v \in \Sigma_0} \mathrm{H}^i(G_v,M) \to \prod_{e \in \Sigma_1} \mathrm{H}^i(G_e,M) \to \mathrm{H}^{i+1}(G,M) \to \cdots$$
where $G_v = \mathrm{Stab}_G(v)$ and $G_e = \mathrm{Stab}_G(e)$.
\end{prop}
We fix an embedding $\mathcal{B}^{\Delta\ell_1}$ into $\mathrm{M}_2(\mathbb{Q}_{\ell_2})$.
Let $\Gamma_{\ell_2, N^+p}$ be the image of the group of reduced norm one elements in $\mathcal{R}_{N^+p}[1/\ell]^\times / \lbrace \pm 1\rbrace$ in $\mathrm{PGL}_2(\mathbb{Q}_{\ell_2})$ and $\Gamma_{\ell_2, N^+,p} \subset \Gamma_{\ell_2, N^+p}$ be the similarly defined subgroup as in $\S$\ref{subsubsec:rational_pts_shimura}.
\begin{rem}
We have isomorphisms $\Gamma_{\ell_2, N^+p} \simeq \Gamma_{\infty, N^+p}$ and $\Gamma_{\ell_2, N^+,p} \simeq \Gamma_{\infty, N^+,p}$ as groups since both maps to $\mathrm{GL}_2(\mathbb{Q}_{\ell_2})$ and $\mathrm{GL}_2(\mathbb{R})$ are embeddings.
\end{rem}
Then we obtain an action of $\Gamma_{\ell_2, N^+p}$ (so also of $\Gamma_{\ell_2, N^+,p}$) on the Bruhat-Tits tree $\mathscr{T}_{\ell_2}$ for $\mathrm{PGL}_2(\mathbb{Q}_{\ell_2})$.

Let $v_0 \in \mathcal{V}(\mathscr{T}_{\ell_2})$ such that $\left( \Gamma_{\ell_2, N^+,p} \right)_{v_0} := \mathrm{Stab}_{\Gamma_{\ell_2, N^+,p}}(v_0) \subseteq \mathrm{PSL}_2(\mathbb{Z}_{\ell_2})$. 
Let $e_0$ be the edge whose source is $v_0$ such that the stabilizer $\left( \Gamma_{\ell_2, N^+,p} \right)_{e_0}$ is the subgroup of $\left( \Gamma_{\ell_2, N^+,p} \right)_{v_0}$ consisting of the upper triangular elements modulo $\ell_2$.
Write $v_1$ for the target of $e_0$.
Consider the identifications
\[
\xymatrix{
\left( \Gamma_{\ell_2, N^+,p} \right)_{v_0} \simeq \left( \Gamma_{\infty, N^+,p} \right)_{v_0} \subset \mathrm{PSL}_2(\mathbb{R}) , &
\left( \Gamma_{\ell_2, N^+,p} \right)_{e_0} \simeq \left( \Gamma_{\infty, N^+,p} \right)_{e_0}\subset \mathrm{PSL}_2(\mathbb{R}).
}
\]
Then the discrete subgroups at the infinite place define the corresponding Shimura curves as follows:
\begin{itemize}
\item $\left( \Gamma_{\infty, N^+,p} \right)_{v_0}$ defines Shimura curve $M^{\Delta\ell_1}(N^+,p)$, and
\item $\left( \Gamma_{\infty, N^+,p} \right)_{e_0}$ defines Shimura curve $M^{\Delta\ell_1}(\ell_2 N^+,p)$, respectively.
\end{itemize}
Applying Proposition \ref{prop:serre_trees} with $i = 1$, $M = \mathbb{F}$ and $G = \Gamma_{\ell_2, N^+,p}$, we have the following statement.
\begin{cor}
The sequence
{\scriptsize
\[
\xymatrix{
\mathrm{Hom}\left(\Gamma_{\infty, N^+,p}, \mathbb{F} \right) \ar[r] &
\mathrm{Hom}\left( \left( \Gamma_{\infty, N^+,p} \right)_{v_0}, \mathbb{F} \right) \oplus
\mathrm{Hom}\left( \left( \Gamma_{\infty, N^+,p} \right)_{v_1}, \mathbb{F} \right) \ar[r]^-{d} & 
\mathrm{Hom}\left( \left( \Gamma_{\infty, N^+,p} \right)_{e_0}, \mathbb{F} \right)
}
\]
}
is exact.
\end{cor}
By \cite[(3), $\S$1]{ling-shimura-subgroups}, a canonical homomorphism
$J^{\Delta\ell_1}(N^+,p)(\mathbb{C}) \to \mathrm{Hom}(\Gamma_{\infty, N^+,p}, S^1)$ is injective where $S^1$ is the unit circle in $\mathbb{C}$. Thus, we can identify the source and the target of $d$ with the $\varpi$-torsions of the Jacobians of the corresponding Shimura curves. Also, we identify the $\varpi$-torsions of the Jacobians with the first Betti cohomologies of the corresponding Shimura curves with $\mathbb{F}(1)$-coefficients, where $\mathbb{F}(1)$ is the Tate twist of $\mathbb{F}$. We identify the Betti cohomologies with mod $\varpi$ \'{e}tale cohomologies and take the inverse twist of the coefficients. Then the map $d$ becomes $\alpha_{\ell_2}$ as in Theorem \ref{thm:ihara_lemma_shimura_curves}. The following commutative diagram summarizes all the identifications.
{\scriptsize
\[
\xymatrix@R=1em{
\mathrm{Hom}\left(\Gamma_{\infty, N^+,p}, \mathbb{F} \right) \ar[r] &
\mathrm{Hom}\left( \left( \Gamma_{\infty, N^+,p} \right)_{v_0}, \mathbb{F} \right) \oplus
\mathrm{Hom}\left( \left( \Gamma_{\infty, N^+,p} \right)_{v_1}, \mathbb{F} \right) \ar[r]^-{d} \ar@{=}[d] & 
\mathrm{Hom}\left( \left( \Gamma_{\infty, N^+,p} \right)_{e_0}, \mathbb{F} \right) \ar@{=}[d]\\
& J^{\Delta\ell_1}(N^+,p)(\mathbb{C})[\varpi] \oplus J^{\Delta\ell_1}(N^+,p)(\mathbb{C})[\varpi] \ar[r] \ar@{=}[d] & J^{\Delta\ell_1}(\ell_2 N^+,p)(\mathbb{C})[\varpi] \ar@{=}[d] \\
& \mathrm{H}^1(M^{\Delta\ell_1}(N^+,p)(\mathbb{C}), \mathbb{F}(1)) \oplus \mathrm{H}^1(M^{\Delta\ell_1}(N^+,p)(\mathbb{C}), \mathbb{F}(1)) \ar[r] \ar[d]^-{\simeq} & \mathrm{H}^1(M^{\Delta\ell_1}(\ell_2 N^+,p)(\mathbb{C}), \mathbb{F}(1)) \ar[d]^-{\simeq} \\
& \mathrm{H}^1_{\mathrm{\acute{e}t}}(M^{\Delta\ell_1}(N^+,p)_{\overline{\mathbb{Q}}}, \mathbb{F}(1)) \oplus \mathrm{H}^1_{\mathrm{\acute{e}t}}(M^{\Delta\ell_1}(N^+,p)_{\overline{\mathbb{Q}}}, \mathbb{F}(1)) \ar[r]  \ar[d]^-{\simeq} & \mathrm{H}^1_{\mathrm{\acute{e}t}}(M^{\Delta\ell_1}(\ell_2 N^+,p)_{\overline{\mathbb{Q}}}, \mathbb{F}(1)) \ar[d]^-{\simeq} \\
& \mathrm{H}^1_{\mathrm{\acute{e}t}}(M^{\Delta\ell_1}(N^+,p)_{\overline{\mathbb{Q}}}, \mathbb{F}) \oplus \mathrm{H}^1_{\mathrm{\acute{e}t}}(M^{\Delta\ell_1}(N^+,p)_{\overline{\mathbb{Q}}}, \mathbb{F}) \ar[r]^-{\alpha_{\ell_2}} & \mathrm{H}^1_{\mathrm{\acute{e}t}}(M^{\Delta\ell_1}(\ell_2 N^+,p)_{\overline{\mathbb{Q}}}, \mathbb{F})
}
\]
}
We regard $g_{\ell_1}$ as a mod $\varpi^n$ modular form on $M^{\Delta\ell_1}(N^+,p)$ with trivial character.
Let $\mathcal{I}^{\Delta\ell_1}_{g_{\ell_1},p} \subset \mathbb{T}^{\Delta\ell_1}(N^+,p)$ be the ideal corresponding to $g_{\ell_1}$ and $\mathfrak{m}_{g_{\ell_1},p}$  be the maximal ideal in $\mathbb{T}^{\Delta\ell_1}(N^+,p)$.
Since $\mathfrak{m}_{g_{\ell_1},p}$ induces a residually irreducible Galois representation, we have
$\mathrm{Hom}\left(\Gamma_{\infty, N^+,p}, \mathbb{F}_p \right)[\mathfrak{m}_{g_{\ell_1},p}] = 0$ due to Corollary \ref{cor:ihara_lemma_shimura_curves}. Therefore,
$\left(\Gamma_{\infty, N^+,p}\right)^{ab} / \mathfrak{m}_{g_{\ell_1},p} = 0$ and also 
$\left(\Gamma_{\infty, N^+,p}\right)^{ab} / \mathcal{I}^{\Delta\ell_1}_{g_{\ell_1},p} = 0$ due to Nakayama's lemma with $\mathbb{T}^{\Delta\ell_1}(N^+,p)_{\mathfrak{m}_{g_{\ell_1},p}}$.
By Proposition \ref{prop:ihara_rational_pts_shimura}, we have
$$\mathfrak{J}^{\Delta\ell_1}(N^+,p)(\mathbb{F}_{\ell^2}) / \mathfrak{J}^{\Delta\ell_1}(N^+,p)(\mathbb{F}_{\ell^2})^{ss} \simeq \left( \Gamma_{\infty, N^+,p} \right)^{ab}.$$
Thus, we have the following statement.
\begin{lem} \label{lem:supersingular_surj}
The image of $\mathfrak{J}^{\Delta\ell_1}(N^+,p)(\mathbb{F}_{\ell^2})^{ss}$ in 
$\mathfrak{J}^{\Delta\ell_1}(N^+,p)(\mathbb{F}_{\ell^2}) / \mathcal{I}^{\Delta\ell_1}_{g_{\ell_1},p}$ is full.
\end{lem}
In order to reduce the level to the $\Gamma_0(N^+p)$-structure, consider the natural map
$$\mathfrak{J}^{\Delta\ell_1}(N^+,p)(\mathbb{F}_{\ell^2}) \to  \mathfrak{J}^{\Delta\ell_1}(N^+p)(\mathbb{F}_{\ell^2}) .$$
Its cokernel can be identified with the abelian quotient of the image of $\Gamma_0(p)$ in $\mathrm{SL}_2(\mathbb{Z}/p\mathbb{Z})$, hence has order dividing $\varphi(p) = p-1$, which is prime to $p$.
Then the composition
\[
\xymatrix{
\mathfrak{J}^{\Delta\ell_1}(N^+,p)(\mathbb{F}_{\ell^2}) \ar[r] \ar@/_2pc/@{->>}[rr] & \mathfrak{J}^{\Delta\ell_1}(N^+p)(\mathbb{F}_{\ell^2}) \ar[r] & \mathfrak{J}^{\Delta\ell_1}(N^+p)(\mathbb{F}_{\ell^2}) /\varpi^n
}
\]
is surjective.
\begin{rem}
In the non-ordinary case, $p$ does not divide the level. Thus, we should replace $N^+$ by $N^+/q$ and $p$ by $q$ under Condition Ihara (Assumption \ref{assu:ihara}).
\end{rem}
Since the following diagram is commutative
\[
\xymatrix@R=1.5em{
\mathfrak{J}^{\Delta\ell_1}(N^+,p)(\mathbb{F}_{\ell^2})^{ss} \ar@{->>}[rr]^-{\textrm{Lemma } \ref{lem:supersingular_surj}}  \ar[d] & & \mathfrak{J}^{\Delta\ell_1}(N^+,p)(\mathbb{F}_{\ell^2}) / \mathcal{I}^{\Delta\ell_1}_{g_{\ell_1},p} \ar@{->>}[d]^-{\textrm{the above surjection}}\\
\mathfrak{J}^{\Delta\ell_1}(N^+p)(\mathbb{F}_{\ell^2})^{ss} \ar[rr] & & \mathfrak{J}^{\Delta\ell_1}(N^+p)(\mathbb{F}_{\ell^2}) / \mathcal{I}^{\Delta\ell_1}_{g_{\ell_1}},
}
\]
the map at the bottom
$$\mathfrak{J}^{\Delta\ell_1}(N^+p)(\mathbb{F}_{\ell^2})^{ss} \to \mathfrak{J}^{\Delta\ell_1}(N^+p)(\mathbb{F}_{\ell^2}) / \mathcal{I}^{\Delta\ell_1}_{g_{\ell_1}}$$
is also surjective.

Let $\overline{\rho}$ be the residual Galois representation associated to the maximal ideal $\mathfrak{m}_{g_{\ell_1\ell_2}} \subset \mathbb{T}^{\Delta\ell_1\ell_2}(N^+p)$. Then $(\overline{\rho}, \Delta\ell_1\ell_2)$ satisfies Condition CR; thus, the following mod $\varpi$ multiplicity one result holds as in \cite[Lemma 3.6]{kim-summary} (also with the $p$-distinguished property of $\overline{\rho}$): 
$S^{\Delta\ell_1\ell_2}_2(N^+p, \mathbb{F})_{ \mathfrak{m}_{ g_{ \ell_1\ell_2 } } }[\mathfrak{m}_{ g_{ \ell_1\ell_2 } } ] \simeq \mathbb{F}$.
Then the argument of Wiles \cite[$\S$2]{wiles} shows that
$S^{\Delta\ell_1\ell_2}_2(N^+p, \mathcal{O})_{ \mathfrak{m}_{ g_{ \ell_1\ell_2 } } }$ is a free module over 
$\mathbb{T}^{\Delta\ell_1\ell_2}_2(N^+p)_{ \mathfrak{m}_{ g_{ \ell_1\ell_2 } } }$ of rank one. This implies that $g_{ \ell_1\ell_2 }$ is uniquely determined up to $\mathcal{O}^\times_n$ as in \cite[Proof of Theorem 5.5]{chida-hsieh-main-conj}. 

\begin{rem}
In \cite{chida-hsieh-main-conj}, such a multiplicity one result comes from the Taylor-Wiles system argument under Condition CR$^+$. See \cite[$\S$6]{chida-hsieh-p-adic-L-functions} for detail.
\end{rem}

\section{Construction of cohomology classes arising from Heegner points} \label{sec:construction_cohomology}
We fix the following data:
\begin{itemize}
\item an $n$-admissible form $\mathcal{D} = (\Delta, g)$;
\item an $n$-admissible prime $\ell \nmid \Delta$.
\end{itemize}
\subsection{Heegner points}
For $m  = \prod q^{r_i}_i$, write $m_0 = \prod q^{r_i + 1}_i$.
Consider the corresponding Shimura curve $M^{\Delta\ell}(N^+)$.
To each $a \in \widehat{K}^\times$, one can associate the Heegner point
$$P^{\Delta \ell}_{m_0} (a) \in M^{\Delta \ell}(N^+p)(H(m_0)) \cap  \mathrm{CM}_{K, \ell}\left( M^{\Delta\ell}(N^+p) \right)$$
using Shimura's reciprocity law.
\begin{rem}
For explicit definition of those points, see \cite[(4.8)]{chida-hsieh-main-conj}.
Since our Shimura curve has the $\Gamma_0(p)$-level structure, we do not impose condition on $r$ with $m=m'p^r$, which is given in \cite[$\S$4.3.2]{chida-hsieh-main-conj} and stated in \cite[Theorem 5.1]{chida-hsieh-main-conj}. Their Shimura curves have level $p^n$ with condition $r \geq n$ in order to use ``mod $\varpi^n$ Hida theory" for reduction of weight. See \cite[Proposition 4.2]{chida-hsieh-main-conj} for their weight reduction argument.
\end{rem}
\subsection{Construction}
Let $P^{\Delta \ell}_{m_0} := P^{\Delta \ell}_{m_0}(1)  \in M^{\Delta \ell}(N^+p)$.
Choose an auxiliary prime $q_0 \nmid p\ell N^+\Delta$ such that $1 + q_0 - a_{q_0}(f) \in \mathcal{O}^\times$.
We define
\begin{align*}
\xi^{\Delta\ell}_{q_0} : \mathrm{Div}M^{\Delta \ell}(N^+p)(H(m_0)) & \to J^{\Delta \ell}(N^+p)_{\mathcal{O}} \\
P & \mapsto \xi^{\Delta\ell}_{q_0}(P) := [  (1 + q_0 - T_{q_0})P ] \otimes (1 + q_0 - a_{q_0})^{-1} ,
\end{align*}
and define
$$D^{\Delta \ell}_m := \sum_{\sigma \in \mathrm{Gal}(H(m_0)/K(m))} \xi^{\Delta\ell}_{q_0} \left( \left(  P^{\Delta \ell}_{m_0} \right)^\sigma \right) \in J^{\Delta\ell} (N^+p)(K(m))_{\mathcal{O}}.$$
Let 
$$\mathrm{Kum}^{\Delta\ell} : J^{\Delta\ell} (N^+p)(K(m))_{\mathcal{O}} \to \mathrm{H}^1 \left( K(m), \mathrm{Ta}_p \left( J^{\Delta\ell} \left( N^+p \right) \right)_{\mathcal{O}} \right)$$
be the Kummer map.
Define 
$$\kappa_{\mathcal{D}}(\ell)_m := \kappa_{\mathcal{D}}(\ell)_{m'p^r} = \dfrac{1}{\alpha^r} \cdot \mathrm{Kum}^{\Delta\ell} (D^{\Delta \ell}_m) \Mod{\mathcal{I}^{\Delta\ell}_{g_\ell}} \in \mathrm{H}^1 \left( K(m), \mathrm{Ta}_p \left( J^{\Delta\ell} \left( N^+\right) \right)_{\mathcal{O}} / \mathcal{I}^{\Delta\ell}_{g_\ell} \right) \simeq \mathrm{H}^1(K(m), T_{f,n}) .$$
Then these cohomology classes satisfy the norm compatibility relation
$$\mathrm{cores}^{K(m'p^{r})}_{{K(m'p^{r-1})}} \left( \kappa_{\mathcal{D}}(\ell)_{m'p^r} \right) = \kappa_{\mathcal{D}}(\ell)_{m'p^{r-1}} .$$
Thus, we have the norm compatible system
 $$\kappa_{\mathcal{D}}(\ell)_{m'p^\infty} := \left( \kappa_{\mathcal{D}}(\ell)_{m'p^r} \right)_r \in \mathrm{H}^1(K(m'p^\infty), T_{f,n}).$$
The following proposition concerns the local properties of $\kappa_{\mathcal{D}}(\ell)_{m'p^\infty}$.
\begin{prop}[{\cite[Proposition 4.7]{chida-hsieh-main-conj}}] \label{prop:local_properties}
$\kappa_{\mathcal{D}}(\ell)_{m'p^\infty} \in \mathrm{Sel}_{\Delta\ell}(K(m'p^\infty), T_{f,n})$.
\end{prop}
\begin{proof}
We need to check the following local statements:
\begin{enumerate}
\item $\partial_q( \kappa_{\mathcal{D}}(\ell)_m ) = 0$ for $q \nmid p\Delta \ell m'$;
\item $\partial_q( \kappa_{\mathcal{D}}(\ell)_m ) = 0$ for $q \mid N^+$;
\item $\partial_q( \kappa_{\mathcal{D}}(\ell)_m ) = 0$ for $q \mid m'$;
\item $\mathrm{res}_q( \kappa_{\mathcal{D}}(\ell)_m ) \in \mathrm{H}^1_{\mathrm{ord}}(K(m)_q, T_{f,n})$ for $q \mid \Delta\ell$; and
\item $\mathrm{res}_q( \kappa_{\mathcal{D}}(\ell)_m ) \in \mathrm{H}^1_{\mathrm{ord}}(K(m)_q, T_{f,n})$ for $q \mid p$.
\end{enumerate}
Part (1) comes from the fact that $J^{\Delta\ell}(N^+)$ has good reduction at primes not dividing $\ell \Delta N^+$ and $K(m)$ is unramified at primes not dividing $m$.
Part (2) comes from Proposition \ref{prop:local_computation}.(1) for $q \vert N^+$. The elephantine condition is used here.
Part (3) comes from the following argument.

Consider the following horizontally exact commutative diagram
\[
\xymatrix{
0 \ar[r] & \mathrm{H}^1_{\mathrm{ur}}(K(p^r)_q, T_{f,n}) \ar[r]^-{\mathrm{inf}} \ar@{-->}[d]_-{\mathrm{res}}^-{\simeq} & \mathrm{H}^1(K(p^r)_q, T_{f,n}) \ar[r]^-{\mathrm{res}} \ar[d]_-{\mathrm{res}}^-{\simeq} & \mathrm{H}^1(I_{K(p^r)_q}, T_{f,n}) \ar[d]_-{\mathrm{res}}^-{\simeq} \\
0 \ar[r] & \mathrm{H}^1_{\mathrm{ur}}(K(m)_q, T_{f,n}) \ar[r]^-{\mathrm{inf}} & \mathrm{H}^1(K(m)_q, T_{f,n}) \ar[r]^-{\mathrm{res}} & \mathrm{H}^1(I_{K(m)_q}, T_{f,n}) .
}
\]
Then the latter two vertical restriction maps are isomorphism because $K(m)/K(p^r)$ is an extension of degree prime-to-$p$ and the former vertical isomorphism comes from these two isomorphisms.
Since $K(p^r)$ is unramified at $q$, it reduces to Part (1).

Part (4) comes from the fact that $J^{\Delta\ell}(N^+)$ has purely toric reduction at primes dividing $\Delta\ell$. 
Part (5) comes from the exactly same argument given in \cite[Proof of Proposition 4.7]{chida-hsieh-main-conj}.
For a more detail, see \cite[Lemma 7.16]{longo-hilbert-modular-case} and \cite[Proposition 4.7]{chida-hsieh-main-conj}.
\end{proof}

\subsection{Reduction to Gross points} \label{subsec:reduction_to_gross_pts}
Reminding the definition of $\mathrm{red}_{\mathcal{V}}$ in $\S$\ref{subsec:cm_points} and the explicit definition of Gross and Heegner points as in \cite[$\S$4.3.1 and (4.8)]{chida-hsieh-main-conj}, respectively, Heegner points map to Gross points following diagram (\ref{eqn:comm_diagram_1st_explicit}) in $\S$\ref{subsubsec:graph}.
\[
\xymatrix{
P_{m_0}(a) \ar@{|->}[r]^-{cl} \ar@{|->}[d]_-{cl_{\mathcal{V}}} & cl(P_{m_0}(a)) \ar@{|->}[d]^-{\mathrm{red}_\ell}\\
\psi_{m_0}(a) \ar@{|->}[r]^-{\simeq} & \psi_{m_0}(a) 
}
\]
where $\mathrm{red}_\mathcal{V} \circ \iota_\ell  = cl_{\mathcal{V}}$.
\begin{rem}
In some sense, the first explicit reciprocity law is a deformation of this reduction to group rings (and the Iwasawa algebra) following Theorem \ref{thm:reduction_diagram}.(2). 
This reduction of Heegner points to Gross points is the Euler system interpretation of the original form of Jochnowitz congruences \cite[Theorem 6.1]{bertolini-darmon-jochnowitz-1999}. The connection can be made via the Gross-Zagier formula and the Gross-Walspurger-Zhang formula. 
\end{rem}
\section{Explicit reciprocity laws} \label{sec:explicit_reciprocity}
We review two explicit reciprocity laws \`{a} la Bertolini-Darmon, which connect the cohomology classes arising from Heegner points and theta elements arising from Gross points.
\subsection{$n$-admissible theta elements and the connection with theta elements} \label{subsec:n-admissible_theta}
Let $\mathcal{D} = (\Delta, g)$ be an $n$-admissible form. Then we define an element in $\mathcal{O}_n[G_{m_0}]$ depending on $\mathcal{D}$ and $\psi$ by
$$\widetilde{\theta}_{m_0}(\mathcal{D}) :=  \sum_{a}  g (\psi_{m_0}(a)) [a]_{m_0} \in \mathcal{O}_n[G_{m_0}] $$
as in $\S$\ref{sec:theta_elements}.
Then we define the \textbf{$n$-admissible theta element associated to $\mathcal{D}$ and $\psi$} by
the image of $$\frac{1}{\alpha^r} \cdot \theta_{m_0}(\mathcal{D})$$ in $\mathcal{O}_n[\Gamma_m]$.
Define the \textbf{$n$-admissible Bertolini-Darmon element associated to $\mathcal{D}$} by
$$L_p(K(m), \mathcal{D}) := \theta_m(\mathcal{D}) \cdot \theta^*_m(\mathcal{D}) \in \mathcal{O}_n[\Gamma_m].$$

\begin{prop}[Existence of $n$-admissible forms; {\cite[Proposition 6.13]{chida-hsieh-main-conj}}] \label{prop:n-admissible-theta}
If $\Delta = N^-$, there exists an $n$-admissible form $\mathcal{D}^f_n = (N^-, f_{\alpha,n})$ such that
$$\theta_m(\mathcal{D}^f_n) \equiv \theta_{m}(f_{\alpha}) \Mod{\varpi^n} $$
in $\mathcal{O}_n[\Gamma_m]$ up to a unit in $\mathcal{O}_n$.
\end{prop}
\begin{rem}[on the existence of $n$-admissible forms]
All the $n$-admissible forms in this article come from $\mathcal{D}^f_n$ and its level raising at two $n$-admissible primes ($\S$\ref{subsec:level_raising_at_two_prime}). Thus, the Euler system argument itself is non-empty.
\end{rem}
\subsection{The first explicit reciprocity law}
Let $\ell$ be an $n$-admissible prime.
The first explicit reciprocity law connects the cohomology classes arising from Heegner points on the Shimura curve of discriminant $\Delta\ell$ and Gross points on the Gross curve of discriminant $\Delta$.

\begin{thm}[The first explicit reciprocity law; {\cite[Theorem 4.1]{bertolini-darmon-imc-2005}, \cite[Theorem 5.1]{chida-hsieh-main-conj}}] \label{thm:first_explicit_reciprocity}
Assume Condition CR.
If $\ell$ is an $n$-admissible prime, then $$v_\ell (\kappa_{\mathcal{D}}(\ell)_m) = 0$$
in $\mathrm{H}^1_{\mathrm{fin}} (K(m)_\ell, T_{f,n} ) \simeq \mathcal{O}_n[\Gamma_m]$, 
and the equality
$$\partial_\ell (\kappa_{\mathcal{D}}(\ell)_m) = \theta_{m}(\mathcal{D})$$ holds under identifications $\mathrm{H}^1_{\mathrm{ord}} (K(m)_\ell, T_{f,n} ) \simeq \mathcal{O}_n[\Gamma_m]$ up to multiplication by $\mathcal{O}^\times_n$ and $\Gamma_m$.
\end{thm}
\begin{rem}
The equality depends on the choices of the embedding $\iota_p$ and the isomorphism 
$$\mathrm{Ta}_p\left( J^{\Delta\ell}(N^+p)\right)_{\mathcal{O}} / \mathcal{I}^{\Delta\ell}_g \simeq T_{f,n},$$
and different choices result in the ```up to multiplication by $\mathcal{O}^\times_n$ and $\Gamma_m$"-part.
\end{rem}
\begin{proof}
The reduction of Heegner points to Gross points in $\S$\ref{subsec:reduction_to_gross_pts}, the diagram in Theorem \ref{thm:reduction_diagram}.(2), and the duality of the spaces of modular forms with $\psi_g$ directly imply the conclusion. The explicit computation is given in \cite[Proof of Theorem 5.1]{chida-hsieh-main-conj}.
\end{proof}

\subsection{The second explicit reciprocity law}
Let $\ell_1$ and $\ell_2$ be $n$-admissible primes.
\begin{thm}[Second explicit reciprocity law] \label{thm:second_explicit_reciprocity}
Assume Condition CR and Condition Ihara.
Then there exists an $n$-admissible form $\mathcal{D}^{\ell_1\ell_2}=(\Delta\ell_1\ell_2, g_{\ell_1\ell_2})$ such that
$$v_{\ell_1} (\kappa_{\mathcal{D}}(\ell_2))_m = v_{\ell_2} (\kappa_{\mathcal{D}}(\ell_1))_m = \theta_m(\mathcal{D}^{\ell_1\ell_2}) $$
holds under the identifications
$\mathrm{H}^1_{\mathrm{fin}}(K(m)_{\ell_1}, T_{f,n}) \simeq \mathrm{H}^1_{\mathrm{fin}}(K(m)_{\ell_2}, T_{f,n})
\simeq \mathcal{O}_n[\Gamma_m]$ up to multiplication by $\mathcal{O}^\times_n$ and $\Gamma_m$.
\end{thm}
\begin{proof}
Using the moduli interpretation of Shimura curves $M^{\Delta\ell_1\ell_2}(N^+p)$ as in \cite[$\S$5.3]{chida-hsieh-main-conj}, we can observe
$$P^{\Delta\ell_1}_{m_0}(a) \Mod{(\ell_2)^2} = \psi_{m_0}(a) \tau$$
under the identification
$$\mathfrak{J}^{\Delta\ell_1}(N^+p)(\mathbb{F}_{(\ell_2)^2})^{ss} \simeq \mathbb{Z} [ B^{\Delta\ell_1\ell_2, \times} \backslash \widehat{B}^{\Delta\ell_1\ell_2, \times} / \widehat{R}^\times_{N^+p} ] .$$
Using this identity, the computation is straightforward.
\end{proof}
\section{The supply of $n$-admissible primes and $n$-admissible sets} \label{sec:chebotarev}
We focus on the infinitude of certain $n$-admissible primes with help of the Chebotarev density theorem. It is a generalization of
\cite[Lemma 2.21]{bertolini-darmon-derived-heights-1994}, \cite[Theorem 3.2]{bertolini-darmon-imc-2005}, and \cite[Theorem 6.3]{chida-hsieh-main-conj} for our purpose,~i.e. not over $K$ but over $K(m')$. It is one of the key technical improvements to remove the difficulties in the generalization of the Euler system argument. If $m' = 1$, the statement becomes the original statement due to Assumption \ref{assu:exact_degree}.
\begin{thm} \label{thm:chebotarev}
Let $n_0 \leq n$ be positive integers. 
Let $m'$ be a positive integer with $(m',Np) = 1$.
Let $s$ be a non-zero element of $\mathrm{H}^1(K(m'), A_{f,n_0})$.
Assume that $A_{f,1}$ has big image (Assumption \ref{assu:bigimage}).
Then there exist infinitely many $n$-admissible primes $\ell$ such that $\partial_\ell (s) = 0$ and the map
$$v_\ell: \langle s \rangle \to \mathrm{H}^1_{\mathrm{fin}}(K(m')_\ell, A_{f,n_0})$$
is injective, where 
$\langle s \rangle$ is the $\mathcal{O}$-submodule of $\mathrm{H}^1(K(m'), A_{f,n_0})$ generated by $s$.
\end{thm}
\begin{proof}
It suffices to consider the case $s \in \mathrm{H}^1(K(m')_\ell, A_{f,1})$.
Let $\overline{\rho}$ be the residual representation, and $F := \overline{\mathbb{Q}}^{\mathrm{ker}(\overline{\rho})}$.
Since $\overline{\rho}$ is unramified outside $Np$ and $(\mathrm{disc}(K)m',Np) = 1$, $K(m')$ and $F$ are linearly disjoint. Let $M(m') = K(m')F$ and $M = KF$ be the compositum fields. Then we have
\[
\xymatrix{
\mathrm{Gal}(M(m')/\mathbb{Q})  = \mathrm{Gal}(K(m')/\mathbb{Q}) \times \mathrm{Gal}(F/\mathbb{Q}), & \mathrm{Gal}(M/\mathbb{Q})  = \mathrm{Gal}(K/\mathbb{Q}) \times \mathrm{Gal}(F/\mathbb{Q}) .
}
\]
Since 
$\mathrm{Gal}(K(m')/\mathbb{Q}) = \mathrm{Gal}(K(m')/K) \rtimes \mathrm{Gal}(K/\mathbb{Q})$, the element can be written as $(\gamma^i_{m'}, \tau^j)$ where $\tau$ is the complex conjugate, $i \in \lbrace 0, \cdots, m'-1 \rbrace$, and $j \in \lbrace 0,1 \rbrace$.
We identify
\begin{align*}
\mathrm{Gal}(M(m')/\mathbb{Q})  = \mathrm{Gal}(K(m')/\mathbb{Q}) \times \mathrm{Gal}(F/\mathbb{Q}) & \subseteq \mathrm{Gal}(K(m')/\mathbb{Q}) \times \mathrm{Aut}_{\mathbb{F}}(A_{f,1}) \\
 \mathrm{Gal}(M/\mathbb{Q})  = \mathrm{Gal}(K/\mathbb{Q}) \times \mathrm{Gal}(F/\mathbb{Q}) & \subseteq \mathrm{Gal}(K/\mathbb{Q}) \times \mathrm{Aut}_{\mathbb{F}}(A_{f,1}),
\end{align*}
respectively. Then we may write
\begin{itemize}
\item an element of $\mathrm{Gal}(M(m')/\mathbb{Q})$ as a triple $(\gamma^i_{m'}, \tau^j, \sigma)$ with $\sigma \in \mathrm{Aut}_{\mathbb{F}}(A_{f,1})$ and
\item an element of $\mathrm{Gal}(M/\mathbb{Q})$ as a pair $(\tau^j, \sigma)$ with $\sigma \in \mathrm{Aut}_{\mathbb{F}}(A_{f,1})$.
\end{itemize} 
Since $\overline{\rho}$ has big image, $\mathrm{Gal}(F/\mathbb{Q})\simeq\overline{\rho}(G_{\mathbb{Q}})$ contains $-I_2$. 
Using the inflation-restriction sequence, we have
$$\mathrm{H}^1(\mathrm{Gal}(M(m')/K(m') ), A_{f,1})  = 0$$
since $- I_2$ does not fix any non-zero vector in $A_{f,1}$
Therefore, the restriction map 
\[
\xymatrix@R=0em{
\mathrm{H}^1(K(m'), A_{f,1}) \ar[r] & \mathrm{H}^1(M(m'), A_{f,1}) \ar@{=}[r]& \mathrm{Hom}(G_{M(m')}, A_{f,1}) \\
s \ar@{|->}[r] & \overline{s} 
}
\]
is injective.
Let $M(m')_{s}$ be the non-trivial extension of $M$ cut out by the image $\overline{s}$ of $s$ under the restriction to 
$\mathrm{H}^1(M(m'), A_{f,1}) = \mathrm{Hom}(G_{M(m')}, A_{f,1})$; thus, $M(m')_s/M(m')$ is an abelian extension.
Let $$A_s := \mathrm{Gal}(M(m')_s/M(m')) \subseteq A_{f,1}$$ under $\overline{s}$.
Without loss of generality, we assume that $s$ belongs to a $\tau$-isotypic subspace such that $\tau s = \delta s$ with $\delta = \pm 1$. Then, under the identification $A_s \subseteq A_{f,1}$, $A_s$ admits the action of  $\sigma \in \mathrm{Gal}(F/\mathbb{Q})$ via $\overline{\rho}$ and $\tau \in \mathrm{Gal}(K/\mathbb{Q})$ by the formula
$$(\tau^j, \sigma)(v) = \delta^j\overline{\rho}(\sigma)v $$
where $v \in A_s$.
However, $M(m')_s/\mathbb{Q}$ may not be Galois since $A_s$ may not be stable under the action of $\gamma^i_{m'} \in \mathrm{Gal}(K(m')/K)$.
We explicitly construct the Galois closure $L$ of $M(m')_s/\mathbb{Q}$ and also show that 
\begin{enumerate}
\item $A_L := \mathrm{Gal}(L/M(m'))$ is an abelian extension and
\item $\mathrm{Gal}(L/\mathbb{Q})$ can be written as $\mathrm{Gal}(L/M(m')) \rtimes \mathrm{Gal}(M(m')/\mathbb{Q}) $.
\end{enumerate}
Since $\mathrm{H}^1(K(m'), A_{f,1})$ admits the action of $\Gamma_{m'} = \mathrm{Gal}(K(m')/K)$, we may consider the $\Gamma_{m'}$-orbit of $s$
$$\Gamma_{m'}s = \lbrace \gamma^i_{m'}s : i = 0, \cdots m'-1 \rbrace .$$

Consider the extension $M(m')_{\gamma^i_{m'}s}$ of $\gamma^{-i}_{m'}M(m') = M(m')$ using $\overline{\gamma^i_{m'}s}$ with
$$A_{\gamma^i_{m'}s} := \mathrm{Gal}(M(m')_{\gamma^i_{m'}s}/M(m')) \subseteq A_{f,1}$$
under $\overline{\gamma^i_{m'}s}$.
Then $A_{\gamma^i_{m'}s} \subseteq A_{f,1}$ also admits the action of $\mathrm{Gal}(F/\mathbb{Q})$ via $\overline{\rho}$ as in the case of $A_s$ since the actions of $\mathrm{Gal}(F/\mathbb{Q})$ and $\mathrm{Gal}(K(m')/\mathbb{Q})$ commute with each other.

Consider the set of all such Galois groups over $M(m')$
$$\lbrace A_{\gamma^i_{m'}s} = \mathrm{Gal}(M(m')_{\gamma^i_{m'}s}/M(m')) : i = 0, \cdots m'-1 \rbrace  .$$
Then this set admits the action of $\Gamma_{m'}$ and $\tau$ by
\[
\xymatrix{
\gamma_{m'} \left( A_{\gamma^i_{m'}s} \right)  = A_{\gamma^{i+1}_{m'}s} , & \tau \left( A_{\gamma^i_{m'}s} \right)  = A_{\gamma^{-i}_{m'}s} .
}
\]
We define $L$ by the compositum of $M(m')_{\gamma^i_{m'}s}$ for all $i = 0, \cdots m'-1$
and then
$A_L := \mathrm{Gal}(L/M(m'))$ is the subgroup of $A_{f,1}$ generated by $A_{\gamma^i_{m'}s}$ for all $i = 0, \cdots m'-1$.
Then $L$ is the Galois closure of $M(m')_s/\mathbb{Q}$ by construction, and $A_L$ is abelian since $A_L \subset A_{f,1}$.
Also, $A_L$ admits the action of $\mathrm{Gal}(K(m')/K)$ by construction as well as the action of $\mathrm{Gal}(F/\mathbb{Q})$ via $\overline{\rho}$ and $\mathrm{Gal}(K/\mathbb{Q})$ via $\tau$.
Thus, $L$ is Galois over $\mathbb{Q}$ and $\mathrm{Gal}(L/\mathbb{Q}) = A_L \rtimes \mathrm{Gal}(M(m')/\mathbb{Q})$.

Each element of $\mathrm{Gal}(L/\mathbb{Q})$ can be written as a quadruple $(v', \gamma^i_{m'}, \tau^j, \sigma)$
where $v' \in A_L$, $\gamma^i_{m'} \in \Gamma_{m'}$, $\tau^j \in \mathrm{Gal}(K/\mathbb{Q})$, and $\sigma \in \mathrm{Gal}(F/\mathbb{Q})$.

Since $\overline{\rho}$ has big image, there exists an element $h \in \overline{\rho}(G_{\mathbb{Q}})$ such that $\mathrm{tr}(h) = \mathrm{det}(h) +1$ with $\mathrm{det}(h) \neq \pm 1 \in \mathbb{F}$ (\cite[Lemma 6.2]{chida-hsieh-main-conj}).
It implies that we can choose a quadruple $(v', 1, \tau, \sigma)$ as an element of $\mathrm{Gal}(L/\mathbb{Q})$ such that:
\begin{enumerate}
\item $\overline{\rho}(\sigma) \in \mathrm{GL}_2(\mathbb{F})$ has eigenvalue $\delta (=\pm 1)$ and $\lambda \in \mathbb{F}^\times$ where $\lambda \neq \pm 1 \Mod{\varpi}$ and the order of $\lambda$ is prime to $p$; 
\item the element $v' \in A_L$ satisfies that its image $v$ in $\mathrm{Gal}(M(m')_s/M(m'))$ is non-zero, and $v$ belongs to the $\delta$-eigenspace for $\sigma$ under $\overline{\rho} = \rho_1$. 
\end{enumerate}
Note that $\mathrm{Gal}(L/M(m')_s)$ is normal in $\mathrm{Gal}(L/M(m'))$ since $\mathrm{Gal}(M(m')_s/M(m'))$ is Galois.

By the Chebotarev density theorem, there exist infinitely many primes $\ell$ with $\ell \nmid Nm'p$ such that $\ell$ is unramified in $L/\mathbb{Q}$ and satisfies
$\mathrm{Frob}_\ell(L/\mathbb{Q})$ is the quadruple we have chosen; in fact, the set of such primes has positive density.
Then the restriction to $M(m')/\mathbb{Q}$ becomes the triple $(1, \tau, \sigma)$. It implies that $\ell$ is $n$-admissible.
Also, the computation in \cite[Page 889]{chida-hsieh-main-conj} yields $$v_\ell (s) =  \overline{s}( \mathrm{Frob}_{\mathfrak{l}}(L/M(m')) )_{\mathfrak{l} \mid \ell} \neq 0.$$
\end{proof}

The following diagram would be helpful to understand the proof of Theorem \ref{thm:chebotarev}:
\[
\xymatrix{
& & L  \ar@{-}[d] \ar@{-}[ddl]_-{A_L}\\
& & M(m')_s  \ar@{-}[dl]^-{A_s}\\
& M(m') = K(m')F \ar@{-}[d]_-{\Gamma_{m'}} \ar@{-}[dl]_-{\overline{\rho}(G_{\mathbb{Q}})}  \\
K(m') \ar@{-}[d]_-{\Gamma_{m'}} & M = KF \ar@{-}[d]_-{\langle \tau \rangle} \ar@{-}[dl]_-{\overline{\rho}(G_{\mathbb{Q}})} \\
K \ar@{-}[d]_-{\langle \tau \rangle} & F \ar@{-}[dl]^-{\overline{\rho}(G_{\mathbb{Q}})} \\
\mathbb{Q}
}
\]
\begin{rem}
After proving Theorem \ref{thm:chebotarev}, we noticed that a similar idea is given in \cite[Proposition 4.5 and Appendix A]{longo-vigni-vanishing}.
\end{rem}

Let $\Delta$ be a square-free integer such that $\Delta / N^-$ is a product of $n$-admissible primes.
\begin{defn}[$n$-admissible set] \label{defn:n-admissible-set}
A finite set $\mathfrak{S}$ of primes is said to be an \textbf{$n$-admissible set for $f$ and $m'$} if
\begin{enumerate}
\item All $\ell \in \mathfrak{S}$ are $n$-admissible for $f$
\item The map $\mathrm{Sel}_{\Delta}(K(m'), A_{f,n}) \to \bigoplus_{\ell \in \mathfrak{S}} \mathrm{H}^1_{\mathrm{fin}} (K(m')_\ell, A_{f,n})$ is injective.
\end{enumerate}
\end{defn}

For a given finite set of $n$-admissible primes satisfying the statement in Theorem \ref{thm:chebotarev}, we are always able to enlarge the set to an $n$-admissible set for $f$ and $m'$. This is one of the standard applications of Theorem \ref{thm:chebotarev}.
\begin{prop}
Any finite collection of $n$-admissible primes can be enlarged to an $n$-admissible set for $f$ and $m'$.
\end{prop}
\section{Control theorems and the freeness of compact Selmer groups} \label{sec:control_freeness}
We prove the freeness theorem of $\Delta$-ordinary $S$-relaxed Selmer groups with $n$-admissible set $\mathfrak{S}$ where $S = \prod_{q \in \mathfrak{S}} q$. The original argument depende heavily on the properties of $p$-extensions. In order to generalize to the general cyclic case, we use the decomposition of modules over group rings.
\subsection{Controlling discrete Selmer groups}
Let $L/K$ be a subextension of $K(m'p^\infty)$.
Let $S$ be a square-free product of $n$-admissible primes.

\begin{lem} \label{lem:control_galois}
Assume that the residual representation $A_{f,1}$ of $G_{\mathbb{Q}}$ is absolutely irreducible.
Then the restriction map
$$\mathrm{H}^1(K, A_{f,n}) \to \mathrm{H}^1(L, A_{f,n})^{\mathrm{Gal}(L/K)}$$
is an isomorphism.
\end{lem}
\begin{proof}
Since $K(m'p^r)$ is Galois over $\mathbb{Q}$ for all $r$, the residual irreducibility of $\overline{\rho}$ implies that $(T_{f,n})^{G_{K(m'p^r)}} = 0$. Then the isomorphism follows from the inflation-restriction sequence. For any subextension of $K(m'p^r)$, the proof is identical.
\end{proof}
\begin{lem}[Exact control theorem] \label{lem:control_galois_selmer}
Assume that 
\begin{itemize}
\item the residual representation $A_{f,1}$ of $G_\mathbb{Q}$ is absolutely irreducible, and
\item $f$ satisfies Condition PO.
\item $f$ is $N^+$-minimal.
\end{itemize}
Then the restriction map
$$\mathrm{Sel}^{S}_{\Delta}(K, A_{f,n}) \to \mathrm{Sel}^{S}_{\Delta}(L, A_{f,n})^{\mathrm{Gal}(L/K)}$$
is an isomorphism.
\end{lem}
\begin{proof}
By Lemma \ref{lem:control_galois}, we know the map $\mathrm{Sel}^{S}_{\Delta}(K, A_{f,n}) \to \mathrm{Sel}^{S}_{\Delta}(L, A_{f,n})^{\mathrm{Gal}(L/K)}$ is injective. To check the surjectivity, we check the following local statements:
\begin{enumerate}
\item $\mathrm{H}^1(K_\ell, A_{f,n} / F^+_\ell A_{f,n}) \to \mathrm{H}^1(L_\ell, A_{f,n} / F^+_\ell A_{f,n})$ is injective for $\ell \mid p\Delta$; and
\item $\mathrm{H}^1(K^{\mathrm{ur}}_\ell, A_{f,n}) \to \mathrm{H}^1(L^{\mathrm{ur}}_\ell, A_{f,n})$ is injective for $\ell \nmid p\Delta$.
\end{enumerate}
Part (1) for $\ell \mid \Delta$ follows from $K_\lambda = L_\lambda$ for any prime $\lambda$ of $L$ above $\ell$ which is inert in $K/\mathbb{Q}$. Part (1) for $\ell = p$ comes from Lemma \ref{lem:local_condition_at_p}.

Since $L/K$ is unramified outside $m = m'p^r$ and anticyclotomic, we have $K^{\mathrm{ur}}_\lambda = L^{\mathrm{ur}}_\lambda$ for each place $\lambda$ of $L$ above $\ell$ not dividing $p m'$.

For a place $\lambda$ of $L$ dividing $m'$,
the kernel of the map is $\mathrm{H}^1 ( L^\mathrm{ur}_\lambda / K^\mathrm{ur}_\lambda, (A_{f,n})^{G_{L^\mathrm{ur}_\lambda}} )$.
Since we have $L^\mathrm{ur}_\lambda = \left( L \cap K(\ell^r) \right)^\mathrm{ur}_\lambda$ where $\lambda$ divides a prime $\ell$ with $\ell^r \Vert m'$,
$\mathrm{Gal}(L^\mathrm{ur}_\lambda / K^\mathrm{ur}_\lambda)$ is an $\ell$-group, which is prime to $p$. By \cite[Corollary 1, Page 130]{serre-local-fields} and the identification with Tate cohomology, the kernel vanishes.
\end{proof}

\begin{lem} \label{lem:control_mod_pn}
Assume that the residual representation $A_{f,1}$ of $G_\mathbb{Q}$ is absolutely irreducible.
Then
$$\mathrm{H}^{1}(L, A_{f,n}) = \mathrm{H}^{1}(L, A_{f})[\varpi^n].$$
\end{lem}
\begin{proof}
See \cite[Proposition 1.9]{chida-hsieh-main-conj}.
\end{proof}

\begin{lem}[Exact control theorem of reduction mod $\varpi^n$] \label{lem:control_mod_pn_selmer}
Assume that
\begin{itemize}
\item the residual representation $A_{f,1}$ of $G_{\mathbb{Q}}$ is absolutely irreducible,
\item $(\overline{\rho}, \Delta)$ satisfies Condition CR,
\item $f$ satisfies Condition PO, and
\item $f$ is $N^+$-minimal.
\end{itemize}
Then
$$\mathrm{Sel}^{S}_{\Delta}(L, A_{f,n}) = \mathrm{Sel}^{S}_{\Delta}(L, A_{f})[\varpi^n].$$
\end{lem}
\begin{proof}
See \cite[Proposition 1.9]{chida-hsieh-main-conj}.
\end{proof}
The following corollary is the combination of Lemma \ref{lem:control_galois} and Lemma \ref{lem:control_mod_pn}.
\begin{cor} \label{cor:control_galois_mod_pn}
Assume that the residual representation $A_{f,1}$ of $G_\mathbb{Q}$ is absolutely irreducible.
Let $\mathfrak{m}_{L,n} = (\varpi, \gamma_L - 1)$ be a maximal ideal of $\mathcal{O}_n \llbracket \mathrm{Gal}(L/K)\rrbracket$ where $\gamma_L$ is a generator of $\mathrm{Gal}(L/K)$.
Then the restriction map
$$\mathrm{H}^1(K, A_{f,1}) \to \mathrm{H}^1(L, A_{f,n})[\mathfrak{m}_{L,n}]$$
is an isomorphism.
\end{cor}

\subsection{The stability of $n$-admissible sets and the Poitou-Tate sequence}
In this subsection, the assumptions of Lemma \ref{lem:control_galois_selmer} are in force.

Let $L/K$ be a subextension of $K(m)$ and $\mathfrak{S}$ an $n$-admissible set for $f$ and $m'$.
Let $L(m') := L \cap K(m')$. 

The following lemma shows that the the concept of $n$-admissible sets is stable under base change of $p$-power degree.
\begin{lem}\label{lem:injectivity}
The natural map
$$\mathrm{Sel}_{\Delta}(L, A_{f,n}) \to \bigoplus_{\ell \in \mathfrak{S}} \mathrm{H}^1_{\mathrm{fin}}(L_\ell, A_{f,n})$$
is injective.
\end{lem}
\begin{proof}
Let $C$ be the kernel of the map. If $C$ is non-zero, then there exists a non-trivial element $s$ in $C$ fixed by $\mathrm{Gal}(L/L(m'))$ since $\mathrm{Gal}(L/L(m'))$ is a $p$-group with help of the orbit-stabilizer theorem.
Then $s$ belongs to $\mathrm{Sel}_{\Delta}(L, A_{f,n})^{\mathrm{Gal}(L/L(m'))} = \mathrm{Sel}_{\Delta}(L(m'), A_{f,n})$ due to Lemma \ref{lem:control_galois_selmer} with $S = 1$.
Then it contradicts the definition of $n$-admissible sets for $f$ and $m'$.
\end{proof}

Let $S = \prod_{q \in \mathfrak{S}}q$.
\begin{lem}[Poitou-Tate exact sequences for Selmer groups] \label{lem:poitou-tate}
We have an exact sequence
$$0 \to \mathrm{Sel}_{\Delta}(L, T_{f,n}) \to \mathrm{Sel}^{S}_{\Delta}(L, T_{f,n}) \to \bigoplus_{\ell \in \mathfrak{S}} \mathrm{H}^1_{\mathrm{sing}}(L_\ell, T_{f,n}) \to \mathrm{Sel}_{\Delta}(L, A_{f,n})^{\vee} \to 0$$
where $\mathrm{Sel}_{\Delta}(L, A_{f,n})^{\vee} = \mathrm{Hom}_{\mathcal{O}} \left( \mathrm{Sel}_{\Delta}(L, A_{f,n}), E/\mathcal{O} \right)$.
\end{lem}
\begin{proof}
The Poitou-Tate sequences for $\Delta$-ordinary Selmer groups and Lemma \ref{lem:injectivity} directly imply the conclusion.
See \cite[Lemma 6.7]{chida-hsieh-main-conj} for detail.
\end{proof}

\subsection{The freeness of compact Selmer groups} \label{subsec:freeness}
In this subsection, we generalize \cite[Proposition 6.8]{chida-hsieh-main-conj}, which is based on \cite[Theorem 3.2]{bertolini-darmon-derived-heights-1994}, to the general cyclic case in detail. The key idea is to apply the enhanced isotypic decomposition in $\S$\ref{subsec:enhanced_decomposition} for the reduction to the $p$-power case.
The assumptions of Lemma \ref{lem:control_galois_selmer} are in force as before.
\subsubsection{Statement}
Let $L/K$ be a subextension of $K(m)$ and $F/K$ be a subextension of $K$ in $L$.
Let $\mathfrak{S}$ be an $n$-admissible set for $f$ and $m'$.
Let $S = \prod_{q \in \mathfrak{S}}q$.
\begin{thm} \label{thm:freeness}
The Selmer group $\mathrm{Sel}^S_{\Delta}(K(m), T_{f,n})$ is free of rank $\#\mathfrak{S}$ over $\mathcal{O}_n[\Gamma_m]$ where $\Gamma_m = \mathrm{Gal}(K(m)/K)$.
\end{thm}
Due to Lemma \ref{lem:kurihara_decomposition_freeness}, it suffices to prove the following statement.
\begin{prop} \label{prop:freeness_p-part}
Let $\omega: \mathrm{Gal}(K(m')/K) \to \mathcal{O}_\omega$ be a character.
Then the Selmer group 
$$\mathrm{Sel}^S_{\Delta}(K(m), T_{f,n}) \otimes_\omega \mathcal{O}_\omega$$
 is free of rank $\#\mathfrak{S}$ over $\mathcal{O}_{\omega,n}[\Gamma_{p^r}]$.
\end{prop}

We prove Proposition \ref{prop:freeness_p-part} in the next subsections and describe the consequences in the the last subsection.
\subsubsection{Step 1: Proof of the case for $r=1$ and $n=1$}
Let $\Gamma_p = \mathrm{Gal}(K(m'p)/K(m'))$.
Since $\Gamma_p$ trivially acts on $\mathcal{O}_\omega$, we have
$$\left( \mathrm{Sel}^S_{\Delta}(K(m'p), T_{f,1}) \otimes_\omega  \mathcal{O}_\omega \right)^{\Gamma_p}
= \mathrm{Sel}^S_{\Delta}(K(m'p), T_{f,1})^{\Gamma_p} \otimes_\omega  \mathcal{O}_\omega.$$
By the control theorem (Lemma \ref{lem:control_galois_selmer} with $S$), it coincides with
$$\mathrm{Sel}^S_{\Delta}(K(m'), T_{f,1})\otimes_\omega  \mathcal{O}_\omega.$$
Due to the Poitou-Tate sequence (Lemma \ref{lem:poitou-tate} with $L = K(m')$),
$\mathrm{Sel}^S_{\Delta}(K(m'), T_{f,1})$ and $\bigoplus_{\ell \in \mathfrak{S}} \mathrm{H}^1_{ \mathrm{sing} } (K(m')_{\ell}, T_{f,1})$ have the same cardinality.

From now on, we focus on the vector space structure on them after tensoring $\mathcal{O}_\omega$ via $\omega$. 
Due to Proposition \ref{prop:local_computation}.(4), we have
$$\bigoplus_{\ell \in \mathfrak{S}} \mathrm{H}^1_{ \mathrm{sing} } (K(m')_{\ell}, T_{f,1}) \simeq \bigoplus_{\ell \in \mathfrak{S}} \mathbb{F}[\Gamma_{m'}].$$
Consider the following vertically exact commutative digram induced from the Poitou-Tate sequence (Lemma \ref{lem:poitou-tate}), Proposition \ref{prop:local_computation}.(4), and the enhanced isotypic decomposition (Lemma \ref{lem:kurihara_decomposition_freeness}):
{\tiny
\[
\xymatrix{
0 \ar[d] & & 0 \ar[d] & 0 \ar@{-->}[d]^-{(?)}\\
\mathrm{Sel}_{\Delta}(K(m'), T_{f,n}) \ar[d] \ar[rr]^-{\simeq} & & \bigoplus_\omega \mathrm{Sel}_{\Delta}(K(m'), T_{f,n}) \otimes_\omega \mathcal{O}_\omega \ar[d] \ar[r]^-{\mathrm{proj}_\omega} & \mathrm{Sel}_{\Delta}(K(m'), T_{f,n}) \otimes_\omega \mathcal{O}_\omega \ar[d]\\
\mathrm{Sel}^{S}_{\Delta}(K(m'), T_{f,n}) \ar[d] \ar[rr]^-{\simeq} & & \bigoplus_\omega \mathrm{Sel}^{S}_{\Delta}(K(m'), T_{f,n}) \otimes_\omega \mathcal{O}_\omega \ar[d] \ar[r]^-{\mathrm{proj}_\omega} & \mathrm{Sel}^{S}_{\Delta}(K(m'), T_{f,n}) \otimes_\omega \mathcal{O}_\omega \ar[d]\\
\bigoplus_{\ell \in \mathfrak{S}} \mathrm{H}^1_{\mathrm{sing}}(K(m')_\ell, T_{f,n}) \ar[d] \ar[r]^-{\simeq} & \bigoplus_{\ell \in \mathfrak{S}} \mathbb{F}[\Gamma_{m'}] \ar[r]^-{\simeq} & \bigoplus_\omega \bigoplus_{\ell \in \mathfrak{S}} \mathbb{F}[\Gamma_{m'}] \otimes_\omega \mathcal{O}_\omega \ar[d] \ar[r]^-{\mathrm{proj}_\omega}  & \bigoplus_{\ell \in \mathfrak{S}} \mathbb{F}[\Gamma_{m'}] \otimes_\omega \mathcal{O}_\omega \ar[d]\\
\mathrm{Sel}_{\Delta}(K(m'), A_{f,n})^{\vee}  \ar[d] \ar[rr]^-{\simeq} & & \bigoplus_\omega \mathrm{Sel}_{\Delta}(K(m'), A_{f,n})^{\vee} \otimes_\omega \mathcal{O}_\omega \ar[d] \ar[r]^-{\mathrm{proj}_\omega}  &\mathrm{Sel}_{\Delta}(K(m'), A_{f,n})^{\vee} \otimes_\omega \mathcal{O}_\omega \ar[d]\\
0 & & 0 & 0\\
}
\]
}
where $\mathrm{proj}_\omega$ is the projection to the $\omega$-component. By the cardinality consideration, 
$$\mathrm{Sel}_{\Delta}(K(m'), T_{f,n}) \otimes_\omega \mathcal{O}_\omega \to \mathrm{Sel}^{S}_{\Delta}(K(m'), T_{f,n}) \otimes_\omega \mathcal{O}_\omega$$ is injective (for any $\omega$), so the (?)-part in the diagram becomes completed. Thus, we obtain
\begin{align*}
\dim_{\mathbb{F}_\omega} \left( \mathrm{Sel}^S_{\Delta}(K(m'), T_{f,1}) \otimes_\omega  \mathcal{O}_\omega \right) & = \dim_{\mathbb{F}_\omega} \left( \bigoplus_{\ell \in \mathfrak{S}} \mathbb{F}[\Gamma_{m'}] \otimes_\omega  \mathcal{O}_\omega \right) \\
& = \#\mathfrak{S}
\end{align*}
where $\mathbb{F}_\omega$ is the residue field of $\mathcal{O}_\omega$.

The group ring $\mathbb{F}_\omega[\Gamma_p]$ is isomorphic to the local ring $\mathbb{F}_\omega[\varepsilon]/(\varepsilon^p)$ via map $\gamma_p \mapsto \varepsilon +1$ where $\gamma_p$ is a generator of $\Gamma_p$.

Every finitely generated $\mathbb{F}_\omega[\Gamma_p]$-module $M$ can be written as a direct sum of cyclic $\mathbb{F}_\omega[\Gamma_p]$-modules and the number of summands is $\dim_{\mathbb{F}_\omega} \left( M^{\Gamma_p} \right)$. Hence, we can write
$$\mathrm{Sel}^S_{\Delta}(K(m'p), T_{f,1})^{\Gamma_p} \otimes_{\omega}\mathcal{O}_\omega = V_1 \oplus \cdots \oplus V_{\#\mathfrak{S}}$$ 
where each $V_i$ is a cyclic module over $\mathbb{F}_\omega[\Gamma_p]$. Let $n_i = \mathrm{dim}_{\mathbb{F}_\omega}( V_i)$.

Applying the same argument with $L = K(m'p)$, we have
$$p \cdot \#\mathfrak{S} = \mathrm{dim}_{\mathbb{F}_\omega} \left( \bigoplus_{\ell \in \mathfrak{S}} \mathrm{H}^1_{ \mathrm{sing} } (K(m'p)_{\ell}, T_{f,1}) \otimes_\omega \mathcal{O}_\omega  \right) =   \mathrm{dim}_{\mathbb{F}_\omega} \left(  \mathrm{Sel}^S_{\Delta}(K(m'p), T_{f,1}) \otimes_\omega \mathcal{O}_\omega \right) = \sum^{\#\mathfrak{S}}_{i=1} n_i .$$
Since $n_i \leq p$ for all $i$, we have $n_i = p$ for all $i$. It immediately shows that $\mathrm{Sel}^S_{\Delta}(K(m'p), T_{f,1})\otimes_\omega \mathcal{O}_\omega$ is free of rank $\#\mathfrak{S}$ over $\mathbb{F}_\omega[\Gamma_p]$.

\subsubsection{Step 2: Proof of the case for arbitrary $r$ and $n=1$}
By the argument of Step 1, we still have
$$\dim_{\mathbb{F}_\omega} \left( \mathrm{Sel}^S_{\Delta}(K(m'p^r), T_{f,1})^{\Gamma_{p^r}} \otimes_\omega \mathcal{O}_\omega\right) = 
\dim_{\mathbb{F}_\omega} \left( \mathrm{Sel}^S_{\Delta}(K(m'), T_{f,1})\otimes_\omega \mathcal{O}_\omega\right) = \#\mathfrak{S}.$$ 
Thus, it suffices to check that $\mathrm{Sel}^S_{\Delta}(K(m'p^r), T_{f,1}) \otimes_\omega \mathcal{O}_\omega$ is a free $\mathbb{F}[\Gamma_{p^r}]$-module.
By \cite[Theorem 6, Page 112]{atiyah-wall}, it suffices to prove 
$$\mathrm{H}^0_{\mathrm{Tate}}\left(\Gamma_{p^r}, \mathrm{Sel}^S_{\Delta}(K(m'p^r), T_{f,1}) \otimes_\omega \mathcal{O}_\omega\right) = 0$$
where $\mathrm{H}^q_{\mathrm{Tate}}(G,M)$ is the Tate cohomology group.
We use the induction on $r$.
Consider a subextension $K(m'p^{r-1})/K(m')$ with $$\Gamma_{p} = \mathrm{Gal}(K(m'p^r)/K(m'p^{r-1})) \simeq \mathbb{Z}/p\mathbb{Z}.$$
By Lemma \ref{lem:injectivity}, the set of places of $K(m'p^{r-1})$ lying above the places of $K(m')$ in $\mathfrak{S}$
is also an $n$-admissible set for $f$ and $m'p^{r-1}$. Thus, the same argument used in Step 1 shows that
$$\mathrm{H}^0_{\mathrm{Tate}}\left(\Gamma_{p}, \mathrm{Sel}^S_{\Delta}(K(m'p^r), T_{f,1}) \otimes_\omega \mathcal{O}_\omega\right) = 0 .$$
Here we have that $\mathrm{Sel}^S_{\Delta}(K(m'p^r), T_{f,1}) \otimes_\omega \mathcal{O}_\omega$ is free of rank $p^{r-1} \cdot \#\mathfrak{S}$ over $\mathbb{F}_\omega[\mathrm{Gal}(K(m'p^r)/K(m'p^{r-1}))]$.
By the induction hypothesis, we have
$$\mathrm{H}^0_{\mathrm{Tate}}\left(\mathrm{Gal}(K(m'p^{r-1})/K(m')), \mathrm{Sel}^S_{\Delta}(K(m'p^{r-1}), T_{f,1}) \otimes_\omega \mathcal{O}_\omega\right) = 0 .$$
Since the norm map is transitive, it completes a proof of Step 2.


\subsubsection{Step 3: Proof of the general case}
The following three lemmas immediately imply the general case.
\begin{lem}[{\cite[Lemma 3.3]{bertolini-darmon-derived-heights-1994}}] \label{lem:modp-freeness-selmer}
Let $\left(\mathrm{Sel}^S_{\Delta}(K(m), T_{f,n}) \otimes_{\omega} \mathcal{O}_\omega \right) [\varpi_\omega]$ be the $\varpi_\omega$-torsion submodule of $\mathrm{Sel}^S_{\Delta}(K(m), T_{f,n}) \otimes_{\omega} \mathcal{O}_\omega$. Then
$$\left(\mathrm{Sel}^S_{\Delta}(K(m), T_{f,n}) \otimes_{\omega} \mathcal{O}_\omega \right) [\varpi_\omega] \simeq \left( \mathbb{F}[\Gamma_m] \otimes_{\omega} \mathcal{O}_\omega \right)^{\#\mathfrak{S}} \simeq \left( \mathbb{F}_\omega[\Gamma_{p^r}]\right)^{\#\mathfrak{S}} .$$
\end{lem}
\begin{proof}
Using Lemma \ref{lem:control_mod_pn_selmer} and the enhanced isotypic decomposition, 
we have
$$\mathrm{Sel}^S_{\Delta}(K(m), T_{f,1})\otimes_{\omega} \mathcal{O}_\omega = \left(\mathrm{Sel}^S_{\Delta}(K(m), T_{f,n})\otimes_{\omega} \mathcal{O}_\omega \right)[\varpi_\omega].$$ 
Applying Step 2 to $\mathrm{Sel}^S_{\Delta}(K(m), T_{f,1}) \otimes_{\omega} \mathcal{O}_\omega$, we obtain the conclusion.
\end{proof}

\begin{lem}[{\cite[Lemma 3.4]{bertolini-darmon-derived-heights-1994}}]
The module $\mathrm{Sel}^S_{\Delta}(K(m), T_{f,n}) \otimes_\omega \mathcal{O}_\omega$ is a free module over $\mathcal{O}_{\omega,n} = \mathcal{O}_\omega/\varpi^n_\omega \mathcal{O}_\omega$.
\end{lem}
\begin{proof}
Let $m = [K(m):K]$. 
By Lemma \ref{lem:modp-freeness-selmer}, we have
$$\mathrm{Sel}^S_{\Delta}(K(m), T_{f,n}) \otimes_\omega \mathcal{O}_\omega \simeq \bigoplus_{i=1}^{p^r \cdot \#\mathfrak{S}} \mathcal{O}_{\omega, n_i}$$
with $n_i \leq n$.
By Lemma \ref{lem:poitou-tate} and the argument of Step 1, we have
$$\# \left( \mathrm{Sel}^S_{\Delta}(K(m), T_{f,n}) \otimes_\omega \mathcal{O}_\omega \right) = \# \left( \bigoplus_{\ell \in \mathfrak{S}} \mathbb{F}[\Gamma_{m}] \otimes_\omega \mathcal{O}_\omega \right) = \mathrm{Nm}(\varpi_\omega)^{p^r \cdot n \cdot \#\mathfrak{S}} .$$
Thus, we have $n_i = n$ for all $i$.
\end{proof}

\begin{lem}[{\cite[Lemma 3.5]{bertolini-darmon-derived-heights-1994}}]
Let $M$ be an $\mathcal{O}_n[\Gamma_{p^r}]$-module such that
\begin{enumerate}
\item $M[\varpi] \simeq \left(\mathbb{F}[\Gamma_{p^r}]\right)^{\#\mathfrak{S}}$, and
\item $M$ is $\mathcal{O}_n$-free.
\end{enumerate}
Then $M \simeq \left(\mathcal{O}_n[\Gamma_{p^r}]\right)^{\#\mathfrak{S}}$.
\end{lem}

\subsubsection{Consequences}

\begin{cor} \label{cor:surjectivity_of_cores}
Under the same assumptions of Theorem \ref{thm:freeness}, the natural map
$$\mathrm{Sel}^S_\Delta (K(m), T_{f,n}) \to \mathrm{Sel}^S_\Delta (K(m''), T_{f,n})$$
with $m'' \vert m$ is surjective.
\end{cor}
\begin{proof}
This is a direct consequence of the definition of the corestriction map, Theorem \ref{thm:freeness}, and the cardinality consideration.
\end{proof}
Taking projective limit with respect to $r$, we also have the following statement.
\begin{cor} \label{cor:Lambda-freeness}
Under the same assumptions of Theorem \ref{thm:freeness}, $\mathrm{Sel}^S_\Delta (K(m'p^\infty), T_{f,n})$ is free of rank $\#\mathfrak{S}$ over $\mathcal{O}_n\llbracket\Gamma_{m'p^\infty}\rrbracket$.
\end{cor}
\begin{cor}[Control theorem with respect to quotients] \label{cor:control}
Under the same assumptions of Theorem \ref{thm:freeness}, we have isomorphisms
\begin{align*}
\mathrm{Sel}^S_\Delta (K(m'p^\infty), T_{f,n}) / \mathfrak{m}_{p^\infty,n} & \simeq \mathrm{Sel}^S_\Delta (K(m'), T_{f,1}) ,\\
\mathrm{Sel}^S_\Delta (K(m'p^\infty), T_{f,n}) / \mathfrak{m}_{m'p^\infty,n} & \simeq \mathrm{Sel}^S_\Delta (K, T_{f,1}) ,\\
\mathrm{Sel}^S_\Delta (K(m'p^\infty), T_{f,n}) / \varpi & \simeq \mathrm{Sel}^S_\Delta (K(m'p^\infty), T_{f,1})
\end{align*}
where $\mathfrak{m}_{p^\infty,n} = ( \gamma_{p^\infty} - 1, \varpi) \subset \mathcal{O}_n\llbracket\Gamma_{m'p^\infty}\rrbracket$ and $\mathfrak{m}_{p^\infty,n} = ( \gamma_{m'p^\infty} - 1, \varpi) \subset \mathcal{O}_n\llbracket\Gamma_{m'p^\infty}\rrbracket$.
\end{cor}

\begin{proof}
This is a direct consequence of Theorem \ref{thm:freeness} and Corollary \ref{cor:surjectivity_of_cores}.
\end{proof}
\section{The Euler system argument: character by character induction} \label{sec:euler_systems}
This section is devoted to proving the character-evaluated version of the main theorem (Theorem \ref{thm:character-by-character}) under the assumptions given in the main theorem.

\subsection{Specialization via characters} 
We recall some material of $\S$\ref{subsec:convention_characters} for $\varphi$.
Let $\varphi : \mathcal{O}\llbracket\Gamma_{m'p^\infty}\rrbracket \to \mathcal{O}_\varphi$ be an $\mathcal{O}$-algebra homomorphism where $\mathcal{O}_\varphi$ is a discrete valuation ring of characteristic zero (depending on $\varphi$).
Let $\varpi_\varphi$ be an uniformizer of $\mathcal{O}_\varphi$ and $\mathfrak{m}_\varphi$ be the maximal ideal of $\mathcal{O}_\varphi$. Let $\mathbb{F}_{\varphi}:=\mathcal{O}_\varphi / \varpi_\varphi$ and
$\mathcal{O}_{\varphi, n}:=\mathcal{O}_\varphi / \varpi^n_\varphi$.
Let $M$ be a finitely generated $\mathcal{O}_\varphi$-module.
For $x \in M$, we define
$$\mathrm{ord}_{\varphi} \left( x \right) := \mathrm{sup}\left\lbrace t \in \mathbb{Z}_{\geq 0} : x \in \varpi^t_\varphi M  \right\rbrace .$$
Then $x = 0$ if and only if $\mathrm{ord}_{\varphi} \left( x \right) = \infty$.

For a $\mathcal{O}\llbracket\Gamma_{m'p^\infty}\rrbracket$-module $M$, $M \otimes_{\varphi} \mathcal{O}_\varphi := M \otimes_{\mathcal{O}\llbracket\Gamma_{m'p^\infty}\rrbracket, \varphi} \mathcal{O}_\varphi$ is called the \textbf{specialization of $M$ via $\varphi$}, which is an $\mathcal{O}_\varphi$-module.

\subsection{Setup for induction}
Let $n$ be a positive integer and $\Delta > 1$ be a square-free product of an odd number of primes satisfying
\begin{itemize}
\item $N^- \vert \Delta$, and
\item $\Delta / N^-$ is a product of $n$-admissible primes.
\end{itemize}
For each $n$-admissible form $\mathcal{D} = (\Delta, g_n)$ as in Definition \ref{defn:n-admissible-forms},
we define a non-negative integer
$$t_\mathcal{D} := \mathrm{ord}_{\varpi_\varphi} \left( \varphi \left( \theta_{{m'p^\infty}} \left(\mathcal{D} \right) \right) \right)$$
where $\varphi \left( \theta_{m'p^\infty} \left(\mathcal{D} \right) \right) \in \mathcal{O}_n\llbracket\Gamma_{m'p^\infty}\rrbracket \otimes_{\varphi} \mathcal{O}_\varphi = \mathcal{O}_{\varphi, n}$.
\begin{rem}
For the $m = p^\infty$ case, the parameter $t_{\mathcal{D}}$ is directly affected by analytic Iwasawa $\mu$-and $\lambda$-invariants. Indeed, the parameter should come from the $\lambda$-invariants only due to vanishing of $\mu$-invariants \`{a} la Vatsal.
For the $m = p^r$ case, the parameter $t_{\mathcal{D}}$ can be regarded the mixture of finite analytic Iwasawa $\mu$-and $\lambda$-invariants. See \cite[Lemma 4.5]{pollack-algebraic} for the explicit formula and \cite[$\S$4]{pollack-algebraic} for the context.
\end{rem}
 
Let $t^* \leq n$ be a non-negative integer.
Let $\mathcal{D}_{n+t^*} := (\Delta, g_{n+t^*})$ be an $(n+t^*)$-admissible form. Then $\mathcal{D}_n := \mathcal{D}_{n+t^*} \Mod{\varpi^n} = (\Delta, g_{n+t^*} \Mod{\varpi^n})$ is also an $n$-admissible form. Suppose that $t_{\mathcal{D}_n} \leq t^*$.
\begin{rem} \label{rem:induction_parameter}
We proceed the induction on $t_{\mathcal{D}_n}$.
\end{rem}
If $t_{\mathcal{D}_n}$ is zero or $\mathrm{Sel}_\Delta(K(m'p^\infty), A_{f,n})^\vee \otimes_{\varphi} \mathcal{O}_\varphi$ is trivial, then we have nothing to prove.
Thus we assume that $t_{\mathcal{D}_n} >0$ and $\mathrm{Sel}_\Delta(K(m'p^\infty), A_{f,n})^\vee \otimes_{\varphi} \mathcal{O}_\varphi$ is non-trivial.
Consider the $(n+t_{\mathcal{D}_n})$-admissible form
$$\mathcal{D} = \mathcal{D}_{n+t_{\mathcal{D}_n}} : = (\Delta, f_{n+t^*} \Mod{\varpi^{ n+t_{\mathcal{D}_n} }}) .$$

Let $\ell$ be an $(n+t_{\mathcal{D}_n})$-admissible prime with $\ell \nmid \Delta$. Enlarge $\lbrace \ell \rbrace $ to an $(n+t_{\mathcal{D}_n})$-admissible set $\mathfrak{S}$ such that $(S, \Delta)=1$ where $S = \prod_{q \in \mathfrak{S}}q$.
Recall the cohomology class
$$\kappa_\mathcal{D}(\ell)_{m'p^\infty} \in \mathrm{Sel}_{\Delta\ell}(K(m'p^\infty), T_{f,n+t_{\mathcal{D}_n}}) \subset \mathrm{Sel}^S_{\Delta}(K(m'p^\infty), T_{f,n+t_{\mathcal{D}_n}})$$
attached to $\mathcal{D}$ and $\ell$ as in $\S$\ref{sec:construction_cohomology}.
By the freeness theorem (Theorem \ref{thm:freeness}), the module $\mathcal{M}_n := \mathrm{Sel}^S_\Delta (K(m'p^\infty), T_{f,n}) \otimes_{\varphi} \mathcal{O}_\varphi$ is free of rank $\#\mathfrak{S}$ over $\mathcal{O}_\varphi / \varphi(\varpi^n)$ for all $n$.
Define the \textbf{specialized cohomology class} $$\kappa_{\mathcal{D}, \varphi}(\ell)_{m'p^\infty} \in \mathcal{M}_{n+t_{\mathcal{D}_n}}$$ by the image of $\kappa_{\mathcal{D}}(\ell)_{m'p^\infty}$ under the specialization
$$\mathrm{Sel}^S_\Delta (K(m'p^\infty), T_{f,n+t_{\mathcal{D}_n}}) \to
\mathcal{M}_{n+t_{\mathcal{D}_n}} = \mathrm{Sel}^S_\Delta (K(m'p^\infty), T_{f,n+t_{\mathcal{D}_n}}) \otimes_{\varphi} \mathcal{O}_\varphi.$$
Define
$$e_\mathcal{D}(\ell) := \mathrm{ord}_{\varpi_\varphi} \left( \kappa_{\mathcal{D}, \varphi}(\ell)_{m'p^\infty} \right)$$
which may also depend on the choice of $\mathfrak{S}$.

We have the following (in)equalities:
\begin{align*}
e_\mathcal{D}(\ell) & \leq  \mathrm{ord}_{\varpi_\varphi} \left( \partial_\ell \kappa_{\mathcal{D}, \varphi}(\ell)_{m'p^\infty} \right) & \partial_\ell \textrm{ is a homomorphism}.\\
& = \mathrm{ord}_{\varpi_\varphi} \left( \varphi \left( \theta_{m'p^\infty}(\mathcal{D}) \right) \right) & \textrm{The first explicit reciprocity law (Theorem \ref{thm:first_explicit_reciprocity})}\\
& \leq \mathrm{ord}_{\varpi_\varphi} \left( \varphi \left( \theta_{m'p^\infty}(\mathcal{D}_n) \right) \right) & \textrm{Reduction modulo $\varpi^n$ is a homomorphism}.\\
& = t_{\mathcal{D}_n}
\end{align*}
In fact, the last inequality is equality since $\mathcal{M}_{n+t_{\mathcal{D}_n}} / \varpi_\varphi \mathcal{M}_{n+t_{\mathcal{D}_n}} \simeq \mathcal{M}_{n} / \varpi_\varphi \mathcal{M}_{n}$. See also the proof of \cite[Lemma 6.16]{chida-hsieh-main-conj}.
\subsection{Reduction of the specialized cohomology classes}
Choose an element $\widetilde{\kappa}_{\mathcal{D}, \varphi}(\ell)_{m'p^\infty} \in \mathcal{M}_{n+t_{\mathcal{D}_n}}$ which satisfies
$$\varpi^{ e_\mathcal{D}(\ell) } \cdot \widetilde{\kappa}_{\mathcal{D}, \varphi} (\ell)_{m'p^\infty} = \kappa_{\mathcal{D}, \varphi} (\ell)_{m'p^\infty} .$$
Then $\widetilde{\kappa}_{\mathcal{D}, \varphi}(\ell)_{m'p^\infty}$ is well-defined modulo the $\varpi^{ e_\mathcal{D}(\ell) }$-torsion subgroup of $\mathcal{M}_{n+t_{\mathcal{D}_n}}$. Since the torsion subgroup is contained in the kernel of $\varpi^{n}$ reduction map $\mathcal{M}_{n+t_{\mathcal{D}_n}} \to \mathcal{M}_{n}$ with $e_\mathcal{D}(\ell) \leq t_{\mathcal{D}_n}$, 
the \textbf{reduced cohomology class} $\kappa'_{\mathcal{D}, \varphi}(\ell)_{m'p^\infty}$ defined by the image under the reduction map
\[
\xymatrix@R=0em{
\mathrm{Sel}^S_\Delta (K(m'p^\infty), T_{f,n+t_{\mathcal{D}_n}}) \otimes_{\varphi} \mathcal{O}_\varphi \ar[r] & \mathrm{Sel}^S_\Delta (K(m'p^\infty), T_{f,n}) \otimes_{\varphi} \mathcal{O}_\varphi \\
\widetilde{\kappa}_{\mathcal{D}, \varphi}(\ell)_{m'p^\infty} \ar@{|->}[r] & \kappa'_{\mathcal{D}, \varphi}(\ell)_{m'p^\infty}
}
\]
is well-defined in $\mathcal{M}_n$.
\begin{prop}[Local properties of the reduced cohomology classes; {\cite[Lemma 6.16]{chida-hsieh-main-conj}}] \label{lem:local_properties_reduced_classes}
The reduced cohomology class $\kappa'_{\mathcal{D}, \varphi}(\ell)_{m'p^\infty} \in \mathrm{Sel}^S_\Delta (K(m'p^\infty), T_{f,n}) \otimes_{\varphi} \mathcal{O}_\varphi$ satisfies the following properties:
\begin{enumerate}
\item $\mathrm{ord}_{\varpi_\varphi} \left( \kappa'_{\mathcal{D}, \varphi}(\ell)_{m'p^\infty} \right) = 0$;
\item $\mathrm{ord}_{\varpi_\varphi} \left( \partial_\ell \left( \kappa'_{\mathcal{D}, \varphi}(\ell)_{m'p^\infty} \right) \right) = t_{\mathcal{D}_n} - e_{\mathcal{D}}(\ell)$;
\item $\partial_q \left( \kappa'_{\mathcal{D}, \varphi}(\ell)_{m'p^\infty} \right) = 0$ for all $q \nmid \Delta\ell$;
\item $\mathrm{res}_q \left( \kappa'_{\mathcal{D}, \varphi}(\ell)_{m'p^\infty} \right) \in \mathrm{H}^1_{\mathrm{ord}}(K(m'p^\infty)_q, T_{f,n})$ for all $q \mid \Delta\ell$.
\end{enumerate}
\end{prop}
\begin{rem}
\begin{enumerate}
\item Property (1) in the proposition comes from the isomorphism
$$\mathcal{M}_{n+t_{\mathcal{D}_n}} /\varpi_\varphi \mathcal{M}_{n+t_{\mathcal{D}_n}} \simeq \mathcal{M}_{n} /\varpi_\varphi \mathcal{M}_{n}$$
 obtained by Corollary \ref{cor:control}.
\item Property (2) is a direct consequence of the first explicit reciprocity law (Theorem \ref{thm:first_explicit_reciprocity}).
\item Property (3) and (4) come from Proposition \ref{prop:local_properties} and the freeness of $\mathrm{H}^1_{\mathrm{ord}}(K(m'p^\infty)_q, T_{f,n++t_{\mathcal{D}_n}})$ by Proposition \ref{prop:local_computation}.(5).
\end{enumerate}
\end{rem}
\begin{rem}
It seems difficult to obtain such reduced cohomology classes without the specialization via $\varphi$ since the reduction without the specialization may harm the freeness of compact Selmer groups.
\end{rem}

\subsection{The first step of induction} \label{subsec:first_step}
We define a natural homomorphism $\eta_\ell$ by
\begin{align*}
\eta_\ell: \mathrm{H}^1_{\mathrm{sing}}(K(m'p^\infty)_{\ell}, T_{f,n}) & \to \mathrm{Sel}_{\Delta}(K(m'p^\infty), A_{f,n})^\vee\\
\kappa & \mapsto \left( \eta_{\ell} (\kappa) : s \mapsto \langle \kappa, v_\ell(s) \rangle_{\ell} \right) 
\end{align*}
and its specialization $\eta^\varphi_\ell$ via $\varphi$
$$\eta^\varphi_\ell: \mathrm{H}^1_{\mathrm{sing}}(K(m'p^\infty)_{\ell}, T_{f,n}) \otimes_{\varphi} \mathcal{O}_\varphi \to \mathrm{Sel}_{\Delta}(K(m'p^\infty), A_{f,n})^\vee \otimes_{\varphi} \mathcal{O}_\varphi .$$
We need a lemma, which uses the global property of the reduced cohomology classes. It is a variant of \cite[Lemma 6.17]{chida-hsieh-main-conj} over $K(m'p^\infty)$.
\begin{lem} \label{lem:eta_ell}
The map $\eta^\varphi_\ell \left( \partial_{\ell} ( \kappa'_{\mathcal{D}, \varphi}(\ell)_{m'p^\infty}  ) \right): \mathrm{Sel}_{\Delta}(K(m'p^\infty), A_{f,n})^\vee \otimes_{\varphi} \mathcal{O}_\varphi \to E/\mathcal{O}$ is the zero map. 
\end{lem}
\begin{proof}
Let $s \in \mathrm{Sel}_{\Delta}(K(m'p^\infty), A_{f,n})[\ker \varphi]$. By the local properties of the reduced class $\kappa'_{\mathcal{D}}(\ell)_{m'p^\infty}$ (Lemma \ref{lem:local_properties_reduced_classes}.(3), (4)), we see that
$$\langle \partial_{q} ( \kappa'_{\mathcal{D}, \varphi}(\ell)_{m'p^\infty}  ), v_q(s) \rangle_{q} = 0$$ for all $q \neq \ell$.
Then by global reciprocity (Proposition \ref{prop:global_reciprocity}), the lemma follows.
\end{proof}

We generalize \cite[Proposition 4.7]{bertolini-darmon-imc-2005} to the general cyclic case.
\begin{prop}[The first step of induction] \label{prop:first_step}
If $t_{\mathcal{D}_n} = 0$, then $\mathrm{Sel}_{\Delta}(K(m'p^\infty), A_{f,n})^\vee \otimes_\varphi \mathcal{O}_\varphi$ is trivial.
\end{prop}
Note that $t_{\mathcal{D}_n} = 0$ is not equivalent to $\theta_{m'p^\infty}(\mathcal{D}_n) \in \mathcal{O}_n\llbracket\Gamma_{m'p^\infty}\rrbracket^\times$ anymore since $\mathcal{O}_n\llbracket\Gamma_{m'p^\infty}\rrbracket$ is not a local ring.
\begin{proof}
We will see $\eta^\varphi_\ell = 0$ but $\eta^\varphi_\ell \Mod{\mathfrak{m}_\varphi} \neq 0$ under the assumption $\mathrm{Sel}_{\Delta}(K(m'p^\infty), A_{f,n})^\vee \otimes_{\varphi} \mathcal{O}_{\varphi} \neq \lbrace 0 \rbrace$ in the following commutative diagram
\[
\xymatrix{
\mathrm{H}^1_{\mathrm{sing}}(K(m'p^\infty)_{\ell}, T_{f,n}) \ar[rr]^-{\eta_\ell} \ar[d]^-{\otimes_{\varphi} \mathcal{O}_\varphi} & & \mathrm{Sel}_{\Delta}(K(m'p^\infty), A_{f,n})^\vee \ar[d]^-{\otimes_{\varphi} \mathcal{O}_\varphi}\\
\mathrm{H}^1_{\mathrm{sing}}(K(m'p^\infty)_{\ell}, T_{f,n}) \otimes_{\varphi} \mathcal{O}_\varphi \ar[rr]^-{\eta^\varphi_\ell}_{(=0)} \ar@{->>}[d]^-{\Mod{\mathfrak{m}_{\varphi}}} & & \mathrm{Sel}_{\Delta}(K(m'p^\infty), A_{f,n})^\vee \otimes_{\varphi} \mathcal{O}_\varphi \ar@{->>}[d]^-{\Mod{\mathfrak{m}_{\varphi}}} \\
\left(\mathrm{H}^1_{\mathrm{sing}}(K(m'p^\infty)_{\ell}, T_{f,n}) \otimes_{\varphi} \mathcal{O}_\varphi \right) / \mathfrak{m}_{\varphi} \ar[rr]^-{\eta^\varphi_\ell \Mod{\mathfrak{m}_\varphi}}_{(\neq 0 \ \Rightarrow\Leftarrow)} & & \left( \mathrm{Sel}_{\Delta}(K(m'p^\infty), A_{f,n})^\vee \otimes_{\varphi} \mathcal{O}_\varphi \right) / \mathfrak{m}_{\varphi} .
}
\]

Since the following diagram 
\[
\xymatrix{
\mathrm{Sel}^S_{\Delta}(K(m'p^\infty), T_{f,n}) \ar[r]^-{\otimes_{\varphi} \mathcal{O}_\varphi} \ar[d]^-{\partial_\ell} & \mathrm{Sel}^S_{\Delta}(K(m'p^\infty), T_{f,n}) \otimes_{\varphi} \mathcal{O}_\varphi \ar[d]^-{\partial_\ell} \\
\mathrm{H}^1_{\mathrm{sing}}(K(m'p^\infty)_{\ell}, T_{f,n}) \ar[r]^-{\otimes_{\varphi} \mathcal{O}_\varphi}  & \mathrm{H}^1_{\mathrm{sing}}(K(m'p^\infty)_{\ell}, T_{f,n}) \otimes_{\varphi} \mathcal{O}_\varphi 
}
\]
commutes, we have
$$\partial_\ell \left( \kappa_{\mathcal{D},\varphi}(\ell)_{m'p^\infty} \right) = \varphi \left( \partial_\ell \left( \kappa_{\mathcal{D}}(\ell)_{m'p^\infty} \right) \right)$$
under the identification $\mathrm{H}^1_{\mathrm{sing}}(K(m'p^\infty)_{\ell}, T_{f,n}) \otimes_{\varphi} \mathcal{O}_\varphi \simeq
\mathcal{O}_n\llbracket\Gamma_{m'p^\infty}\rrbracket \otimes_{\varphi} \mathcal{O}_\varphi \simeq \mathcal{O}_{\varphi, n}$.

If $\varphi( \theta_{m'p^\infty}(\mathcal{D}_n) )$ is a unit in $\mathcal{O}_{\varphi,n}$, then the class $\partial_\ell \left( \kappa_{\mathcal{D},\varphi}(\ell)_{m'p^\infty} \right)$ generates $\mathrm{H}^1_{\mathrm{sing}} (K(m'p^\infty)_{\ell}, T_{f,n}) \otimes_{\varphi} \mathcal{O}_\varphi \simeq \mathcal{O}_{\varphi, n}$ for any $n$-admissible prime $\ell$ via the first explicit reciprocity law (Theorem \ref{thm:first_explicit_reciprocity}).
Thus, $\eta^{\varphi}_\ell$ is trivial for any $n$-admissible prime due to Lemma \ref{lem:eta_ell} with $\kappa_{\mathcal{D},\varphi}(\ell)_{m'p^\infty} = \kappa'_{\mathcal{D},\varphi}(\ell)_{m'p^\infty}$ under $t_{\mathcal{D}_n} = 0$.
Note that $\mathcal{D}_n$ and $\mathcal{D}$ are congruent modulo $\varpi^n$.

Suppose that $$\mathrm{Sel}_{\Delta}(K(m'p^\infty), A_{f,n})^\vee \otimes_{\varphi} \mathcal{O}_{\varphi}$$ is non-trivial.
Then Nakayama's lemma with $\mathcal{O}_{\varphi}$ implies that
$$\left( \mathrm{Sel}_{\Delta}(K(m'p^\infty), A_{f,n})^\vee \otimes_{\varphi} \mathcal{O}_{\varphi} \right) / \mathfrak{m}_{\varphi} $$
is also non-trivial.

Let $\gamma_{p^r} = (\gamma_m)^{m'}$. 
Since $\varphi(\gamma_{p^r}) - 1$ is divisible by $\varpi_\varphi$, we have
\begin{align*}
\left( \mathrm{Sel}_{\Delta}(K(m'p^\infty), A_{f,n})^\vee \otimes_{\varphi} \mathcal{O}_{\varphi} \right) / \mathfrak{m}_{\varphi} & = \left(\mathrm{Sel}_{\Delta}(K(m'p^\infty), A_{f,n}) \right)^\vee  / \mathfrak{m}_{p^\infty,n} \otimes_{\varphi} \mathcal{O}_{\varphi} / \mathfrak{m}_{\varphi} &\\
 & = \left(\mathrm{Sel}_{\Delta}(K(m'p^\infty), A_{f,n})[\mathfrak{m}_{p^\infty,n}] \right)^\vee  \otimes_{\varphi} \mathcal{O}_{\varphi} / \mathfrak{m}_{\varphi} & \\
 & = \left(\mathrm{Sel}_{\Delta}(K(m'), A_{f,1}) \right)^\vee  \otimes_{\varphi} \mathcal{O}_{\varphi} / \mathfrak{m}_{\varphi} & \textrm{Corollary } \ref{cor:control_galois_mod_pn}
\end{align*}
where $\mathfrak{m}_{p^\infty,n} = (\gamma_{p^\infty} - 1, \varpi) \subset \mathcal{O}_n\llbracket\Gamma_{m'p^\infty}\rrbracket$.

Let $s \in \mathrm{Sel}_{\Delta}(K(m'), A_{f,1}) \subseteq \mathrm{H}^1(K(m'), A_{f,1})$ be a non-trivial element. By Theorem \ref{thm:chebotarev}, we can choose an $n$-admissible prime $\ell$ not dividing $m'p\Delta N^+$ such that $\partial_\ell(s) = 0$ and $v_\ell(s) \neq 0 \in \mathrm{H}^1_{\mathrm{fin}}(K(m')_\ell, A_{f,1})$.
For such $s$ and $\ell$,
$\kappa \mapsto \langle \kappa, v_\ell(s) \rangle_\ell$ cannot be identically zero on $\kappa \in \mathrm{H}^1_{\mathrm{sing}}(K(m'), T_{f,1})$
since the local Tate pairing (in $\mathrm{H}^1_{\mathrm{fin}}(K(m')_\ell, A_{f,1})$) at any prime $\ell$ is perfect.
This implies that $\eta^{\varphi}_\ell \Mod{\mathfrak{m}_\varphi}$ cannot be trivial for such an $\ell$. A contradiction follows considering the commutative diagram given first.   
\end{proof} 

\subsection{Appealing to the second explicit reciprocity law} \label{subsec:appeal_to_second}
Due to Proposition \ref{prop:first_step}, we may assume $t_{\mathcal{D}_n} > 0$.
Let $\Pi$ be the set of $(n+t^*)$-admissible primes $\ell$ satisfying
\begin{enumerate}
\item $\ell \nmid m'p\Delta  N^+$,
\item the integer $e_{\mathcal{D}}(\ell) = \mathrm{ord}_{\varpi_\varphi} \left( \kappa_{\mathcal{D},\varphi}(\ell)_{m'p^\infty} \right)$ is minimal among $\ell$ satisfying the condition (1).
\end{enumerate}
Then $\Pi$ is non-empty by Theorem \ref{thm:chebotarev}. Let $e_{\mathrm{min}} := e_\mathcal{D}(\ell)$ for any $\ell \in \Pi$.

\begin{lem}[{\cite[Lemma 6.19]{chida-hsieh-main-conj}}] \label{lem:strict-inequality}
The strict inequality $e_{\mathrm{min}} < t_{\mathcal{D}_n}$ holds.
\end{lem}
\begin{proof}
Suppose that $e_{\mathrm{min}} = t_{\mathcal{D}_n}$.
Then $e_{\mathcal{D}}(\ell) = t_{\mathcal{D}_n}$ for all $(n+t^*)$-admissible primes $\ell$ due to inequality $e_{\mathcal{D}}(\ell) \leq  t_{\mathcal{D}_n}$.
By Corollary \ref{cor:control_galois_mod_pn}, we have $\mathrm{H}^1(K, A_{f,1}) \simeq \mathrm{H}^1(K({m'p^\infty}), A_{f,n})[\mathfrak{m}_{m'p^\infty,n}]$. Then there exists a non-zero element $s \in \mathrm{Sel}_{\Delta}(K({m'p^\infty}), A_{f,n})[\mathfrak{m}_{m'p^\infty,n}] \subseteq\mathrm{H}^1(K, A_{f,1}) \otimes_\varphi \mathcal{O}_\varphi$.
By Theorem \ref{thm:chebotarev} with $m'=1$, there exists an $(n+t^*)$-admissible prime $\ell$ such that $v_\ell(s)$ is non-zero in $\mathrm{H}^1_{\mathrm{fin}} (K_{\ell}, A_{f,1}) \otimes_\varphi \mathcal{O}_\varphi$.
Also, by Lemma \ref{lem:local_properties_reduced_classes}.(2), the image of $\partial_\ell (\kappa'_{\mathcal{D}, \varphi}(\ell))$ in $\mathrm{H}^1_{\mathrm{sing}}(K_\ell, T_{f,1})  \otimes_\varphi \mathcal{O}_\varphi$ is non-zero.
By Lemma \ref{lem:eta_ell}, the image of $\partial_\ell (\kappa'_{\mathcal{D}, \varphi}(\ell))$ in $\mathrm{H}^1_{\mathrm{sing}}(K_\ell, T_{f,1})  \otimes_\varphi \mathcal{O}_\varphi$ and $v_\ell(s)$ are orthogonal to each other with respect to the local Tate pairing.
Since the local Tate pairing
$\mathrm{H}^1_{\mathrm{fin}} (K_{\ell}, A_{f,1}) \otimes_\varphi \mathcal{O}_\varphi \times \mathrm{H}^1_{\mathrm{sing}}(K_\ell, T_{f,1})  \otimes_\varphi \mathcal{O}_\varphi \to \mathcal{O}_\varphi / \varpi\mathcal{O}_\varphi$ is perfect, and
$\mathrm{H}^1_{\mathrm{fin}} (K_{\ell}, A_{f,1}) \otimes_\varphi \mathcal{O}_\varphi$ and $\mathrm{H}^1_{\mathrm{sing}}(K_\ell, T_{f,1})  \otimes_\varphi \mathcal{O}_\varphi$ are one-dimensional over $\mathcal{O}_\varphi / \varpi\mathcal{O}_\varphi$, we have a contradiction.
\end{proof}

Let $\ell_1 \in \Pi$ and $\mathfrak{S}$ be an $(n+t^*)$-admissible set with $S = \prod_{q\in \mathfrak{S}}q$. Let $\kappa_{m'} \in \mathrm{H}^1(K(m'),T_{f,1}) \otimes_\varphi  \mathcal{O}_\varphi$ be the image of $\kappa'_{\mathcal{D}, \varphi}(\ell_1)_{m'p^\infty}$ under the map
\[
\xymatrix@R=0em{
\mathrm{Sel}^S_{\Delta}(K({m'p^\infty}), T_{f,n}) \otimes_\varphi \left( \mathcal{O}_\varphi /\varpi_\varphi \mathcal{O}_\varphi \right) \ar[r]^-{\simeq} & \left( \mathrm{Sel}^S_{\Delta}(K({m'p^\infty}), T_{f,n}) / \mathfrak{m}_{p^\infty,n} \right) \otimes_\varphi \mathcal{O}_\varphi \ar@{^{(}->}[r]^-{\mathrm{cores}} & \mathrm{H}^1(K(m'),T_{f,1}) \otimes_\varphi  \mathcal{O}_\varphi \\
\kappa'_{\mathcal{D}, \varphi}(\ell_1)_{m'p^\infty} \ar@{|->}[rr] & & \kappa_{m'} .
}
\]
The last inclusion comes from Corollary \ref{cor:control}, which comes from Theorem \ref{thm:freeness}.

Therefore, $\kappa_{m'}$ is a non-zero element in $\mathrm{H}^1(K(m'),T_{f,1}) \otimes_\varphi  \mathcal{O}_\varphi$.
By Theorem \ref{thm:chebotarev}, there exists an $(n+t^*)$-admissible prime $\ell_2$ not dividing $m'p\Delta N^+$ such that
\begin{itemize}
\item $\partial_{\ell_2} (\kappa_{m'}) = 0$, and
\item $v_{\ell_2}(\kappa_{m'}) \neq 0$ in $\mathrm{H}^1(K(m')_{\ell_2}, T_{f,1}) \otimes_\varphi \mathcal{O}_\varphi$.
\end{itemize}
Then we have $v_{\ell_2}(\kappa_{m'}) \neq 0$. The following commutative diagram
\[
{\scriptsize
\xymatrix{
\mathrm{Sel}^S_{\Delta}(K(m'p^\infty), T_{f,n}) \otimes_\varphi \mathcal{O}_\varphi /\varpi_\varphi \mathcal{O}_\varphi \ar@{^{(}->}[r]^-{\mathrm{cores}} \ar[d]^-{v_{\ell_2}} & \mathrm{H}^1(K(m'),T_{f,1}) \otimes_\varphi  \mathcal{O}_\varphi \ar[d]^-{v_{\ell_2}} & \kappa'_{\mathcal{D}, \varphi}(\ell_1)_{m'p^\infty} \ar@{|->}[r] \ar@{|->}[d] & \kappa_{m'} = \mathrm{cores}\left( \kappa'_{\mathcal{D}, \varphi}(\ell_1)_{m'p^\infty} \right) \ar@{|->}[d]\\
 \mathrm{H}^1_{\mathrm{fin}} (K(m'p^\infty)_{\ell_2}, T_{f,n}) \otimes_\varphi \mathcal{O}_\varphi /\varpi_\varphi \mathcal{O}_\varphi \ar[r]^-{\simeq} &  \mathrm{H}^1_{\mathrm{fin}} (K(m')_{\ell_2}, T_{f,1}) \otimes_{\varphi} \mathcal{O}_\varphi & v_{\ell_2} \left( \kappa'_{\mathcal{D}, \varphi}(\ell_1)_{m'p^\infty} \right) \ar@{|->}[r] & v_{\ell_2} \left( \kappa_{m'}\right) \neq 0
}
}
\]
shows that $\mathrm{ord}_{\varpi_\varphi} \left(v_{\ell_2} \left( \kappa_{\mathcal{D},\varphi}(\ell_1)_{m'p^\infty} \right)\right)
= \mathrm{ord}_{\varpi_\varphi} \left(\kappa_{\mathcal{D},\varphi}(\ell_1)_{m'p^\infty} \right)$. In other words, the homomorphism $v_{\ell_2}$ does not increase the valuation considering Lemma \ref{lem:local_properties_reduced_classes}.(1).
Also, the minimality of $e_{\mathrm{min}} = e_{\mathcal{D}}(\ell_1) = \mathrm{ord}_{\varpi_\varphi} \left(\kappa_{\mathcal{D},\varphi}(\ell_1)_{m'p^\infty} \right)$ imply that
\begin{align} \label{eqn:inequality}
\mathrm{ord}_{\varpi_\varphi} \left(v_{\ell_2} \left( \kappa_{\mathcal{D},\varphi}(\ell_1)_{m'p^\infty} \right)\right) & = 
\mathrm{ord}_{\varpi_\varphi} \left(\kappa_{\mathcal{D},\varphi}(\ell_1)_{m'p^\infty} \right) & \textrm{the property of } v_{\ell_2} \nonumber \\
& \leq \mathrm{ord}_{\varpi_\varphi} \left(\kappa_{\mathcal{D},\varphi}(\ell_2)_{m'p^\infty} \right) & \textrm{the minimality of } e_{\mathrm{min}} \\
& \leq \mathrm{ord}_{\varpi_\varphi} \left( v_{\ell_1} \left( \kappa_{\mathcal{D},\varphi}(\ell_2)_{m'p^\infty} \right) \right) & v_{\ell_1} \textrm{ is a homomorphism} .\nonumber 
\end{align} 
By the second explicit reciprocity law (Theorem \ref{thm:second_explicit_reciprocity}), there exists an $(n + t^*)$-admissible form
$\mathcal{D}^{\ell_1\ell_2}_{n+t^*} = (\Delta\ell_1\ell_2, g_{n+t^*})$
such that
$$v_{\ell_2} (\kappa_\mathcal{D}(\ell_1)_{m'p^\infty} )
= v_{\ell_1} (\kappa_\mathcal{D}(\ell_2)_{m'p^\infty} )
= \theta_{m'p^\infty}(\mathcal{D}^{\ell_1\ell_2}_{n+t_{\mathcal{D}_n}})$$
where $\mathcal{D}^{\ell_1\ell_2}_{n+t_{\mathcal{D}_n}} = (\Delta\ell_1\ell_2, g_{n+t^*} \Mod{\varpi^{n+_{\mathcal{D}_n}}})$.
This directly implies that
$\mathrm{ord}_{\varpi_\varphi} \left( v_{\ell_1} (\kappa_{\mathcal{D}, \varphi}(\ell_2)_{m'p^\infty}) \right) =
\mathrm{ord}_{\varpi_\varphi} \left( v_{\ell_2} (\kappa_{\mathcal{D}, \varphi}(\ell_1)_{m'p^\infty}) \right)$.
Therefore, we have
$\mathrm{ord}_{\varpi_\varphi} \left( v_{\ell_1} (\kappa_{\mathcal{D}, \varphi}(\ell_2)_{m'p^\infty}) \right) =
\mathrm{ord}_{\varpi_\varphi} \left( v_{\ell_2} (\kappa_{\mathcal{D}, \varphi}(\ell_1)_{m'p^\infty}) \right)$.
Then all the inequalities (\ref{eqn:inequality}) become equalities and we have 
$$\mathrm{ord}_{\varpi_\varphi} \left( v_{\ell_2} (\kappa_{\mathcal{D}, \varphi}(\ell_1)_{m'p^\infty}) \right) = e_{\mathcal{D}}(\ell_1) =  e_{\mathcal{D}}(\ell_2) = e_{\mathrm{min}}$$ and $\ell_2 \in \Pi$.

Let $\mathcal{D}^{\ell_1\ell_2}_{n} := (\Delta\ell_1\ell_2, g_{n+t^*} \Mod{\varpi^n})$.
Then we have
$$t_{\mathcal{D}^{\ell_1\ell_2}_{n}} = \mathrm{ord}_{\varpi_\varphi} \left(  \theta_{m'p^\infty} ( \mathcal{D}^{\ell_1\ell_2}_{n} ) \right) = e_\mathrm{min}
< t_{\mathcal{D}_n}
\leq t^* $$ due to Lemma \ref{lem:strict-inequality}.

Thus, we can apply the induction hypothesis to $\mathcal{D}^{\ell_1\ell_2}_{n+t^*}$ reminding Remark \ref{rem:induction_parameter}. In other words, we may assume to have the following statement.
\begin{assu}[Induction hypothesis] \label{assu:induction}
$$\varphi( L_p(K(m'p^\infty) , \mathcal{D}^{\ell_1\ell_2}_{n+t^*}  )  ) \in \mathrm{Fitt}_{\mathcal{O}_{\varphi,n+t^*}} (\mathrm{Sel}_{\Delta\ell_1\ell_2}(K(m'p^\infty), A_{f,n+t^*})^\vee \otimes_\varphi \mathcal{O}_\varphi).$$
\end{assu}
\subsection{Reduction for induction} \label{subsec:reduction_for_induction}
Let $S_{[\ell_1 \ell_2]} \subseteq \mathrm{Sel}_{\Delta}(K({m'p^\infty}), A_{f,n})$ be the subgroup consisting of classes which are locally trivial at the primes dividing $\ell_1$ and $\ell_2$.
Then we have two exact sequences of $\mathcal{O}\llbracket\Gamma_{m'p^\infty}\rrbracket$-modules
{\scriptsize
\[
\xymatrix@R=0em{
\mathrm{H}^1_{\mathrm{sing}}(K(m'p^\infty)_{\ell_1}, T_{f,n}) \oplus \mathrm{H}^1_{\mathrm{sing}}(K(m'p^\infty)_{\ell_2}, T_{f,n}) \ar[r]^-{\eta_s} & \mathrm{Sel}_{\Delta}(K(m'p^\infty), A_{f,n})^\vee \ar[r] & S^{\vee}_{[\ell_1 \ell_2]} \ar[r] & 0 \\
\mathrm{H}^1_{\mathrm{fin}}(K(m'p^\infty)_{\ell_1}, T_{f,n}) \oplus \mathrm{H}^1_{\mathrm{fin}}(K(m'p^\infty)_{\ell_2}, T_{f,n}) \ar[r]^-{\eta_f} & \mathrm{Sel}_{\Delta \ell_1 \ell_2}(K(m'p^\infty), A_{f,n})^\vee \ar[r] & S^{\vee}_{[\ell_1 \ell_2]} \ar[r] & 0
}
\]
}
where $\eta_s$ and $\eta_f$ are induced by the local paring $\langle -,- \rangle_{\ell_1} \oplus \langle -,- \rangle_{\ell_2}$.
Fix isomorphisms
\begin{align*}
\mathrm{H}^1_{\mathrm{sing}}(K({m'p^\infty})_{\ell_1}, T_{f,n}) \oplus \mathrm{H}^1_{\mathrm{sing}}(K({m'p^\infty})_{\ell_2}, T_{f,n}) & \simeq
\mathcal{O}_n\llbracket\Gamma_{m'p^\infty}\rrbracket \oplus \mathcal{O}_n\llbracket\Gamma_{m'p^\infty}\rrbracket \\
\mathrm{H}^1_{\mathrm{fin}}(K(m'p^\infty)_{\ell_1}, T_{f,n}) \oplus \mathrm{H}^1_{\mathrm{fin}}(K(m'p^\infty)_{\ell_2}, T_{f,n}) & \simeq
\mathcal{O}_n\llbracket\Gamma_{m'p^\infty}\rrbracket \oplus \mathcal{O}_n\llbracket\Gamma_{m'p^\infty}\rrbracket .
\end{align*}
Tensoring with $\mathcal{O}_\varphi$ via $\varphi$, two sequences become
{\scriptsize
\[
\xymatrix@R=0em{
 \mathcal{O}_{\varphi,n} \oplus  \mathcal{O}_{\varphi,n} \ar[r]^-{\eta^{\varphi}_s} & \mathrm{Sel}_{\Delta}(K(m'p^\infty), A_{f,n})^\vee \otimes_\varphi \mathcal{O}_\varphi \ar[r] & S^{\vee}_{[\ell_1 \ell_2]} \otimes_\varphi \mathcal{O}_\varphi \ar[r] & 0 \\
 \mathcal{O}_{\varphi,n} \oplus  \mathcal{O}_{\varphi,n}  \ar[r]^-{\eta^{\varphi}_f} & \mathrm{Sel}_{\Delta \ell_1 \ell_2}(K(m'p^\infty), A_{f,n})^\vee \otimes_\varphi \mathcal{O}_\varphi \ar[r] & S^{\vee}_{[\ell_1 \ell_2]} \otimes_\varphi \mathcal{O}_\varphi \ar[r] & 0
}
\]
}
where $\eta^\varphi_s$ and $\eta^\varphi_f$ are induced from $\eta_s$ and $\eta_f$, respectively, as in $\S$\ref{subsec:first_step}.

By Lemma \ref{lem:eta_ell}, we know that
$\eta^\varphi_s$ factors through
$$\mathcal{O}_{\varphi,n} / \left( \partial_{\ell_1} (\kappa'_{\mathcal{D}, \varphi} (\ell_1)_{m'p^\infty}) \right) \oplus \mathcal{O}_{\varphi,n} / \left( \partial_{\ell_2} (\kappa'_{\mathcal{D}, \varphi} (\ell_2)_{m'p^\infty}) \right) .$$
In other words, the first sequence becomes
{\scriptsize
\[
\xymatrix{
\mathcal{O}_{\varphi,n} / \left( \partial_{\ell_1} (\kappa'_{\mathcal{D}, \varphi} (\ell_1)_{m'p^\infty}) \right) \oplus \mathcal{O}_{\varphi,n} / \left( \partial_{\ell_2} (\kappa'_{\mathcal{D}, \varphi} (\ell_2)_{m'p^\infty}) \right) \ar[r]^-{\eta^\varphi_s} & \mathrm{Sel}_{\Delta}(K(m'p^\infty), A_{f,n})^\vee \otimes_\varphi \mathcal{O}_{\varphi} \ar[r] & S^{\vee}_{[\ell_1 \ell_2]} \otimes_\varphi \mathcal{O}_{\varphi} \ar[r] & 0 .
}
\]
}
By Lemma \ref{lem:local_properties_reduced_classes}.(2), we have
$$t_{\mathcal{D}_n} - t_{\mathcal{D}^{\ell_1\ell_2}_{n}} = \mathrm{ord}_{\varpi_{\varphi}} \left( \partial_{\ell_1} \kappa'_{\mathcal{D}, \varphi}(\ell_1)_{m'p^\infty}  \right)  
 = \mathrm{ord}_{\varpi_{\varphi}} \left( \partial_{\ell_2} \kappa'_{\mathcal{D}, \varphi}(\ell_2)_{m'p^\infty}  \right).$$
The first sequence can be restated as
\[
\xymatrix{
\left( \mathcal{O}_{\varphi,n} / \varpi^{t_{\mathcal{D}_n} - t_{\mathcal{D}^{\ell_1\ell_2}_{n}}}_\varphi \right)^2 \ar[r]^-{\eta^\varphi_s} & \mathrm{Sel}_{\Delta}(K(m'p^\infty), A_{f,n})^\vee \otimes_\varphi \mathcal{O}_\varphi \ar[r] & S^{\vee}_{[\ell_1 \ell_2]} \otimes_\varphi \mathcal{O}_\varphi \ar[r] & 0 .
}
\]
In order to deal with the second sequence, we need the following lemma.
\begin{lem}
The kernel of $\eta^\varphi_f$ contains $\left( 0, v_{\ell_2}( \kappa'_{\mathcal{D},\varphi}(\ell_1)_{m'p^\infty}  )  \right)$ and $\left( v_{\ell_1}( \kappa'_{\mathcal{D},\varphi}(\ell_2)_{m'p^\infty}  ), 0  \right)$.
\end{lem}
\begin{proof}
The proof is similar to the one of Lemma \ref{lem:eta_ell}. See \cite[Lemma 6.20]{chida-hsieh-main-conj} for detail.
\end{proof}
The second sequence becomes
{\scriptsize
\[
\xymatrix{
 \mathcal{O}_{\varphi,n} / \left( v_{\ell_1}( \kappa'_{\mathcal{D},\varphi}(\ell_2)_{m'p^\infty}  ) \right) \oplus  \mathcal{O}_{\varphi,n} /\left( v_{\ell_2}( \kappa'_{\mathcal{D},\varphi}(\ell_1)_{m'p^\infty}  ) \right) \ar[r]^-{\eta^{\varphi}_f} & \mathrm{Sel}_{\Delta \ell_1 \ell_2}(K(m'p^\infty), A_{f,n})^\vee \otimes_\varphi \mathcal{O}_\varphi \ar[r] & S^{\vee}_{[\ell_1 \ell_2]} \otimes_\varphi \mathcal{O}_\varphi \ar[r] & 0 .
}
\]
}
Note that we have
$$\mathrm{ord}_{\varpi_\varphi} (v_{\ell_2} \kappa'_{\mathcal{D}, \varphi} (\ell_1)_{m'p^\infty} )
=\mathrm{ord}_{\varpi_\varphi} (v_{\ell_1} \kappa'_{\mathcal{D}, \varphi} (\ell_2)_{m'p^\infty} )
=t_{\mathcal{D}^{\ell_1\ell_2}_n} - e_\mathrm{min} = 0.$$
This shows that
$$\mathrm{Sel}_{\Delta \ell_1 \ell_2}(K(m'p^\infty), A_{f,n})^\vee \otimes_\varphi \mathcal{O}_\varphi \simeq S^{\vee}_{[\ell_1 \ell_2]} \otimes_\varphi \mathcal{O}_\varphi .$$

Plugging this isomorphism into the first sequence, we have
\[
\xymatrix{
\left( \mathcal{O}_{\varphi,n} / \varpi^{t_{\mathcal{D}_n} - t_{\mathcal{D}^{\ell_1\ell_2}_{n}}}_\varphi \right)^2 \ar[r]^-{\eta^\varphi_s} & \mathrm{Sel}_{\Delta}(K(m), A_{f,n})^\vee \otimes_\varphi \mathcal{O}_\varphi \ar[r] & \mathrm{Sel}_{\Delta \ell_1 \ell_2}(K(m), A_{f,n})^\vee \otimes_\varphi \mathcal{O}_\varphi \ar[r] & 0 .
}
\]
Let $\mathcal{S}$ be the exact kernel of the map $\mathrm{Sel}_{\Delta}(K(m), A_{f,n})^\vee \otimes_\varphi \mathcal{O}_\varphi \to \mathrm{Sel}_{\Delta \ell_1 \ell_2}(K(m), A_{f,n})^\vee \otimes_\varphi \mathcal{O}_\varphi$ in the sequence. Then we have
\begin{align*}
\varphi ( L_p(K(m), \mathcal{D}_n)  ) & = \varpi^{2 t_{\mathcal{D}_n}}_\varphi \\
& = \varpi^{2 \left( t_{\mathcal{D}_n} - t_{\mathcal{D}^{\ell_1\ell_2}_{n}} \right)}_\varphi  \cdot \varpi^{2t_{\mathcal{D}^{\ell_1\ell_2}_{n}}}_\varphi  \\
& \in \mathrm{Fitt}_{\mathcal{O}_{\varphi,n}}( \mathcal{S} )  \cdot \mathrm{Fitt}_{\mathcal{O}_{\varphi,n}}( \mathrm{Sel}_{\Delta \ell_1 \ell_2}(K(m), A_{f,n})^\vee \otimes_\varphi \mathcal{O}_\varphi ) & \textrm{Assumption \ref{assu:induction}} \\
& = \mathrm{Fitt}_{\mathcal{O}_{\varphi,n}}( \mathcal{S} )  \cdot \mathrm{Fitt}_{\mathcal{O}_{\varphi,n}}( S^{\vee}_{[\ell_1 \ell_2]} \otimes_\varphi \mathcal{O}_\varphi ) \\
& \subseteq \mathrm{Fitt}_{\mathcal{O}_{\varphi,n}}( \mathrm{Sel}_{\Delta}(K(m), A_{f,n})^\vee \otimes_\varphi \mathcal{O}_\varphi ) . & \textrm{Lemma \ref{lem:fitting_exact}}
\end{align*}
\begin{rem}
More precisely, the induction hypothesis (Assumption \ref{assu:induction}) is used modulo $\varpi^n$ with the Selmer group argument given in $\S$\ref{sec:completion_proof}.
\end{rem}

To sum up, we obtain the following theorem with the first step of induction (Proposition \ref{prop:first_step}).
\begin{thm} \label{thm:character-by-character}
Assume the following conditions:
\begin{enumerate}
\item $(\overline{\rho}, \Delta)$ satisfies Condition CR;
\item $f$ satisfies Condition PO;
\item $f$ is $N^+$-minimal;
\item $\overline{\rho}$ has big image.
\end{enumerate}
For any $n$-admissible form $\mathcal{D}_n = (\Delta, g_n)$ and any homomorphism $\varphi : \mathcal{O}[\Gamma_m] \to \mathcal{O}_\varphi$ where $\mathcal{O}_\varphi$ is a discrete valuation ring of characteristic zero,
we have
$$\varphi ( L_p(K(m'p^\infty),\mathcal{D}_n)  ) \in \mathrm{Fitt}_{\mathcal{O}_{\varphi,n}}( \mathrm{Sel}_{\Delta}(K(m'p^\infty), A_{f,n})^\vee \otimes_\varphi \mathcal{O}_\varphi ) .$$
\end{thm}
\section{Completion of proof: lifting to characteristic zero} \label{sec:completion_proof}
In this section, we complete the proof of the main theorem (Theorem \ref{thm:main_theorem}) with help of Proposition \ref{prop:generalized_divisibility}. Of course, all the conditions in the main theorems are assumed in this section.
We deal with the case $\mathcal{D}^f_n = (\Delta, g_n) = (N^-, f_{\alpha,n})$.

First, we observe that the equality of ideals
 $$\mathrm{Fitt}_{\mathcal{O}_{\varphi,n}}(\mathrm{Sel}_{N^-}(K(m'p^\infty), A_{f,n})^\vee \otimes_{\varphi} \mathcal{O}_\varphi) = \mathrm{Fitt}_{\mathcal{O}_\varphi}(\mathrm{Sel}_{N^-}(K(m'p^\infty), A_{f,n})^\vee \otimes_{\varphi} \mathcal{O}_\varphi) \Mod{\varpi^n_\varphi}$$
holds in $\mathcal{O}_{\varphi,n}$ due to Lemma \ref{lem:fitting-ideals-base-change}.(2).

Consider the following diagram
\[
\xymatrix{
\left( \varphi (L_p( K(m'p^\infty) , f) ) \right) \ar@{-->}[rr]^-{\subseteq\textbf{?}} \ar[d]^-{\Mod{\varpi^{n}_\varphi}}_-{\textrm{Proposition \ref{prop:n-admissible-theta}}} & & \mathrm{Fitt}_{\mathcal{O}_{\varphi}}( \mathrm{Sel}_{N^-}(K(m'p^\infty), A_{f,n})^\vee \otimes_\varphi \mathcal{O}_\varphi ) \ar[r]^-{\subseteq} \ar[d]^-{\Mod{\varpi^{n}_\varphi}} & \mathcal{O}_{\varphi} \ar[d]^-{\Mod{\varpi^{n}_\varphi}}\\
\left( \varphi (L_p( K(m'p^\infty) , \mathcal{D}^f_{n})  ) \right) \ar[rr]^-{\subseteq}_-{\textrm{Theorem \ref{thm:character-by-character}}} & & \mathrm{Fitt}_{\mathcal{O}_{\varphi,n}}( \mathrm{Sel}_{N^-}(K(m'p^\infty), A_{f,n})^\vee \otimes_\varphi \mathcal{O}_\varphi ) \ar[r]^-{\subseteq} & \mathcal{O}_{\varphi,n}  .
}
\]
Since $\mathcal{O}_\varphi$ is a discrete valuation ring, all the ideals of $\mathcal{O}_\varphi$ are filtered. Thus, the inclusion in characteristic zero follows.
\begin{rem}
If we work with the $p$-isolated condition and the rigid pairs as in \cite{bertolini-darmon-imc-2005}, then we bypass this argument since every $n$-admissible form under the $p$-isolated condition uniquely lifts to a characteristic zero form. See
\cite[Proposition 3.12 and Theorem 9.3]{bertolini-darmon-imc-2005} for detail.
\end{rem}
Thus, we have
$$ \varphi (L_p( K(m'p^\infty) , f) ) \in \mathrm{Fitt}_{\mathcal{O}_{\varphi}}( \mathrm{Sel}_{N^-}(K(m'p^\infty), A_{f,n})^\vee \otimes_\varphi \mathcal{O}_\varphi )$$
for all $n \geq 1$.
By Lemma \ref{lem:fitting_lim}, we know
$$\mathrm{Fitt}_{\mathcal{O}_{\varphi}}( \mathrm{Sel}_{N^-}(K(m), A_{f})^\vee \otimes_\varphi \mathcal{O}_\varphi ) = \bigcap^{\infty}_{n=1}\mathrm{Fitt}_{\mathcal{O}_{\varphi}}( \mathrm{Sel}_{N^-}(K(m), A_{f,n})^\vee \otimes_\varphi \mathcal{O}_\varphi ).$$
Thus, we have inclusion
$\varphi (L_p( K(m) , f) ) \in \mathrm{Fitt}_{\mathcal{O}_{\varphi}}( \mathrm{Sel}_{N^-}(K(m), A_{f})^\vee \otimes_\varphi \mathcal{O}_\varphi )$. 
Applying Proposition \ref{prop:generalized_divisibility} to this inclusion, Theorem \ref{thm:main_theorem} follows.

\section{A more global application with anticyclotomic ``partially adelic" $L$-functions} \label{sec:speculation}
In this section, we investigate a more global aspect of the Mazur-Tate type conjecture by gluing the main theorem (Theorem \ref{thm:main_theorem}) varying $p$. The exposition of this section is rather brief.

For the notational convenience, we focus on the case of elliptic curves over $\mathbb{Q}$.

Let $E$ be an elliptic curve over $\mathbb{Q}$ of conductor $N$.
Let $K$ be an imaginary quadratic field with the same assumption with $N$ as before.
Let $m$ be an integer with $(m,N) = 1$.
Let $M = \prod_i p^{r_i}_i$ be an integer where each $p_i >3$ is a prime such that
\begin{enumerate}
\item each $p_i$ is good ordinary for $E$;
\item each residual representation $\overline{\rho}_{E, p_i}$ is surjective;
\item condition PO holds for $(E, K(m'_i)/K, p_i)$ for each $p_i$ where $m'_i$ is the prime-to-$p_i$ part of $m$;
\item $N = N(\overline{\rho}_{E, p_i}) $ for each $p_i$.
\end{enumerate} 
Then, by Theorem \ref{thm:main_theorem}, we know
$$L_p(K(m),E) \Mod{p^r} \in \mathrm{Fitt}_{(\mathbb{Z}/p^r\mathbb{Z})[\Gamma_m]}(\mathrm{Sel}(K(m),E[p^r])^\vee) $$
for any $p^r \mid M$.
Let $f$ be the newform corresponding to $E$.
\begin{assu} \label{assu:normalization}
Suppose that we are able to regard $f$ as a $\mathbb{Z}$-valued function
$$f: B^\times \backslash \widehat{B}^\times  / \widehat{R}^\times \to \mathbb{Z}$$
with integral normalization $f \Mod{p_i} \neq 0$ for each prime factor $p_i$ of $M$.
\end{assu}
Fixing an embedding $\iota_{p_i} : \overline{\mathbb{Q}} \hookrightarrow \overline{\mathbb{Q}}_{p_i}$,
the composition $\iota_{p_i} \circ f$ becomes a $\mathbb{Z}_{p_i}$-valued quaternionic modular form defined in $\S$\ref{subsec:modular_forms}.
\begin{rem}
Due to Remark \ref{rem:normalization_modular_forms}, $f$ may not be able to be integrally normalized when $p_i = 2$ or 3; thus, we exclude 2 and 3. For each prime $p_i$, the $p_i$-integral normalization of quaternionic modular forms is related to the ``$N^-$-new congruence number" defined in \cite[$\S$6.6]{pw-mu}.
 \end{rem}
Applying the same construction of the Bertolini-Darmon element as in $\S$\ref{subsec:theta_elements} to $f$, we can define an analogous element $L_M(K(m),f) \in \mathbb{Z}[\Gamma_m]$.
Then it is not difficult to see
$$L_M(K(m),f) \Mod{p^{r_i}_i} = L_{p_i}(K(m), \iota_{p_i} \circ f ) \Mod{p^{r_i}_i}$$
in $\mathbb{Z}/p^{r_i}_i\mathbb{Z}[\Gamma_{m}]$
for $p^{r_i}_i \mid M$ under Assumption \ref{assu:normalization}.
By Chinese remainder theorem with $E[M]$ and $E(K(m))/M \cdot E(K(m))$, we have
$$\mathrm{Sel}(K(m),E[M])[p^{r_i}_i] \simeq \mathrm{Sel}(K(m),E[p^{r_i}_i])$$
where $M = \prod_i p^{r_i}_i$ as we defined.
With Lemma \ref{lem:control_mod_pn_selmer} and Condition (4) for $M$, we also have
$$\mathrm{Sel}(K(m),E[M])[p^r] \simeq \mathrm{Sel}(K(m),E[p^r])$$
for any $p^r \mid M$.
 
%
\begin{thm}
Under Assumption \ref{assu:normalization}, we have
$$L_M(K(m), E) \in \mathrm{Fitt}_{\mathbb{Z}/M\mathbb{Z}[\Gamma_m]} \left(\mathrm{Sel}(K(m),E[M])^\vee \right)$$
in $\mathbb{Z}/M\mathbb{Z}[\Gamma_m]$ where $(-)^\vee = \mathrm{Hom}(-, \mathbb{Q}/\mathbb{Z})$.
\end{thm}
\begin{proof}
Let $M = \prod_i p^{r_i}_i$.
By Theorem \ref{thm:main_theorem}, we have
$$L_{M}(K(m), E) \Mod{p^{r_i}_i} \in \mathrm{Fitt}_{\mathbb{Z}/p^{r_i}_i\mathbb{Z}[\Gamma_m]} \left(\mathrm{Sel}(K(m),E[p^{r_i}_i])^\vee \right)$$
for each prime divisor $p_i$ of $M$.
By Lemma \ref{lem:fitting-ideals-base-change}.(2), we have
$$\mathrm{Fitt}_{\mathbb{Z}/M\mathbb{Z}[\Gamma_m]} \left(\mathrm{Sel}(K(m),E[M])^\vee \right) \Mod{p^{r_i}_i} = \mathrm{Fitt}_{\mathbb{Z}/p^{r_i}_i\mathbb{Z}[\Gamma_m]} \left(\mathrm{Sel}(K(m),E[p^{r_i}_i])^\vee \right).$$
By Chinese remainder theorem, we have
$$\mathbb{Z}/M\mathbb{Z}[\Gamma_m] \simeq \left(\prod_{i}\mathbb{Z}/p^{r_i}_i\mathbb{Z}\right)[\Gamma_m] .$$
Thus, we obtain the conclusion.
\end{proof}
In the same manner as in Corollary \ref{cor:vanishing_order_zero}, we obtain
\begin{cor} \label{cor:global_consequence}
Suppose that Assumption \ref{assu:normalization} holds.
Let $\chi : \Gamma_m \to \mathbb{Z}^{\times}_\chi$ be a character and assume that $\chi(L_M(K(m), f)) \neq 0$.
Then
$$\sum_{p \vert M}  \left( \mathrm{ord}_{p} \left( \# \vert E(K(m))_\chi \vert \right) + \mathrm{ord}_{p} \left( \# \vert \textrm{{\cyr SH}}(E/K(m))[M]_\chi \vert \right) \right) \leq \sum_{p \vert M} \mathrm{ord}_{p} \left( \chi \left( L_M(K(m), f) \right) \right) $$
where $\textrm{{\cyr SH}}(E/K(m))[M]_\chi$ is the $\chi$-isotypic quotient of the $M$-torsion subgroup of the Shafarevich-Tate group of $E$ over $K(m)$.
\end{cor}
\begin{rem}
Corollary \ref{cor:global_consequence} gives a partial affirmative answer to the question raised in \cite[Remark 1 to Corollary 4, Introduction]{bertolini-darmon-imc-2005} on the size of $\textrm{{\cyr SH}}(E_f/K(m))[M]$ for not necessarily prime power $M$.
\end{rem}
\section*{Acknowledgement}
The author thanks to Robert Pollack for suggesting the problem and correcting many errors in the earlier version. In fact, it was the author's very original thesis problem, but the author could not see how to attack at the moment. More than three years later, the author started to reconsider the problem after discussion with Cristian Popescu at Banach International Center, B\c{e}dlewo, Poland.
The author deeply thanks to Masato Kurihara for generously sharing many of his ideas and his countless encouragement. 
The author learned the argument in $\S$\ref{sec:reduction_to_main_theorem} and $\S$\ref{sec:tame_exceptional_zeros} from him. The author also thanks to Peter Jaehyun Cho, Suh Hyun Choi, Henri Darmon, Ick Sun Eum, Ming-Lun Hsieh, Byungheup Jun, Keunyoung Jeong, Myoungil Kim, Taekyung Kim, Takahiro Kitajima, Kazuto Ota, Hae-Sang Sun, Donggeon Yhee, and Myungjun Yu for helpful discussion and encouragement.
The author thanks to Ulsan National Institute of Science and Technology (UNIST) and Keio University for their generous hospitality during a visit.

The author deeply appreciates the careful and detailed comments of the referees, which correct many mathematical and grammatical inaccuracies in the earlier version.

This research was partially supported by ``Overseas Research Program for Young Scientists" through Korea Institute for Advanced Study (KIAS) and Basic Science Research Program through the National Research Foundation of Korea (NRF-2018R1C1B6007009).

\bibliographystyle{amsalpha}
\bibliography{library}

\providecommand{\bysame}{\leavevmode\hbox to3em{\hrulefill}\thinspace}
\providecommand{\MR}{\relax\ifhmode\unskip\space\fi MR }
\providecommand{\MRhref}[2]{%
  \href{http://www.ams.org/mathscinet-getitem?mr=#1}{#2}
}
\providecommand{\href}[2]{#2}
\begin{thebibliography}{BCDT01}

\bibitem[AW10]{atiyah-wall}
M.~F. Atiyah and C.~T.~C. Wall, \emph{Cohomology of groups}, Algebaraic
  {N}umber {T}heory: Proceedings of an {I}nstructional {C}onference,
  {B}righton, 1965 (J.~W.~S. Cassels and A.~Fr\"{o}hlich, eds.), London
  {M}athematical {S}ociety, 2010, Second {E}dition.

\bibitem[BCDT01]{bcdt}
Christophe Breuil, Brian Conrad, Fred Diamond, and Richard Taylor, \emph{On the
  modularity of elliptic curves over {$\mathbb{Q}$}: wild 3-adic exercises}, J.
  Amer. Math. Soc. \textbf{14} (2001), no.~4, 843--939.

\bibitem[BD94]{bertolini-darmon-derived-heights-1994}
Massimo Bertolini and Henri Darmon, \emph{Derived heights and generalized
  {M}azur-{T}ate regulators}, Duke Math. J. \textbf{76} (1994), no.~1, 75--111.

\bibitem[BD96]{bertolini-darmon-mumford-tate-1996}
\bysame, \emph{Heegner points on {M}umford-{T}ate curves}, Invent. Math.
  \textbf{126} (1996), no.~3, 413--456.

\bibitem[BD97]{bertolini-darmon-rigidanalytic-1997}
\bysame, \emph{A rigid analytic {G}ross-{Z}agier formula and arithmetic
  applications}, Ann. of Math. (2) \textbf{146} (1997), no.~1, 111--147, with
  appendix by Bas Edixhoven.

\bibitem[BD99]{bertolini-darmon-jochnowitz-1999}
\bysame, \emph{Euler systems and {J}ochnowitz congruences}, Amer. J. Math.
  \textbf{121} (1999), 259--281.

\bibitem[BD05]{bertolini-darmon-imc-2005}
\bysame, \emph{Iwasawa's main conjectures for elliptic curves over
  anticyclotomic {$\mathbb{Z}_p$}-extensions}, Ann. of Math. (2) \textbf{162}
  (2005), no.~1, 1--64.

\bibitem[Buz07]{buzzard-families}
Kevin Buzzard, \emph{On {$p$}-adic families of automorphic forms}, Modular
  {C}urves and {A}belian {V}arieties (Basel) (John Cremona, Joan-{C}arles
  Lario, Jordi Quer, and Kenneth Ribet, eds.), Progr. Math., vol. 224,
  Birkh\"{a}user, 2007, pp.~23--44.

\bibitem[CH15]{chida-hsieh-main-conj}
Masataka Chida and Ming-Lun Hsieh, \emph{On the anticyclotomic {I}wasawa main
  conjecture for modular forms}, Compos. Math. \textbf{151} (2015), no.~5,
  863--897.

\bibitem[CH16]{chida-hsieh-p-adic-L-functions}
\bysame, \emph{Special values of anticyclotomic {$L$}-functions for modular
  forms}, J. Reine Angew. Math. (2016), 45 pages, Published Online: 01/09/2016.

\bibitem[Dar92]{darmon-refined-bsd}
Henri Darmon, \emph{A refined conjecture of {M}azur-{T}ate type for {H}eegner
  points}, Invent. Math. \textbf{110} (1992), 123--146.

\bibitem[DI08]{darmon-iovita}
Henri Darmon and Adrian Iovita, \emph{The anticyclotomic main conjecture for
  elliptic curves at supersingular primes}, J. Inst. Math. Jussieu \textbf{7}
  (2008), no.~2, 291--325.

\bibitem[DT94]{diamond-taylor-non-optimal}
Fred Diamond and Richard Taylor, \emph{Non-optimal levels of mod {$\ell$}
  modular representations}, Invent. Math. \textbf{115} (1994), 435--462.

\bibitem[EPW]{epw2}
Matthew Emerton, Robert Pollack, and Tom Weston, \emph{Explicit reciprocity
  laws and {I}wasawa theory for modular forms}, in preparation.

\bibitem[EPW06]{epw}
\bysame, \emph{Variation of {I}wasawa invariants in {H}ida families}, Invent.
  Math. \textbf{163} (2006), no.~3, 523--580.

\bibitem[Fla90]{flach-cassels-tate}
Matthias Flach, \emph{A generalisation of the {C}assels-{T}ate pairing}, J.
  Reine Angew. Math. \textbf{412} (1990), 113--127.

\bibitem[Gom14]{gomez-thesis}
Cl{\'{e}}ment Gomez, \emph{On {M}azur-{T}ate type conjectures for quadratic
  imaginary fields and elliptic curves}, Ph.D. thesis, McGill, March 2014,
  under the supervision of {H}enri {D}armon.

\bibitem[Gre99]{greenberg-lnm}
Ralph Greenberg, \emph{Iwasawa theory for elliptic curves}, Arithmetic theory
  of elliptic curves ({C}etraro, 1997) (Berlin) (C.~Viola, ed.), Lecture Notes
  in Math., vol. 1716, Centro Internazionale Matematico Estivo (C.I.M.E.),
  Florence, Springer-Verlag, 1999, Lectures from the 3rd {C.I.M.E.}~{S}ession
  held in {C}etraro, {J}uly 12-–19, 1997, pp.~51–--144.

\bibitem[Hel07]{helm-israel}
David Helm, \emph{On maps between modular {J}acobians and {J}acobians of
  {S}himura curves}, Israel J. Math. \textbf{160} (2007), 61--117.

\bibitem[HM00]{hachimori-matsuno}
Yoshitaka Hachimori and Kazuo Matsuno, \emph{On finite {$\Lambda$}-submodules
  of {S}elmer groups of elliptic curves}, Proc. Amer. Math. Soc. \textbf{128}
  (2000), no.~9, 2539--–2541.

\bibitem[Hun17]{hung-nonvanishing}
Pin-Chi Hung, \emph{On the non-vanishing mod {$\ell$} of central {$L$}-values
  with anticyclotomic twists for {H}ilbert modular forms}, J. Number Theory
  \textbf{173} (2017), 170--209.

\bibitem[Iha99]{ihara-shimura-curves}
Yasutaka Ihara, \emph{Shimura curves over finite fields and their rational
  points}, Applications of curves over finite fields (Michael~D. Fried, ed.),
  Contemp. Math., vol. 245, American Mathematical Society, 1999, pp.~15--23.

\bibitem[IP06]{iovita-pollack}
Adrian Iovita and Robert Pollack, \emph{Iwasawa theory of elliptic curves at
  supersingular primes over {$\mathbb{Z}_p$}-extensions of number fields}, J.
  Reine Angew. Math. \textbf{598} (2006), 71--103.

\bibitem[Kim17]{kim-summary}
Chan-Ho Kim, \emph{Anticyclotomic {I}wasawa invariants and congruences of
  modular forms}, Asian J. Math. \textbf{21} (2017), no.~3, 499--530.

\bibitem[KK]{kim-kurihara}
Chan-Ho Kim and Masato Kurihara, \emph{On the refined conjectures on {F}itting
  ideals of {S}elmer groups of elliptic curves with supersingular reduction},
  submitted, \href{https://arxiv.org/abs/1804.00418}{arXiv:1804.00418}.

\bibitem[Kob03]{kobayashi-thesis}
Shinichi Kobayashi, \emph{Iwasawa theory for elliptic curves at supersingular
  primes}, Invent. Math. \textbf{152} (2003), no.~1, 1--36.

\bibitem[KPW17]{kim-pollack-weston}
Chan-Ho Kim, Robert Pollack, and Tom Weston, \emph{On the freeness of
  anticyclotomic {S}elmer groups of modular forms}, Int. J. Number Theory
  \textbf{13} (2017), no.~6, 1443--1455.

\bibitem[Kur03]{kurihara-fitting}
Masato Kurihara, \emph{Iwasawa theory and {F}itting ideals}, J. Reine Angew.
  Math. \textbf{561} (2003), 39--86.

\bibitem[Kur14]{kurihara-iwasawa-2012}
\bysame, \emph{The structure of {S}elmer groups of elliptic curves and modular
  symbols}, Iwasawa Theory 2012: State of the Art and Recent Advances (Thanasis
  Bouganis and Otmar Venjakob, eds.), Contributions in Mathematical and
  Computational Sciences, vol.~7, Springer, 2014, pp.~317--356.

\bibitem[Lee15]{joongul_character_values}
Joongul Lee, \emph{Characterization of a cyclic group ring in terms of
  character values}, Commun. Korean Math. Soc. \textbf{30} (2015), no.~1, 1--5.

\bibitem[Lin93]{ling-shimura-subgroups}
San Ling, \emph{Shimura subgroups of {J}acobians of {S}himura curves}, Proc.
  Amer. Math. Soc. \textbf{118} (1993), no.~2, 385--390.

\bibitem[Lon12]{longo-hilbert-modular-case}
Matteo Longo, \emph{Anticyclotomic {I}wasawa's main conjecture for {H}ilbert
  modular forms}, Comment. Math. Helv. \textbf{87} (2012), 303--353.

\bibitem[Lun16]{lundell-level-lowering}
Benjamin Lundell, \emph{Quantitative level lowering}, Amer. J. Math.
  \textbf{138} (2016), no.~2, 419--448.

\bibitem[LV10]{longo-vigni-vanishing}
Matteo Longo and Stefano Vigni, \emph{On the vanishing of {S}elmer groups for
  elliptic curves over ring class fields}, J. Number Theory \textbf{130}
  (2010), 128–--163.

\bibitem[LV17]{longo-vigni-refined}
\bysame, \emph{A refined {B}eilinson-{B}loch conjecture for motives of modular
  forms}, Trans. Amer. Math. Soc. \textbf{369} (2017), no.~10, 7301--7342.

\bibitem[MT87]{mazur-tate}
Barry Mazur and John Tate, \emph{Refined conjectures of the ``{B}irch and
  {S}winnerton-{D}yer type"}, Duke Math. J. \textbf{54} (1987), no.~2,
  711--750.

\bibitem[MW84]{mazur-wiles-main-conj}
Barry Mazur and Andrew Wiles, \emph{Class fields of abelian extensions of
  {$\mathbb{Q}$}}, Invent. Math. \textbf{76} (1984), no.~2, 179–--330.

\bibitem[Nek12]{nekovar-hilbert}
Jan Nekov\'{a}\v{r}, \emph{Level raising and anticyclotomic {S}elmer groups for
  {H}ilbert modular forms of weight two}, Canad. J. Math. \textbf{64} (2012),
  no.~3, 588--668.

\bibitem[Nuc10]{nuccio-fitting}
Filippo A.~E. Nuccio, \emph{Fitting ideals}, The Iwasawa theory of totally real
  fields (J.~Coates, C.~S. Dalawat, A.~Saikia, and R.~Sujatha, eds.), Ramanujan
  Math. Soc. Lect. Notes Ser., vol.~12, International {P}ress, 2010,
  pp.~83--95.

\bibitem[Ota18]{ota-thesis}
Kazuto Ota, \emph{Kato's {E}uler system and the {M}azur-{T}ate refined
  conjecture of {BSD} type}, Amer. J. Math. \textbf{140} (2018), no.~2,
  495--542.

\bibitem[Pol03]{pollack-thesis}
Robert Pollack, \emph{On the {$p$}-adic {$L$}-function of a modular form at a
  supersingular prime}, Duke Math. J. \textbf{118} (2003), no.~3, 523--558.

\bibitem[Pol05]{pollack-algebraic}
\bysame, \emph{An algebraic version of a theorem of {K}urihara}, J. Number
  Theory \textbf{110/1} (2005), 164--177, (special issue in honor of {A}rnold
  {R}oss).

\bibitem[PW11]{pw-mu}
Robert Pollack and Tom Weston, \emph{On anticyclotomic {$\mu$}-invariants of
  modular forms}, Compos. Math. \textbf{147} (2011), 1353--1381.

\bibitem[Rib75]{ribet-img-1}
Kenneth Ribet, \emph{On {$\ell$}-adic represenations attached to modular
  forms}, Invent. Math. \textbf{28} (1975), 245--275.

\bibitem[Rib97]{ribet-pacific}
\bysame, \emph{Images of semistable {G}alois representations}, Pacific J. Math.
  \textbf{181} (1997), no.~3, 277--297, Special Issue: Olga Taussky-Todd in
  memoriam.

\bibitem[Ser95]{serre-local-fields}
Jean-Pierre Serre, \emph{Local {F}ields}, Grad. Texts in Math., vol.~67,
  Springer-{V}erlag, 1995.

\bibitem[Ser03]{serre-trees}
\bysame, \emph{Trees}, corrected second printing ed., Springer Monogr. Math.,
  Springer-{V}erlag, 2003, translated by {J}ohn {S}tillwell.

\bibitem[SU14]{skinner-urban}
Christopher Skinner and Eric Urban, \emph{The {I}wasawa main conjectures for
  {$\mathrm{GL}_2$}}, Invent. Math. \textbf{195} (2014), no.~1, 1--277.

\bibitem[Vat02]{vatsal-nonvanishing}
Vinayak Vatsal, \emph{Uniform distribution of {H}eegner points}, Invent. Math.
  \textbf{148} (2002), no.~1, 1--48.

\bibitem[vO13]{vanorder-dihedral}
Jeanine van Order, \emph{On the dihedral main conjectures of {I}wasawa theory
  for {H}ilbert modular eigenforms}, Canad. J. Math. \textbf{65} (2013),
  403--466.

\bibitem[Wan15]{haining-thesis}
Haining Wang, \emph{Anticyclotomic {I}wasawa theory for {H}ilbert modular
  forms}, Ph.D. thesis, The {P}ennsylvania {S}tate {U}niversity, December 2015,
  under the supervision of Ming-Lun Hsieh and Winnie Wen-Ching Li.

\bibitem[Wil95]{wiles}
Andrew Wiles, \emph{Modular elliptic curves and {F}ermat's last theorem}, Ann.
  of Math. (2) \textbf{141} (1995), 443--551.

\end{thebibliography}

\end{document}